\documentclass[reqno]{amsart}

\usepackage{graphicx,pdfsync,amssymb, latexsym,amsfonts,amsbsy,color,esint,mathtools,enumerate,nicefrac,tikz}
\usepackage{hyperref}
\hypersetup{
colorlinks=true,
linkcolor=blue,
filecolor=blue,
urlcolor=blue,
citecolor=blue,
}

\usepackage[utf8]{inputenc}% for special characters
\usepackage[T1]{fontenc}% for special characters

\DeclareMathOperator{\Id }{Id}

\DeclareMathOperator{\D}{div}

\newtheorem{theorem}{Theorem}[section]
\newtheorem{lemma}[theorem]{Lemma}
\newtheorem{proposition}[theorem]{Proposition}
\newtheorem{definition}[theorem]{Definition}
\newtheorem{corollary}[theorem]{Corollary}
\newtheorem{remark}[theorem]{Remark}

\def \TT  {\mathbb{T}} %circle
\def \RR {\mathbb{R}}  %real numbers
\def \NN {\mathbb{N}}  %natural numbers
\def \ZZ {\mathbb{Z}}  %integer numbers

\def \p {\partial}
\def \g {\gamma}
\def \k {\kappa}

\def \ep {\epsilon}
\def \om {\omega}

\def \T {\mathbf{T}}

\def \N {\mathbf{N}}

\def \H {\mathcal{H}}
\numberwithin{equation}{section}

\begin{document}

\title{illposedness of $C^{2}$ vortex patches}

\author{Alexander Kiselev}

\address{Department of Mathematics, Duke University, Durham, NC 27708, USA.}
\email{kiselev@math.duke.edu}

\author{Xiaoyutao Luo}

\address{Department of Mathematics, Duke University, Durham, NC 27708, USA.}

\email{xiaoyutao.luo@duke.edu}

%    General info
%\subjclass[2020]{}

\date{\today}

 \begin{abstract}
It is well known that vortex patches are wellposed in  $C^{1,\alpha}$  if $0<\alpha <1$. In this paper, we prove the illposedness of $C^{2}$ vortex patches. The setup is to consider the vortex patches in Sobolev spaces $W^{2,p}$ where the curvature of the boundary is $L^p$ integrable.  In this setting, we show the persistence of $W^{2,p}$ regularity when $1<p <\infty$ and construct $C^{2}$ initial patch data for which the curvature of the patch boundary becomes unbounded immediately for $t>0$,
though it regains $C^2$ regularity precisely at all integer times without being time periodic.
  %- though the patch does turn back $C^2$ at all integer times.
The key ingredient is the evolution equation for the curvature, the dominant term in which turns out to be linear and dispersive.
 \end{abstract}

\date{\today}

\maketitle
%%%%%%%%%%%%%%%%%%%%%%%%%%%%%%%%%%%%%%%%%%%%%%%%%%%%%%%%%%%%%%%%%%%%%%%%%%%%%
\section{Introduction}
%%%%%%%%%%%%%%%%%%%%%%%%%%%%%%%%%%%%%%%%%%%%%%%%%%%%%%%%%%%%%%%%%%%%%%%%%%%%%

\subsection{Vortex patches}

Vortex patches are an important family of weak solutions to the 2D Euler equations. We recall that the 2D Euler equation in the vorticity form is given by
\begin{equation}\label{eq:euler_vorticity}
\p_t \omega + (v \cdot \nabla ) \omega =0.
\end{equation}
At each time $t$, the velocity field $u$ is determined by the Biot-Savart law
\begin{equation}\label{eq:euler_biot_savart}
v(x, t) =  K \star \om (x,t) : =  \frac{1}{2\pi }\int_{\RR^2} \frac{(x-y)^\perp }{|x - y|^2} \omega(y, t) \, dy,
\end{equation}
where $x^\perp = (-x_2, x_1)$ for any $x\in \RR^2$.

A vortex patch is a solution to \eqref{eq:euler_vorticity} of the form
\begin{equation}\label{eq:euler_vortex_patch}
\omega(t,x) = \chi_{\Omega(t)},
\end{equation}
where $\chi_{\Omega(t)}$ denotes the characteristic function of a connected bounded domain %(simply-connected and open)
$\Omega(t) \subset \RR^2$ that evolves according to \eqref{eq:euler_biot_savart}. Note that in the literature, the term vortex patch solution often refers to a solution of the form $ \omega = \sum_{1\leq i\leq N} \theta_i(x, t) \chi_{\Omega_i(t)}$ with $ \Omega_i$ being mutually disjoint bounded domains and $\theta_i(x,t)$ being the profiles of vorticity. In this paper, we consider only the case of a single patch in $\RR^2$ with a constant vorticity $2\pi$.

Given an initial patch data $\omega_0(x) = \chi_{\Omega_0}(x)$, there exists the unique patch solution $\omega  = \chi_{\Omega } $ thanks to the Yudovich theory \cite{MR0163529} of $L^1 \cap L^\infty$ weak solutions. The Yudovich theory only implies that $\Omega(t)$ remains a bounded domain whose area is constant in time, but does not address the regularity of $\Omega(t)$, which in this case refers to the smoothness of the patch boundary $\p\Omega$. The question of whether the smoothness of the patch boundary breaks down in finite time was a subject of debate \cite{doi:10.1063/1.857353,MR1050012} in numerical analysis. However, this controversy was settled by Chemin \cite{MR1235440} in 1993 who proved that the patch boundary remains smooth for all times if it is smooth initially.

A key step towards understanding global regularity for smooth vortex patches is the global wellposedness of patches in $C^{1,\alpha}$ for $0< \alpha < 1$ (see e.g. \cite{MR1867882}). The restriction $0<\alpha <1$ and the recent illposedness result of Bourgain-Li \cite{MR3320889} for the 2D and 3D Euler equations in the smooth setting for $\om_0 \in C^{k}$ with integer $k\geq 1$ suggest that the patch problem may also be ill-posed in $C^{k}$ for integer $k\geq 1$. The main purpose of this paper is twofold.
First, we confirm this conjecture by showing that the vortex patch problem is indeed illposed in $C^2$ or $C^{1,1}$.
Secondly, our approach to this question involves an analysis of the evolution equation for the curvature of the patch.
Essentially, the equations for intrinsic geometric parameters of the patch, the curvature and arc-length, can be viewed as an alternative formulation
for the patch evolution problem. Such type of reformulation using intrinsic quantities has been used before in the context of fluids mechanics in a variety of models~\cite{Mullins56,HLSjcp94,HKISjcp98}, and we refer interested readers to references therein.
%a reformulation of the patch evolution in terms of intrinsic geometric properties of the patch: curvature and arc-length.
In our case, the resulting equations do not contain information on patch orientation or position
but otherwise recapture the patch precisely (and the former details can be recovered by solving simple ODEs if needed).
Interestingly, using this reformulation allows us to see the illposedness in $C^2$ on the conceptual level quite directly (notwithstanding the technical estimates one needs to carry out). We believe that this approach can be useful in further analysis of finer features of patch dynamics. %, and in particular singular scenarios observed numerically in \cite{MR2141918,SD14,MR3903908}.

\subsection{Historical development}
As we mentioned above, the study of vortex patch dynamics reduces to the analysis of the patch boundary evolution. If the patch boundary is at least piece-wise $C^1$, then one can derive a 1D equation for the parametrization of the boundary. This equation, known as the contour dynamics equation (CDE), first appeared in \cite{MR524163} and \cite{MR861488}:
\begin{equation}\tag{CDE}\label{eq:CDE}
    \p_t \gamma (\xi,t) = \int_{\TT} \dot{\g}(\eta,t )\ln |\g (\xi,t) - \g(\eta,t)|  \, d\eta,
\end{equation}
where $\gamma :\TT\times [0,T]  \to \RR^2$ is a parametrization of the patch boundary at each time $t \in  [0,T] $.
In this paper, $\TT := \RR / 2 \pi\ZZ$ denotes the $1$-dimensional torus which we identify with an interval of length $2\pi$ with periodic boundary condition.
The local wellposedness of \eqref{eq:CDE} in the framework of $C^{k,\alpha}$ patches was proved by Bertozzi in \cite{MR2686124} (see also \cite{MR1867882}).

After the work of Chemin, several different proofs of global regularity for Euler patches appeared:  by Serfati \cite{MR1270072} and by Bertozzi and Constantin \cite{MR1207667} (see also more recent \cite{verdera2021regularity} for patches of other active scalar equations).
We should emphasize that these works consider the patch regularity problem in the $C^{k,\alpha}$ setting for $k\in \NN$ and $0<\alpha <1$, and more importantly, the 2D dynamics of the patches are used in an essential way. In other words, we are not aware of
a proof of global regularity for vortex patch using only the contour dynamics \eqref{eq:CDE}.

%On the other side, the investigation of more singular patches has also been studied in the community, i.e. patches boundary with angles or cusps. In this direction, Danchin \cite{MR1452164,MR1809342} proved the propagation of smoothness of the patch boundary away from the isolated singularities. It was %shown in~\cite{MR2515784} that corners can cusp instantaneously in the half-plane path problem. Singular patches with rotational symmetry were studied systematically in \cite{1903.00833,1909.13555} by the first author and Jeong. We refer readers to %~\cite{MR795112,MR968694,MR1373741,MR1397348,MR1730570,MR1998937,MR3054601,MR3294414,MR3941226} for other interesting results on the Euler patches in different settings.
Our interest in the ill-posedness patch problem was partly motivated by the recent significant developments on the Cauchy problem of the Euler equations (both in $2D$ and $3D$ in the non-patch setting) with initial data in integer H\"older spaces $C^k$ or in critical Sobolev spaces~
\cite{MR3359050,
MR3320889,
MR4065655,
MR3451386}.
%The criticality can be understood as the velocity $u = K \star \om $ just failing to gain an entire derivative of regularity in the endpoint spaces,
%e.g. to be Lipschitz for vorticity $ \om \in L^\infty(\RR^2).$ %and in $3D$ as the standard energy methods fail to close the argument of wellposedness.
%For example, smooth solutions satisfy the energy estimates
%\begin{equation}\label{eq:classical_energy_estimate}
%\frac{d |\om|_{X}}{dt} \lesssim_X | \nabla u |_{L^\infty}|\om|_{X}
%\end{equation}
%for $X = W^{s,p}$ with $s> \frac{d}{p}$ or $X= C^{k,\alpha}$ for $0 < \alpha <1$.
In particular, it has been shown by Bourgain and Li \cite{MR3359050,MR3320889} that the Euler equations (in $2D$ and $3D$) are illposed for vorticity in critical Sobolev spaces $\om_0 \in W^{\frac{d}{p},p}$ for $1\leq p< \infty$ and integer H\"older spaces $\om_0 \in C^{k}$ with $ k \geq 0$.
%In the patch setting, however, such questions remained open.
%Since the velocity field $u$ enjoys a Lipschitz bound in terms of the $C^{1,\alpha}$ norm of the patch boundary when $\alpha>0$, one can classify the space $C^{1,0}$  (or more generally $W^{\frac{2}{p},p}$ for $1\leq p \leq \infty$) as critical for the patch problem.
In  contrast to the illposeness results in the smooth setting, the following question has remained open for patches:
\begin{center}
    \emph{Is the vortex patch problem illposed
    %for initial patch
    in $C^k$ or $C^{k-1,1}$ with integer $k\geq 1$?}
\end{center}
Even whether Sobolev regularity of the patch boundary, say $W^{k,p}$ for $k \geq 2$, persists globally or locally in time was not known.
In this paper, we show global well-posedness in $W^{2,p}$ with $1 < p<\infty,$ and prove that the vortex patch problem is ill-posed in $C^2$
%We are able to provide an affirmative answer to this question in the case of $C^2$
(and also $C^{1,1}$).

Another motivation for our work are beautiful simulations by Scott and Dritschel \cite{SD1,SD2} on singularity formation for SQG patches.
Rigorously, singularity formation for $\alpha$-SQG patches has only been proved for small $\alpha$ in the half-plane \cite{MR3549626,MR4235799}.
%In forthcoming work, we plan to prove ill-posedness results for $\alpha$-SQG patches using ideas related to this paper.
We think that curvature/arc-length equations can be useful for further analysis of singular scenarios in the $\alpha-$SQG patch without
boundary setting.

\subsection{Main results}

We now present the main results of this paper. To streamline the presentation, we introduce the definition of $C^{k,\alpha}$ (and  $W^{k,p}$) domains and refer to the domain $\Omega(t)$ as a patch solution.

\begin{definition}
Let $\Omega \subset \RR^2 $ be a simply-connected bounded domain. We say $ \Omega$ is  $C^{k,\alpha}$ (respectively $W^{k,p}$) if the arc-length parametrization $\gamma :\frac{L}{2\pi} \TT \to \RR^2 $ satisfies $\gamma \in C^{k,\alpha}( \frac{L}{2\pi} \TT) $ (resp. $W^{k,p}( \frac{L}{2\pi} \TT) )$ where $L$ is the length of $ \gamma  $
and $\frac{L}{2\pi} \TT = \RR / L \ZZ$.

We say $ \Omega (t) $ is a $C^{k,\alpha}$ (respectively $W^{k,p}$) patch solution (or vortex patch) on a time interval $I\subset \RR$ if
\begin{itemize}
    \item The characteristic function $\chi_{\Omega(t)}$ is a solution of \eqref{eq:euler_vorticity}--\eqref{eq:euler_vortex_patch};
    \item The domain $\Omega(t) $ is $C^{k,\alpha}$ (respectively $W^{k, p}$) for all $t \in I$.
\end{itemize}
\end{definition}

Our first main result is the wellposedness of $W^{2,p}$ patch solutions, which is needed to show the $C^2$ illposedness.
\begin{theorem}\label{thm:main_wellposedness}
Let $1<p <\infty$ and $\Omega_0 \subset \RR^2 $ be a $W^{2,p}$ domain. Then the unique patch solution $ \Omega(t) $ with initial data $ \Omega_0 $  is a $W^{2, p}$ patch solution for all $t>0$.
\end{theorem}

\begin{remark}\label{remark:wellposedness}
\hfill
\begin{enumerate}
\item This $W^{2, p}$ regularity result does not follow from simple modification of the arguments in \cite{MR1235440,MR1207667} for the $C^{1,\alpha}$ patches.
%Indeed, in these works the $C^{1,\alpha}$ regularity of the patch boundary was lifted off to certain $C^{1,\alpha}$ regularity on the plane. In the Sobolev setting, such a lift-off would result in a loss of derivative/integrability.

 \item We use the contour equation \eqref{eq:CDE} to establish the local wellposedness of $W^{2,p}$ patches. The global wellposedness of $W^{2,p}$ patches follows from a continuation criterion based on known global regularity for $C^{1,\alpha}$ patches.

%These two quantities uniquely determine a patch up to translation and rotation.

 %   \item
    %We are currently not able to prove directly the global wellposedness using only the curvature and arc-length equations.

    \item The result generalizes to higher order Sobolev spaces $W^{k,p}$ for all $k\geq 2$ and $1<p<\infty$ in a straightforward (though highly computational) manner.

\end{enumerate}
\end{remark}

In contrast to the wellposedness of $W^{2,p}$ patch solutions when $1<p<\infty$, we show that the patch problem is ill-posed in $C^2$. % which is also the first illposedness result for the vortex patches in the smooth settings.

\begin{theorem}\label{thm:main_illposedness}
There exist $C^{2}$ domains  $\Omega_0 \subset \RR^2 $ and $T>0$ such that the unique patch solution $ \Omega(t) $ with initial data $ \Omega_0 $ is not $W^{2, \infty}$ for any $t\in (0,T]$.

\end{theorem}
\begin{remark}\label{remark:illposedness}
A few remarks concerning Theorem \ref{thm:main_illposedness}.
\begin{enumerate}

\item It follows that the patch problem is illposed in $C^{2} $ and $C^{1,1}$ thanks to the equvalence $C^{1,1}(\TT) = W^{2,\infty}(\TT)$.

\item The illposedness mechanism is based on a certain dispersion effect in the evolution of the curvature which is purely linear. We discuss this in detail in the coming subsection.

\item  In fact, one can show that the constructed patch solution $ \Omega(t) $ in our ill-posedness example is $C^2$ if and only if $t\in \ZZ$, and for other times $t\in \RR \setminus \ZZ$, $ \Omega(t) $ is only $W^{2, p} $ for all $p<\infty$. See Theorem \ref{thm:final_illposedness} for details.

\item Using our setup it should be possible to prove the illposedness of the patch problem in $C^{k}$ for any integer $k\geq 2$.
The case $C^{1}$ seems to be out of reach at the moment.

\end{enumerate}
\end{remark}

\subsection{Outline of the proof}
Our proof for both the wellposedness and illposedness is Lagrangian in its heart and based on the contour equation \eqref{eq:CDE} $\gamma: \TT \times [0,T] \to \RR^2$.

The existence of $W^{2,p}$ vortex patches follows from a standard Banach fixed-point argument. The key step is to show the contraction estimates in this setup. It is known that for the vortex patch the velocity gradient is not continuous in $\RR^2$, specifically across the patch boundary. However, we show that in the Lagrangian variable, the velocity field $v $ is $W^{2,p}$ along the patch boundary. This can be used to establish local well-posedness.
%To show this, we formulate the related integrals for the velocity in terms of the arc-length metric $ g= |\dot{\g}|$, tangent vector $\T$, and (signed) curvature $\k$. Relating to these geometric quantities allows us to reveal the cancellations in the related integrals and show the $W^{2,p}$ contraction estimates.
Once we have the local existence for $W^{2,p}$ vortex patches, the global regularity follows from the well-known $C^{1,\alpha}$ global regularity.

As expected, the $W^{2,p}$ wellposedness argument breaks down at $p=1$ and $p=\infty$, but for different reasons. The case $p =1$ is due to a lack of regularity/control and seems to be out of reach of the techniques in the current paper.
The case $p=\infty$ fails due to the unboundedness of certain singular operator appearing in the second-order derivative of the velocity $v$, indicative of the illposedness of the patch problem in $L^\infty$-based spaces.

To prove the $C^2$ illposedness, we do not use the contour equation \eqref{eq:CDE} directly.
Instead, we track the evolution of the (signed) curvature $\k:\TT \times [0,T] \to \RR$ of the patch boundary.
%In fact,
%the curvature $\k$ and the arc-length metric $g:\TT \times [0,T] \to \RR^+$ under the flow of the velocity vector field
%can be used to give an alternative formulation of the patch evolution problem.
%The arc-length and curvature determine a patch solution up to a translation and rotation. As we will see later, the position and the orientation of the patch can be obtained by solving simple ODEs when $ \k$ and $g$ are given,
%thus fully recovering the parametric curve $\g$, which is the unique solution of the patch boundary by the Yudovich theory.
%The derivation of the evolution of curvature and arc-length is done in Section \ref{sec:illposedness}, and here we write it in the following schematic way
%\begin{equation}\label{eq:intro_g_k}
%\begin{cases}
%\p_t \k = - 2 \k \p_s v \cdot\T +  \p_s^2 v\cdot \N &\\
%\p_t g = \p_s v \cdot\T g&.
%\end{cases}
%\end{equation}
%Even though at the moment, we are not able to prove $W^{2,p}$ wellposedness of vortex patches using directly the system \eqref{eq:intro_g_k}, this system should be useful in further analysis of fine features of patch dynamics.
A careful examination of the curvature equation reveals that it has the following structure:
%Based on the appearance of singular operators in the second order derivative $\p_s^2 v $, we examine more closely the first equation of \eqref{eq:intro_g_k} for the curvature and find that in a suitable regime of regularity
\begin{equation}\label{eq:heuristic_intro}
  \p_t \k  = (-\p_s v \cdot\T   + \pi \mathcal{H}) \k + l.o.t
\end{equation}
where  $\mathcal{H}$ is the periodic Hilbert transform and $\T$ is the tangent vector. The rest of the terms entering the equation are strongly nonlinear but can be shown to have higher (H\"older) regularity.
This type of equations is known to be illposed in $L^\infty$-based spaces,
such as $C^2$ or $W^{2,\infty}$ (see e.g. \cite{MR4065655}). %The main task amounts to proving that the lower order terms in \eqref{eq:heuristic_intro} are controlled uniformly as $p\to \infty$.
%In fact, we are able to show these lower order terms are H\"older continuous when $(g,  \k) \in W^{1,p}  \times L^p$ for $p>\frac{3}{2}$ and the H\"older index is increasing as a function of $p$.

These observations together allow us to pick fairly explicit initial data with curvature $\k_0 \in C(\TT) $ such that the unique solution  $ \k (t) \in L^{p}$ for all $p<\infty$ but $ \k (t) \not \in L^{\infty}$ on $0<t<T$ for some $T>0$. The claim in Remark \ref{remark:illposedness} about the patch being $C^2$ only for $t \in \ZZ$ follows from the explicit formula of the group $e^{ t \pi \mathcal{H}} = \cos(\pi t)\Id + \sin(\pi t) \mathcal{H}$ which follows from $\mathcal{H}^2=-1$ identity.

\subsection{Organization of the paper}
The rest of the paper is as follows.
\begin{enumerate}

\item In Section \ref{sec:existence}, we establish the (local) wellposedness for $W^{2,p}$. The proof is based on $W^{2,p}$ Sobolev estimates of the velocity field along the patch boundary.

\item In Section \ref{sec:globalregularity}, we show      how to obtain the global $W^{2,p}$ regularity  of vortex patches  by using known $C^{1,\alpha}$ regularity results, thus proving Theorem \ref{thm:main_wellposedness}.

\item In Section \ref{sec:curvatureeq}, we introduce some preliminary geometric calculations and then derive the evolution equation for the (signed) curvature of the patch boundary. We also discuss the arc-length/curvature
system as an alternative formulation of patch evolution.

\item In Section \ref{sec:illposedness}, we use the curvature equation to show that the vortex patch problem is ill-posed in $C^{2}$.

\end{enumerate}

\subsection*{Acknowledgment.}
AK acknowledges partial support of the NSF-DMS grant 2006372 and of the Simons Fellowship grant 667842. XL is partially supported by the NSF-DMS grant 1926686.
Tarek Elgindi should have been a rightful co-author of this article, but modestly stepped aside.
We would like to thank him for numerous helpful discussions and essential insights. We also thank Thomas Yizhao Hou for pointing us to a few helpful references.

%%%%%%%%%%%%%%%%%%%%%%%%%%%%%%%%%%%%%%%%%%%%%%%%%%%%%%%%%%%%%%%%%%%%%%%%%%%%%
\section{Existence of \texorpdfstring{$W^{2,p}$}{W2p} vortex patches}\label{sec:existence}
%%%%%%%%%%%%%%%%%%%%%%%%%%%%%%%%%%%%%%%%%%%%%%%%%%%%%%%%%%%%%%%%%%%%%%%%%%%%%
In this section, we prove the existence of $W^{2,p}$ vortex patches using the contour equation \eqref{eq:CDE}. The proof is based on a fixed-point argument in a suitable Sobolev setting. Compared to the $C^{1,\alpha}$ local existence result~\cite{MR1867882}, we emphasize geometric quantities such as tangent vector, arc-length metric, and curvature. Such an emphasis allows us not only to prove the local existence of vortex patches in the Sobolev spaces $W^{2,p}$ but also to reveal the dispersive characteristics in the evolution that we will exploit in Section \ref{sec:illposedness} for the illposedness.

As a general comment, throughout the paper, the time variable $t$ will often be suppressed, and the evolution equations are understood at each fixed time $t$.

%%%%%%%%%%%%%%%%%%%%%%%%%%%%%%%%%%%%%%%%%%%%%%%%%%%%%%%%%
\subsection{Preliminaries}
%%%%%%%%%%%%%%%%%%%%%%%%%%%%%%%%%%%%%%%%%%%%%%%%%%%%%%%%%
The motion of a vortex patch  is given by a parametric curve $\gamma :\TT \times [0,T] \to \RR^2$
\begin{equation}\label{eq:gamma_curve}
\p_t \gamma  (\xi,t )= v(\gamma (\xi,t ),t).
\end{equation}
where the velocity field $v :\RR^2 \times [0,T] \to \RR^2$ is defined by the Biot-Savart law,
\begin{align} \label{eq:velocity}
v(\gamma(\xi,t),t) &   =   \int_{ \TT  }    \dot{\gamma } (\eta)  \ln |\gamma(\xi)-\g( \eta ) |     \, d\eta .
\end{align}
Here the factor $\frac{1}{2\pi}$ is dropped from \eqref{eq:euler_biot_savart} since we consider a single patch with constant vorticity $2\pi$.

Throughout the paper, we also view the parameterization $\g : \TT\times [0,T] \to \RR^2 $ as a time-dependent $2\pi$-periodic function on $\RR$; we denote the arc-length metric by $g= |\dot{\gamma}|$, where the dot on top indicates $\p_\xi = \frac{\p}{\p \xi}$, i.e. the usual differentiation; we use $\p_s =\frac{1}{g} \p_\xi $ to indicate the derivative with respect to the arc-length parameter $s = \ell(\xi) : \xi \mapsto \ell(\xi) = \int_0^\xi g(\eta)\, d\eta$.

We consider the usual counterclockwise orientation of the curve $\g$, and let $\T $ be the unit tangent vector and $\N=-\T^\perp$ be the outer unit normal vector for the curve $\gamma $. More precisely, in the Lagrangian coordinates we write $\T(\xi,t)$ or just $\T(\xi)$ for the tangent vector at a point $\g(\xi)$ and time $t$. The (signed) curvature $\k : \TT\times [0,T] \to \RR$ of $\gamma$ is defined via
\begin{equation}
\begin{cases}\label{Neq1222}
\p_s \T  = - \k  \N, &\\
\p_s \N  =   \k  \T .  &
\end{cases}
\end{equation}

Here and in what follows, we sometimes switch between a given parametrization $\TT \ni \xi \mapsto \gamma(\xi)$ and an arc-length parametrization $ \frac{L }{2\pi}\TT \ni s \mapsto \gamma(s)$ and still use the same notation $\gamma,\k,\T,\N\dots$ for the position $\g$, curvature $\k$, tangent $\T$, normal $\N$ etc, except that we use arguments $s$ and $s'$ as the arc-length parameters instead of $\xi,\eta$ and $\tau$ for Lagrangian labels.

%%%%%%%%%%%%%%%%%%%%%%%%%%%%%%%%%%%%%%%%%%%%%%%%%%%%%%%%%
\subsection{Functional setup}
%%%%%%%%%%%%%%%%%%%%%%%%%%%%%%%%%%%%%%%%%%%%%%%%%%%%%%%%%
In this subsection, we set up functional spaces for the parametric curves. To ensure a $C^1$ function $\g:\TT\to \RR^2$ parameterizes a simple curve (no self-intersection), we need to consider the so-call arc-chord $\Gamma$, defined by
\begin{equation}\label{eq:def_arc_chord}
 \Gamma (\g)=  \sup_{\xi,\eta \in \TT,\, \xi \ne \eta} \frac{ |\xi -\eta | }{|\gamma(\xi)  - \gamma( \eta)| },
\end{equation}
where the choice of distance $|\xi -\eta|$ is inessential for our purposes--- it can be the distance on the torus or the Euclidean one when viewing $\xi,\eta$ as points in an interval of $\RR$.
To show the existence of $W^{2,p}$ patches, let us consider an open subset $X_p$ of the Banach space $W^{2,p}(\TT) $ 
$$
X_p := \{     \gamma :\TT   \to \RR^2 : \g\in W^{2,p}(\TT)  \,\, \text{such that} \, \,   \Gamma < \infty   \},
$$
which includes all proper parametrizations of the boundaries of all $W^{2,p}$ simply-connected bounded domains since $\Gamma < \infty$ implies $g=|\dot{\g}|>0$.

To run the Banach fixed point argument, we need a complete metric space in $X_p$. For any $M > 1$, we consider a closed set $B^M_p \subset X_p $
\begin{equation}\label{eq:B_p_M_definition}
B^M_p:= \{   \gamma  \in X_p:  | \gamma  |_{W^{2,p}(\TT)}  \leq M , \, | g|_* \leq M, \, \Gamma \leq M  \},
\end{equation}
where for $g\in W^{1,p}(\TT)$ the functional $ |\cdot |_*$ is defined by
\begin{equation}\label{eq:g_*_definition}
 | g|_*: = \max\{ |1/g|_{L^\infty(\TT)} ,|g|_{L^\infty(\TT)} \}.
\end{equation}
(in particular, it is not a norm).

For any $M>1$, $B^M_p$ is a subset of simple closed $W^{2,p}$ curves, i.e  with $L^p$ curvature, and $X_p = \bigcup_{M>0} B^M_p.$
Next, we show that $B^M_p$ is a complete metric space with the natural norm $ | \cdot  |_{W^{2,p}(\TT)}$, which is necessary for the Banach fixed point argument later.
\begin{lemma}\label{lemma:completeness}
For any $1< p \leq \infty$  and $M>1$, $B^M_p$ is a complete metric space with metric $d(\g_1,\g_2) = | \g_1 -\g_2|_{W^{2,p}(\TT)}$.
\end{lemma}
\begin{proof}
Let $\g_n  \in B^M_p$ be a Cauchy sequence in the metric $d$. Then there is $\g  \in W^{2,p}(\TT) $ such that $\g_n \to \g$. We need to show that in fact $\g \in  B^M_p $.

Apparently, $|\g   |_{W^{2,p}(\TT)}  \leq M $, so it remains to verify the bounds for $g$ and $\Gamma$. Due to the Sobolev embedding $  W^{1,p}(\TT)\subset  C (\TT)  $ for $p>1$ and the inequality
$$
g_n - |g_n -g|_{L^\infty(\TT)} \leq g \leq g_n + |g_n -g|_{L^\infty(\TT)}
$$
we have that
$$
1/M -\ep< g < M +  \ep \quad \text{for any $\ep>0$},
$$
which concludes that $ |g|_* \leq M$.

For the arc-chord
$\Gamma$, we see that  for any $\eta ,\xi \in \TT$, $\eta\neq \xi$ and any $n\in \NN$,
\begin{align*}
 \frac{ |\xi -\eta | }{|\g(\xi) - \g(\eta ) |}
 &\leq  \frac{ |\xi -\eta |  }{|\g_n(\xi) - \g_n(\eta ) | -   4\pi|\dot{\g}-\dot{\g}_n|_{L^\infty(\TT)} }.
\end{align*}
Since $|\dot{\g}-\dot{\g}_n|_{L^\infty(\TT)} \leq C_p  | {\g}- {\g}_n|_{W^{2,p}(\TT)}  \to 0  $ as $n\to \infty$, for any $\ep>0$, we can choose $n$ sufficiently large depending on $\xi,\eta, M$, and $\ep$ such that
$$
\frac{ |\xi -\eta |  }{|\g(\xi) - \g(\eta ) |}
 \leq   M  + \ep.
$$
Since $\ep>0$ is arbitrary, this implies $\Gamma(\g) \leq M $.
We conclude that $ \g  \in  B^M_p$ .
\end{proof}

%%%%%%%%%%%%%%%%%%%%%%%%%%%%%%%%%%%%%%%%%%%%%%%%%%%%%%%%%
\subsection{Basic estimates of \texorpdfstring{$W^{2, p }$}{W2p} curves}
%%%%%%%%%%%%%%%%%%%%%%%%%%%%%%%%%%%%%%%%%%%%%%%%%%%%%%%%%
In this subsection, we derive suitable Sobolev estimates for the $W^{2, p }$ curves.  These estimates will rely on $L^p$-boundedness of the maximal function, for which we recall the necessary definitions here.

Given any $f \in L^1(\TT)$,  we denote by $\mathcal{M}f:\TT\to \RR$ the maximal function of $f$,
\begin{equation}\label{eq:def_maximal_function}
\mathcal{M}f(\xi) = \sup_{0< \ep < 4\pi }\frac{1}{2\ep}\int_{ \xi-\ep}^{\xi + \ep} |f (\eta )| \, d\eta .
\end{equation}
The restriction of $\ep < 4\pi$ is non-essential and the boundedness of $\mathcal{M}$ on $L^p(\TT)$ for $1<p\leq \infty$ follows from the standard $\RR^d$ results \cite{MR0290095}.

In the remainder of this paper, we denote by $C_M$ a positive constant depending only on $M$ and $p$ that may change from line to line. We also recall the big O notation $X = O(Y)$ for a quantity $X$ such that $|X| \leq C Y$ for some absolute constant $C>0$.

Now we state a few basic estimates that we will use as the building blocks.
\begin{lemma}\label{lemma:building_blocks}
Let  $1< p \leq \infty$, $\alpha = 1- \frac{1}{p}$,  and $\g \in B^M_p$. For any $\xi, \eta\in \TT$, we have
\begin{subequations}
\begin{align}
\T(\xi)  \cdot \T(\eta) &= 1+ O(C_M |\xi -\eta |^{2\alpha }) \label{eq:building_blocks_a}\\
\T(\xi)  -\T(\eta) &=  O(C_M |\xi -\eta |^{ \alpha }) \label{eq:building_blocks_b}\\
\T(\eta) \cdot \N(\xi) & = O(C_M |\xi -\eta |^{\alpha }) \label{eq:building_blocks_c} \\
(\g(\xi) - \g(\eta)) \cdot \N(\xi) &= O(C_M |\xi -\eta |^{1+\alpha })  \label{eq:building_blocks_d} \\
(\T(\xi) - \T(\eta)) \cdot \T(\xi) &= O(C_M |\xi -\eta |^{2\alpha }) \label{eq:building_blocks_e} \\
|\g(\xi) - \g(\eta)|^{-1} &= O(C_M |\xi -\eta|^{-1}) \label{eq:building_blocks_f}
\end{align}
\end{subequations}
and for any $ \zeta \in \TT$ such that $|\eta-\zeta|, |\xi-\zeta| \leq |\xi-\eta|$, the maximal estimates
\begin{subequations}
\begin{align}
\T(\eta) \cdot \N(\xi) & =  O(C_M \mathcal{M}\k ( \zeta ) |\xi -\eta | ) \label{eq:building_blocks_M_a} \\
\T(\eta) \cdot \T(\xi) & = 1+ O(C_M \mathcal{M}\k ( \zeta ) |\xi -\eta |^{1+\alpha } ) \label{eq:building_blocks_M_b} \\
(\g(\xi) - \g(\eta)) \cdot \N(\xi) &= O(C_M \mathcal{M}\k ( \zeta ) |\xi -\eta |^{2})  \label{eq:building_blocks_M_c} \\
  \T(\xi)  \cdot \left[\T(\xi) -\T(\eta)  \right] &= O(C_M \mathcal{M}\k ( \zeta ) |\xi -\eta |^{1+\alpha}) \label{eq:building_blocks_M_d} \\
\left[(\g(\xi) - \g(\eta) \right]\cdot \left[\T(\xi) -\T(\eta)  \right] &= O(C_M \mathcal{M}\k (\zeta ) |\xi -\eta |^{2+\alpha}) \label{eq:building_blocks_M_e}.
\end{align}

\end{subequations}

Finally,
\begin{align}\label{eq:building_blocks_linear}
(\g(\xi) - \g(\eta)) \cdot \T(\xi) &= g(\zeta)(\xi -\eta )+ O(C_M |\xi -\eta |^{1+ \alpha }).
\end{align}
\end{lemma}
\begin{proof}
Let us consider the first set of estimates. \eqref{eq:building_blocks_f} is a consequence of the assumption $\g \in B^M_p$.

Bounds \eqref{eq:building_blocks_b} and \eqref{eq:building_blocks_c} follow from the fundamental theorem of calculus and the assumption $\g \in B^M_p$. For instance,
\[ \T(\eta) \cdot \N(\xi) \leq |\int_{\eta}^{\xi}\k( \tau ) g(\tau) \N(\tau )\cdot   \N(\eta) \, d\tau| \leq |\k|_{L^p(\TT)} |g|_{L^\infty(\TT)}|\xi -\eta|^{\alpha} \]
 and due to $g \in W^{1,p}(\TT)$ with $|g|_* \leq M$ we have $ |\k|_{L^p(\TT)}|g|_{L^\infty(\TT)} \leq C_M$.

The fundamental theorem of calculus with \eqref{eq:building_blocks_c} implies \eqref{eq:building_blocks_a} and \eqref{eq:building_blocks_e}. Indeed, we have
\begin{align*}
(\T(\xi) - \T(\eta)) \cdot \T(\xi) = -\int_{\eta}^{\xi} \k( \tau ) g( \tau )\N(\tau )\cdot   \T(\xi) \, d\tau = O(C_M | \xi - \eta|^{2\alpha} ) .
\end{align*}

At last, \eqref{eq:building_blocks_d} follows from \eqref{eq:building_blocks_c}, and the estimate \eqref{eq:building_blocks_a} together with the H\"older continuity of $g$ also imply the last identity \eqref{eq:building_blocks_linear}.

Now we focus on the second set of estimates \eqref{eq:building_blocks_M_a}--\eqref{eq:building_blocks_M_e} involving maximal function. These estimates rely on the following simple bound:
\begin{equation}\label{eq:building_blocks_aux1}
\int_{\eta}^{\xi}    |\k(\tau)g(\tau)| \, d\tau \leq  2|\xi -\eta | \mathcal{M}  \k(\zeta) |g|_{L^\infty(\TT)}
\end{equation}
for any $\zeta \in\TT $ such that $|\zeta - \xi|,|\zeta-\eta| \leq |\xi-\eta|.$

We demonstrate how to obtain \eqref{eq:building_blocks_M_a} and \eqref{eq:building_blocks_M_c}, as the rest follows mutatis mutandis.
By the fundamental theorem of calculus,
\begin{align*}
 \T(\eta)\cdot  \N(\xi)& = -\int_{\eta}^{\xi} \N(\tau)\cdot  \N(\xi) g(\tau) \k(\tau)\, d\tau  = O\left(C_M \int_{\eta}^{\xi}    |\k(\tau)| \, d\tau \right)
\end{align*}
So  \eqref{eq:building_blocks_M_a} follows from the bound \eqref{eq:building_blocks_aux1}. For \eqref{eq:building_blocks_M_c}, we use the fundamental theorem of calculus twice:
\begin{align*}
(\g(\xi) - \g(\eta)) \cdot \N(\xi) & = \int_{\eta}^{\xi}  \T(\tau)     \cdot \N(\xi) g(\tau) \,d\tau \\
& =\int_{\eta}^{\xi} \int_{\tau}^{\xi}  \N(\tau')     \cdot \N(\xi) \k(\tau') g(\tau') \,d\tau' g(\tau) \,d\tau
\end{align*}
and hence for any $\zeta$ lying between $\xi$ and $\eta$,
\begin{align*}
\Big |(\g(\xi) - \g(\eta)) \cdot \N(\xi)   \Big| \leq C_M  \int_{\eta}^{\xi} \int_{\eta }^{\xi}  |\k(\tau')| \,d\tau'    \,d\tau = O(C_M \mathcal{M}\k (\zeta) |\xi -\eta |^2 ).
\end{align*}
\end{proof}

The next result will be crucial for both the wellposedness and ill-posedness results.
\begin{corollary}\label{corollary:building_blocks}
Let  $1< p \leq \infty$, $\alpha = 1- \frac{1}{p}$,  and $\g \in B^M_p$. For any $\xi,\eta \in \TT$, $\xi \neq \eta$ and any $\zeta \in\TT $ such that $|\zeta - \xi|,|\zeta-\eta| \leq |\xi-\eta|$, there holds
\begin{align*}
 \frac{  (\gamma(\xi)- \gamma(\eta)) \cdot \T(\xi)}{  |\gamma(\xi)- \gamma(\eta)|^{2   }}  = \frac{\xi - \eta}{g(\zeta) |\xi -\eta|^2}   + O(C_M |\xi -\eta |^{\alpha -1}).
\end{align*}

\end{corollary}
\begin{proof}
Taylor expansion shows that for any $M>1$, there exists a small $\ep_M>0$ such that for any $0< |\xi -\eta|\leq \ep_M$ we have
\begin{equation}\label{aux22122}
\frac{  |\xi - \eta |^{2 } }{  |\gamma(\xi) - \gamma(\eta)|^{2 } }  = \frac{1}{g(\xi)^2}  + O( C_M|\xi -\eta|^\alpha  );
\end{equation}
this follows from H\"older regularity of $g$ and $\T$.
Then due to the arc-chord $\Gamma\leq M$, we also have that \eqref{aux22122} holds in the regime $|\xi -\eta|\geq \ep_M$ as well (for a sufficiently large $C_M>0$). The conclusion then follows from the estimate \eqref{eq:building_blocks_linear} in Lemma \ref{lemma:building_blocks} and the assumption $\g \in B^M_p$.
\end{proof}

%%%%%%%%%%%%%%%%%%%%%%%%%%%%%%%%%%%%%%%%%%%%%%%%%%%%%%%%%
\subsection{Differentiablity of the velocity field}
%%%%%%%%%%%%%%%%%%%%%%%%%%%%%%%%%%%%%%%%%%%%%%%%%%%%%%%%%

A key ingredient in the existence proof is the differentiability of the velocity field along the patch boundary. In this subsection, we show that if $\g$ is a simple closed $W^{2,p}(\TT)$ parametric curve, i.e. $\g\in B^M_p$, then the velocity $ v(\g)$ defined according to the Biot-Savart law \eqref{eq:velocity} is also $W^{2,p}(\TT)$.  Note that the argument below proving $L^p$ bounds of $\p_s^2 v $ fails when $p=\infty$, which suggests the $C^2 $ illposedness of the vortex patch.

The main result of this subsection is the following. Note that we specifically use arc-length derivatives as these formulas will be used in Section \ref{sec:illposedness}.

\begin{proposition}\label{prop:derivatives_v}

Let  $1< p < \infty$ and $\g \in B^M_p$ for some $M>1$. Then the velocity $v = v(\g) $ defined according to \eqref{eq:velocity} satisfies $  v  \in W^{2,p}(\TT)$. In particular, $ v $ is differentiable a.e with arc-length derivative
\begin{align}\label{eq:d_v_xi}
\p_s v(\g(\xi)) = P.V. \int_{ \TT   } \T (\eta )\frac{   (\g(\xi  )-\g(\eta)) \cdot \T(\xi ) }{  |\g(\xi)-\g(\eta)|^{2  }} g(\eta)\,d\eta  ,
\end{align}
and  $ \p_s   v $ is a.e. differentiable  with arc-length derivative
\begin{equation}\label{eq:d2_v_xi}
	\begin{aligned}
		 &\p^2_s   v (  \g(\xi)  )  =      -  P.V. \int_{\TT }   \k( \eta )\N ( \eta )    \frac{  (\gamma( \xi )- \gamma(  \eta )) \cdot \T( \xi )}{  |\gamma( \xi )- \gamma(  \eta )|^{2  }}  g(\eta )\, d \eta  \\
		&   + \int_{\gamma } \T( \eta )  \frac{ \big[ (\T( \xi ) - \T(\eta  ))\cdot \T( \xi ) -\k( \xi ) (\gamma(\xi )- \gamma(\eta ))\cdot \N( \xi )  \big]  }{ |\gamma(\xi )- \gamma( \eta )|^{2 } }    g(\eta ) \, d\eta \\
		&  -  2\int_{\gamma } \T(\eta ) \frac{    \big( (  \gamma(\xi)- \gamma( \eta )  ) \cdot \T(\xi) \big) \big( (  \gamma(\xi)- \gamma(\eta )  ) \cdot (\T(\xi) -\T(\eta )) \big)  }{|\gamma(\xi)- \gamma(\eta )|^{4  }}  g(\eta )\, d\eta,
	\end{aligned}
\end{equation}
where the right-hand side of \eqref{eq:d_v_xi} and \eqref{eq:d2_v_xi} are well-defined functions in $W^{1,p}(\TT)$ and $L^p(\TT)$ respectively.

In addition, for any $M>1$ there exists a constant $C_M$ such that
\begin{equation}
| v  |_{W^{2,p}(\TT)}  \leq C_M.
\end{equation}

\end{proposition}
\begin{proof}
We proceed in several steps: first, show the expressions for derivatives are $W^{1,p}$ and $L^p$ functions; next, show that they are derivatives of the velocity in the distributional sense; then, conclude that velocity is differentiable a.e. with derivative equal to those expressions.

\noindent
{\bf Step $1$:  bounds of \eqref{eq:d_v_xi}}

Let us first show that the right-hand side of \eqref{eq:d_v_xi} is a well-defined function in $L^{\infty}(\TT)$. Appealing to Corollary \ref{corollary:building_blocks} and the bound \eqref{eq:building_blocks_b}, we have
\begin{equation}\label{eq:eq:d_v_xi_aux1}
\begin{aligned}
  &   P.V. \int_{ \TT   } \T (\eta )\frac{   (\g(\xi  )-\g(\eta)) \cdot \T(\xi ) }{  |\g(\xi)-\g(\eta)|^{2  }} g(\eta)\,d\eta  \\
& = P.V. \int_{ \TT   } \T (\xi )\frac{    \xi -  \eta }{  g(\eta)| \xi  - \eta |^{2  }} g(\eta)\,d\eta + \int_{ \TT   } O(C_M |\xi-\eta|^{-1+\alpha}) \,d\eta   \\
&=   \int_{ \TT   } O(C_M |\xi-\eta|^{-1+\alpha}) \,d\eta  \leq C_M;
\end{aligned}
\end{equation}
the first term in the penultimate line is zero due to the odd symmetry.

\noindent
{\bf Step $2$: $W^{1,p}$ differentiability of $v$}

%Since $\g(s)$ is parametrized in arc-length, without loss of generality we assume $L=2 \pi$ so that we can identify $\g= \TT$.
%For simplicity of notation, in the argument below, we will denote $\TT$ and $\int_{\TT} ds$ the torus $L\TT$
%and integral $\int_{L\TT} ds$ where $L$ is the total length of the patch boundary. This short-hand convention applies to the proof of this lemma only.

To justify the formal differentiation under the integral, we define $ D_\ep (\xi )$, a function on $\TT$, by
$$
 D_\ep (\xi ) = g(\xi) \int_{ |\xi - \eta |\geq \ep } \T (\eta )\frac{   (\g( \xi )-\g( \eta )) \cdot \T( \xi ) }{  |\g( \xi )-\g(\eta)|^{2  }} g(\eta ) \, d\eta ,
$$
such that $\lim_{\ep \to 0^+}   D_\ep(\xi)   $ equals to the right hand side of \eqref{eq:d_v_xi} times $g(\xi)$. In fact, the proof of \eqref{eq:eq:d_v_xi_aux1} above implies that $|D_\ep ( \xi ) | \leq C_M  $ for any $ \xi \in \TT$ uniformly for $0<\ep\leq 1$. So by the dominated convergence, it suffices to show that for any $ \varphi \in C^\infty  (  \TT )$ we have
\begin{equation}\label{eq:d_v_aux0}
\lim_{\ep \to 0^+} \int_{ \TT } D_\ep    \varphi  \, d \xi  = - \int_{\TT }   v    \varphi'   \, d \xi .
\end{equation}

By Fubini for all $\ep>0$ we have
\begin{align*}
\int_{\TT  } D_\ep   \varphi \, d \xi   = \int_{ \TT  }\int_{ | \xi  -\eta|\geq \ep } g(\xi )\T (\eta) \frac{   (\g( \xi )-\g(\eta)) \cdot \T( \xi ) }{  |\g( \xi )-\g(\eta)|^{2  }} \varphi( \xi ) \,d \xi    \, g(\eta )d\eta  .
\end{align*}
Since in the region $|\xi -\eta| \geq \ep$, we have $g(\xi ) \frac{   (\g( \xi )-\g(\eta)) \cdot \T(\xi) }{  |\g( \xi )-\g(\eta)|^{2  }} = \p_\xi \ln|\g(\xi )-\g(\eta)|$, we can integrate by parts in the variable $\xi $, and the conclusion would follow if we can show that the boundary terms vanish:
\begin{equation}\label{eq:d_v_aux1}
\varphi( \xi  ) \ln|\g(\xi ) -\g(\eta) |  \Big|_{ \xi  =\eta -  \ep  }^{ \xi =\eta
+ \ep  }  \to 0 \quad \text{uniformly as $\ep \to 0$.}
\end{equation}
This is not hard to show by using the $W^{2,p}$, $p>1$ regularity of $\gamma$. Indeed, let  $\alpha=1 - \frac{1}{p}>0$ be such that $  W^{2,p}  (  \TT )\subset C^{1,\alpha} (  \TT )$. Then the mean value theorem with $\g \in B^M_p$ implies  that
\begin{align*}
|\g(\eta) -\g( \eta +\ep)|  & \leq g(\eta) \ep    +C_M \ep^{1+ \alpha} \\
 |\g(\eta) -\g( \eta - \ep)| & \geq g(\eta) \ep -C_M \ep^{1+ \alpha}.
\end{align*}
With these bounds, we can show \eqref{eq:d_v_aux1}.
Indeed, we have
\begin{align*}
\left|  \varphi ( \xi ) \ln|\g( \xi ) -\g(\eta) |  \Big|_{ \xi  =\eta -  \ep  }^{ \xi =\eta
+ \ep  } \right|   &=    \bigg| \big[   \varphi( \eta  +\ep   )  -  \varphi (\eta  - \ep  ) \big]  \ln| \g( \eta +\ep ) -\g( \eta) |\\
&\qquad \quad  + \varphi(\eta -\ep)  \ln \Big|  \frac{\g(\eta+\ep ) -\g(\eta) }{\g(\eta-\ep ) -\g(\eta)}  \Big|  \bigg| \\
 &\lesssim \ep |\varphi'|_{L^\infty(\TT)} \big(1+ |\ln  \ep   |\big) + \ln \Big|\frac{1+ C_M \ep^\alpha }{1- C_M \ep^\alpha} \Big| \to 0 \quad \text{as $\ep \to 0^+$}.
\end{align*}
With \eqref{eq:d_v_aux1} proved, the integration by parts is justified and we have established \eqref{eq:d_v_aux0}. So $v \in W^{1,p}(\TT)$ and \eqref{eq:d_v_xi} holds.

\noindent
{\bf Step $3$: bounds of \eqref{eq:d2_v_xi}}

For simplicity, let us still denote by $\p^2_s   v$ the right-hand side of \eqref{eq:d2_v_xi}. To show that this object is in $L^p(\TT)$, we consider the decomposition
\begin{align*}
		 \p^2_s   v(\xi )  & =   \sum_{1\leq i\leq 4}  K_i(\xi )
\end{align*}
where terms $K_i$ are explicitly defined below
\begin{align*}
K_1 &  =  -P.V.  \int_{\TT }   \k( \eta )\N ( \eta )    \frac{  (\gamma( \xi )- \gamma(  \eta )) \cdot \T( \xi )}{  |\gamma( \xi )- \gamma(  \eta )|^{2  }}  g(\eta )\, d \eta   \\
K_2 &=
  \int_{\TT } \T( \eta )  \frac{  (\T( \xi ) - \T(\eta  ))\cdot \T( \xi )   }{ |\gamma(\xi )- \gamma( \eta )|^{2 } }    g(\eta ) \, d\eta\\
K_3 &= - \k( \xi )\int_{\TT } \T( \eta )  \frac{  (\gamma(\xi )- \gamma(\eta ))\cdot \N( \xi )  \big] }{ |\gamma(\xi )- \gamma( \eta )|^{2 } }    g(\eta ) \, d\eta\\
K_4 &=-2\int_{\TT } \T(\eta ) \frac{    \big( (  \gamma(\xi)- \gamma( \eta )  ) \cdot \T(\xi) \big) \big( (  \gamma(\xi)- \gamma(\eta )  ) \cdot (\T(\xi) -\T(\eta )) \big)  }{|\gamma(\xi)- \gamma(\eta )|^{4  }}  g(\eta )\, d\eta.
\end{align*}

\noindent
{\bf Estimate of $K_1$: }

By Corollary \ref{corollary:building_blocks},
\begin{align*}
 | K_1 	 |_{L^p(\TT )}   &\leq  \left[   \int_{ \TT }  \left|   \int_{\TT}   \k(\eta  )\N ( \eta)   \frac{  ( \xi  - \eta  )  }{  g(\eta)| \xi - \eta  |^{2   }}  g(\eta ) \, d \eta    \right|^p  \, d\xi  \right]^\frac{1}{p} \\
 & \quad + O\left( C_M  \int_{\TT}  \left|   \int_{ \TT }  | \k(\eta )|  | \xi-   \eta  |^{ \alpha -1   }   \, d \eta    \right|^p  \, d\xi    \right)^\frac{1}{p}.
\end{align*}
Since $1<p<\infty$ and $ |\k|_{L^p(\TT)} \leq M   $, the $L^p$-boundedness of the Hilbert transform and Young's inequality imply that $ | K_1 	 |_{L^p(\TT )}    \leq C_M$.

\noindent
{\bf Estimate of $K_2$: }

By the maximal estimate \eqref{eq:building_blocks_M_d} from Lemma \ref{lemma:building_blocks} with $\zeta=\xi$,
\begin{align*}
| K_2 (\xi)| & \leq \int_{\TT }\Big|  \frac{  (\T( \xi ) - \T(\eta  ))\cdot \T( \xi )   }{ |\gamma(\xi )- \gamma( \eta )|^{2 } }    g(\eta ) \Big| \, d\eta\\
& \leq C_M  \int_{\TT }  \mathcal{M}\k ( \xi ) |\xi -\eta |^{-1+\alpha}  \, d\eta,
\end{align*}
so Young's inequality implies that $ | K_2 	 |_{L^p(\TT )}    \leq C_M$.

\noindent
{\bf Estimate of $K_3$: }

Since $|\k|_{L^p (\TT)} \leq C_M$, it suffices to show that
$$
\sup_{\xi} \Big|  \int_{\TT } \T( \eta )  \frac{    (\gamma(\xi )- \gamma(\eta ))\cdot \N( \xi )  \big] }{ |\gamma(\xi )- \gamma( \eta )|^{2 } }    g(\eta ) \, d\eta \Big| < \infty.
$$
Thanks to \eqref{eq:building_blocks_d}, we have
\begin{align*}
\sup_{\xi} \Big| \int_{\TT } \T( \eta )  \frac{    (\gamma(\xi )- \gamma(\eta )\cdot \N( \xi )  \big] }{ |\gamma(\xi )- \gamma( \eta )|^{2 } }    g(\eta ) \, d\eta  \Big| \leq C_M \sup_{\xi} \int_{\TT }  |\xi -\eta|^{-1+\alpha}    \, d\eta,
\end{align*}
and thus the bound for $K_3$ is established.

\noindent
{\bf Estimate of $K_4$: }

For $K_4$, we apply absolute value to the integrand and then the estimate \eqref{eq:building_blocks_M_e} to obtain
\begin{align*}
K_4 (\xi) & \leq C_M \int_{\TT } \left|   \frac{    \big( (  \gamma(\xi)- \gamma( \eta )  ) \cdot \T(\xi) \big) \big( (  \gamma(\xi)- \gamma(\eta )  ) \cdot (\T(\xi) -\T(\eta )) \big)  }{|\gamma(\xi)- \gamma(\eta )|^{4  }} \right|  \, d\eta.\\
&\leq C_M \int_{\TT } \left|   \frac{       (  \gamma(\xi)- \gamma(\eta )  ) \cdot (\T(\xi) -\T(\eta ))    }{|\gamma(\xi)- \gamma(\eta )|^{3  }} \right|  \, d\eta\\
&\leq C_M \int_{\TT } \mathcal{M}\k(\xi)  |\xi - \eta|^{-1+\alpha} \, d\eta.
\end{align*}
Since $|\k|_{L^p(\TT) } \leq C_M$, we conclude that $|K_4|_{L^p(\TT) } \leq C_M$.

We have thus shown the right-hand side of \eqref{eq:d2_v_xi} defines a function in $L^p(\TT)$.

\noindent
{\bf Step $4$: $W^{2,p}$ differentiability of $v$}

To establish  \eqref{eq:d2_v_xi}, we will rewrite the integrals involved in the arc-length parameter and differentiate in $s.$ This will reduce computations compared
to working in Lagrangian labels.
%This will also imply  $v  \in W^{2 ,p} $ in the original Lagrangian label since the metric $g \in W^{1,p}$.
Consider the arc-length $s = \ell(\xi) \equiv \int_0^\xi g(\tau ) \, d\tau $ of the curve, and we are going to denote $\gamma(s)$ the same curve in the arc-length variable $s$ and similarly for its tangent $\T(s)$, normal $\N(s)$, curvature $\k(s)$, and the velocity $v(s) $. In this parametrization, the right-hand side of \eqref{eq:d2_v_xi} becomes
\begin{equation}\label{eq:d2_v_aux000}
\begin{aligned}
       &-  \lim_{\ep\to 0 } \int_{|\ell^{-1}(s) -\ell^{-1}(s') | \geq \ep }   \k( s' )\N ( s' )    \frac{  (\gamma( s )- \gamma(  s' )) \cdot \T( s )}{  |\gamma( s )- \gamma(  s' )|^{2  }}   \, d s'  \\
		&   + \int_{\gamma } \T( s' )  \frac{ \big[ (\T( s ) - \T(s'  ))\cdot \T( s ) -\k(s ) (\gamma(s )- \gamma(s' ))\cdot \N( s )  \big]  }{ |\gamma( s )- \gamma( s' )|^{2 } }     \, ds'  \\
		&  -  2\int_{\gamma } \T(s' ) \frac{    \big( (  \gamma(  s )- \gamma( s'  )  ) \cdot \T(  s ) \big) \big( (  \gamma(s  )- \gamma(  s' )  ) \cdot (\T(s  ) -\T(  s'  )) \big)  }{|\gamma(  s  )- \gamma( s' )|^{4  }}   \, d s',
\end{aligned}
\end{equation}
where $\ell^{-1}$ is the inverse of the map $\xi \mapsto \ell(\xi) $.
Now we show that the first term in \eqref{eq:d2_v_aux000} is the Cauchy principal value integral in $s.$ Since the inverse map $\ell^{-1}$ is $C^{1,\alpha}$ with $\alpha =1 -\frac{1}{p}>0$, the (symmetric) difference of the sets
	$$
	 D_\ep:= \{s' : |s-s'| \geq  \ep \} \Delta \{s': |\ell^{-1}(s)-\ell^{-1}(s')|  \geq g( \ell^{-1}(s) )  \ep\}
	 $$
has Lebesgue measure at most $C_M \ep^{1+\alpha}$, uniformly in $s$. Indeed, by the fundamental theorem of calculus and $C^{1,\alpha}$ regularity of $\ell^{-1}$, for all
sufficiently small $\ep,$ the set $D_\ep$ can be covered by two intervals of length $C_M \ep^{1+\alpha}$ centered at $s \pm  \ep.$  This smallness condition on $\ep$ can be
made uniformly in $s$ since we are on a compact set.
Hence, we have that the difference between the first term in \eqref{eq:d2_v_aux000} (with modulated approximation parameter $  \frac{\ep}{g( \ell^{-1}(s) )}   $) and its counterpart of the Cauchy principal value  satisfies
\begin{align*}
 \Big|  	\int_{D_\ep }   \k( s' )\N ( s' )   & \frac{  (\gamma( s )- \gamma(  s' )) \cdot \T( s )}{  |\gamma( s )- \gamma(  s' )|^{2  }}   \, d s'    \Big|  \leq C_M \int_{D_\ep }|\k(s')| |s-s'|^{-1}	 \\
 & \leq  C_M \ep^{\alpha} \big[\mathcal{M} \k ( s -   \ep) +  \mathcal{M} \k ( s+ \ep) \big]
\end{align*}
where we have used the definition of the maximal function and the bound $|s-s'|^{-1} \leq C_M \ep^{-1}$ on $D_\ep$.
Since $\k \in L^p$, this term converges to $0$ in $L^p$ as $\ep\to 0$. %On the other hand, the principal value integral with modulated approximation parameter coincides almost everywhere with the usual one (this follows from a.e. convergence that is a consequence of the $L^p$ boundedness of the maximal Hilbert transform e.g. \cite{SteinWeiss}).
Therefore, the first term in \eqref{eq:d2_v_aux000} is equal to  $ -P.V.    \int_{\gamma }   \k( s' )\N ( s' )    \frac{  (\gamma( s )- \gamma(  s' )) \cdot \T( s )}{  |\gamma( s )- \gamma(  s' )|^{2  }}   \, d s' $ almost everywhere $s\in \frac{L}{2\pi} \TT$. A similar reasoning also shows that in the arc-length \eqref{eq:d_v_xi} becomes $\p_s v = P.V.    \int_{\gamma }   \T ( s' )    \frac{  (\gamma( s )- \gamma(  s' )) \cdot \T( s )}{  |\gamma( s )- \gamma(  s' )|^{2  }}   \, d s' $. Hence to verify \eqref{eq:d2_v_xi}, it suffices to show that
\begin{equation}\label{eq:d2_v_aux0}
\begin{aligned}
\p_s(\p_s  v)  =       &-P.V.    \int_{\gamma }   \k( s' )\N ( s' )    \frac{  (\gamma( s )- \gamma(  s' )) \cdot \T( s )}{  |\gamma( s )- \gamma(  s' )|^{2  }}   \, d s'  \\
		&   + \int_{\gamma } \T( s' )  \frac{ \big[ (\T( s ) - \T(s'  ))\cdot \T( s ) -\k(s ) (\gamma(s )- \gamma(s' ))\cdot \N( s )  \big]  }{ |\gamma( s )- \gamma( s' )|^{2 } }     \, ds'  \\
		&  -  2\int_{\gamma } \T(s' ) \frac{    \big( (  \gamma(  s )- \gamma( s'  )  ) \cdot \T(  s ) \big) \big( (  \gamma(s  )- \gamma(  s' )  ) \cdot (\T(s  ) -\T(  s'  )) \big)  }{|\gamma(  s  )- \gamma( s' )|^{4  }}   \, d s'.
\end{aligned}
\end{equation}
Since Step 3 above also shows the right-hand side of \eqref{eq:d2_v_aux0} is $L^p $, it suffices to show  that for any $ \varphi \in C^\infty (  \gamma )$,
\begin{equation}\label{eq:d2_v_aux1}
\lim_{\ep \to 0^+} \int_{\gamma  }   \int_{ |s - s' |\geq \ep } \T ( s'  )\frac{   (\g( s )-\g( s' )) \cdot \T( s ) }{  |\g(s )-\g(  s' )|^{2  }}  \, d s'        \p_s \varphi (s) \, d s  = - \int_{\gamma }   %\p_s^2 v(s)
\eqref{eq:d2_v_aux0}  \varphi (s) \, ds .
\end{equation}
We proceed to a proof of \eqref{eq:d2_v_aux1}. Reparametrizing the inner integral  via   $ s' \mapsto s + s'  $ and then applying Fubini yield
\begin{align}
\lim_{\ep \to 0^+} \int_{\gamma  }   & \int_{ |s - s' |\geq \ep } \T ( s'  )\frac{   (\g( s )-\g( s' )) \cdot \T( s ) }{  |\g(s )-\g(  s' )|^{2  }}  \, d s'       \p_s \varphi (s) \, d s  \nonumber \\
&  = \lim_{\ep \to 0^+}\int_\gamma \int_{|  s' | \geq \ep } \T (s  + s' )\frac{   (\g(s )-\g (  s  + s' )) \cdot \T( s ) }{  |\g( s )-\g(  s  + s' )|^{2  }}   \,d s'   \p_s \varphi (s ) \, ds \nonumber \\
&= \lim_{\ep \to 0^+}\int_{|  s' | \geq \ep } \int_\gamma \underbrace{\T (s  + s' )\frac{   (\g(s )-\g (  s  + s' )) \cdot \T( s ) }{  |\g( s )-\g(  s  + s' )|^{2  }} }_{ \text{a.e.  differentiable with respect to $s$} }   \p_s \varphi (s) \, ds \,   d s' \label{eq:d_v_aux2}.
\end{align}
Since $\g$ is a simple $W^{2,p} $ curve, for each fixed   $|s'| \geq \ep$, the inner integrand in \eqref{eq:d_v_aux2} is  a.e.  differentiable with respect to $s$, with derivative
\begin{equation}\label{eq:d_v_aux3}
	\begin{aligned}
 & -  \k(s  + s' )\N (s  + s' ) \frac{   (\gamma(  s  )- \gamma(  s  + s'  )) \cdot \T(  s  )}{  |\gamma( s  )- \gamma( s  + s' )|^{2  }}   \\
		& \quad +  \T( s  + s'  ) \frac{ \big[ (\T(  s  ) - \T( s  + s'  ))\cdot \T(s ) - (\gamma(  s )- \gamma( s  + s' ))\cdot \k( s )\N( s )  \big] }{ |\gamma( s )- \gamma( s  + s' )|^{2  } }     \\
		&\qquad -2   \frac{   \T( s  + s' ) \big( (  \gamma( s  )- \gamma( s  + s'  )  ) \cdot \T( s ) \big) \big( (  \gamma( s )- \gamma( s  + s'  )  ) \cdot (\T( s ) -\T( s  + s' )) \big)  }{|\gamma( s )- \gamma( s  + s' )|^{4  }} .
	\end{aligned}
\end{equation}
As a result, based on the $L^p$ estimates we proved earlier, we can integrate by parts in the $s $ variable in \eqref{eq:d_v_aux2} and use Fubini once again to obtain that
\begin{align*}
\lim_{\ep \to 0^+} \int_{\gamma  }   & \int_{ |s - s' |\geq \ep } \T ( s'  )\frac{   (\g( s )-\g( s' )) \cdot \T( s ) }{  |\g(s )-\g(  s' )|^{2  }}  \, d s'       \p_s \varphi (s) \, d s \\
&= -\lim_{\ep \to 0^+}\int_{|  s' | \geq \ep } \int_\gamma \eqref{eq:d_v_aux3} \varphi(s  ) \, d s \,d s' \\
&= -\lim_{\ep \to 0^+} \int_\gamma \int_{| s' | \geq \ep } \eqref{eq:d_v_aux3} \varphi(s ) \, ds' \,d s .
\end{align*}
We reparametrize back $s'+s \mapsto s'$ and compare the above to \eqref{eq:d2_v_aux1} to obtain
\begin{align*}
\lim_{\ep \to 0^+} \int_{\gamma  }    & \int_{ |s - s' |\geq \ep } \T ( s'  )\frac{   (\g( s )-\g( s' )) \cdot \T( s ) }{  |\g(s )-\g(  s' )|^{2  }}  \, d s'       \p_s \varphi (s) \, d s  \\
&  = - \int_{\gamma }   \p_s^2 v   \varphi(s)  \, d s +\lim_{\ep \to 0^+}  \int_\gamma    \varphi(s)  \int_{   |  s  - s' | < \ep} Z ( s  ,s')\, ds'  \,d  s
\end{align*}
where the integrand $Z ( s  , s') $ corresponds to the last two lines of \eqref{eq:d_v_aux3} (since the principal value term cancels by definition) and is given by
\begin{align*}
Z ( s , s' ):= &    -     \T(s' )  \frac{ \big[ (\T(s ) - \T(s' ))\cdot \T(s ) -\k(s ) (\gamma(s)- \gamma(s' ))\cdot \N( s )  \big] }{ |\gamma(s )- \gamma( s' )|^{2 } }      \\
		&\qquad + 2   \T(s' ) \frac{    \big( (  \gamma( s )- \gamma(s' )  ) \cdot \T(s ) \big) \big( (  \gamma( s )- \gamma( s' )  ) \cdot (\T( s ) -\T( s' )) \big)  }{|\gamma( s )- \gamma( s' )|^{4  }}.
\end{align*}
To show \eqref{eq:d2_v_aux1} holds, we need to show the error vanishes:
\begin{equation}\label{eq:d_v_auxZerror}
 \begin{aligned}
 \int_\g   \varphi  \int_{   |  s -s'| < \ep} Z ( s  ,s' )\, ds'  \,d  s    \to 0 \quad \text{as $\ep \to 0$} .
\end{aligned}
\end{equation}

This is essentially done in Step 3, and here we present a proof of \eqref{eq:d_v_auxZerror} using Lemma \ref{lemma:building_blocks} for completeness.  Observe that \eqref{eq:building_blocks_M_d}, \eqref{eq:building_blocks_d}, and \eqref{eq:building_blocks_f} imply
\begin{align}
\Big| \T( s'  )  \frac{ \big[ (\T( s  ) - \T(s' ))\cdot \T( s  )   \big] }{ |\gamma( s )- \gamma( s' )|^{2 } }  \Big| &\leq C_M  \mathcal{M}\k( s ) | s -  s' |^{-1+\alpha} \label{eq:d_v_auxZ1a}\\
 \Big |\T( s'  )  \frac{ \k(s  ) (\gamma(s  )- \gamma( s'  ))\cdot \N(s  )  \big] }{ |\gamma(s  )- \gamma( s' )|^{2 } }   \Big| &  \leq   C_M \k( s ) |s - s'|^{-1+\alpha}  \label{eq:d_v_auxZ1b};
\end{align}
while \eqref{eq:building_blocks_M_e} and \eqref{eq:building_blocks_f} imply
\begin{align}\label{eq:d_v_auxZ1c}
 \Big| \T( s'  ) &\frac{    \big( (  \gamma(s )- \gamma( s' )  ) \cdot \T( s  ) \big)  \big( (  \gamma( s  )- \gamma( s' )  ) \cdot (\T( s  ) -\T( s'  )) \big)  }{|\gamma( s )- \gamma( s' )|^{4  }} \Big|\nonumber \\
  &\qquad \leq  C_M \mathcal{M}\k( s ) | s  -  s' |^{-1+\alpha} .
\end{align}
These estimates \eqref{eq:d_v_auxZ1a}--\eqref{eq:d_v_auxZ1c} imply for all $s' \neq s $ the bound
\begin{align}\label{eq:d_v_auxZ2}
 \left|  Z ( s  , s' )\right| & \leq C_M \Big( \mathcal{M}\k(s ) |s  -s' |^{-1+\alpha } + \k(s ) | s - s' |^{-1 +\alpha}   \Big).
\end{align}
By \eqref{eq:d_v_auxZ2}, the H\"older inequality, and the boundedness of the maximal function in $L^p$ for $p>1$, we have that
\begin{align*}
 \lim_{\ep \to 0^+} &\bigg|\int_\g \varphi(s ) \int_{   | s  - s' | < \ep} Z \, d s'  \,d s \bigg|   \\
 & \leq C_M |\varphi|_{L^\infty }\lim_{\ep \to 0^+}    \int_{   | s | < \ep}   |s |^{-1+ \alpha}   \,d s     = 0.
\end{align*}
Now that \eqref{eq:d2_v_aux1} is established, we have that $ \p_s^2 v \in L^{ p} $ and  \eqref{eq:d2_v_aux0} holds, and thus in the original label, \eqref{eq:d2_v_xi} holds as well.

\end{proof}

\subsection{Contraction estimates of the solution map}
%%%%%%%%%%%%%%%%%%%%%%%%%%%%%%%%%%%%%%%%%%%%%%%%%%%%%%%%%

The last ingredient for the Banach fixed-point argument is the Lipschitz continuity for the nonlinear map $\gamma \mapsto v(\g)  $ in the Sobolev space $W^{2,p}$. The main results are Proposition \ref{prop:Lip_v} and Proposition \ref{prop:Lip_ds2v} below.

We fix $1<p<\infty$ in this subsection and  let $\g_i  \in B^M_p$ for $i=1,2$. We use the notation $ g_i $, $\T_i$, $\N_i$ and $\k_i$ to denote the arc-length, tangent, normal, and curvature of the curve $\g_i$. In addition $v_i:\TT\to \RR^2$ denotes the velocity associated to $\g_i$, and $ \p_s v_i   = \frac{1}{g_i} \dot{v_i}$ denotes the corresponding differentiation in arc-length on the curve $\g_i$.

Since we will be frequently taking difference between functions defined by $\g_1 $ and $\g_2$, we introduce the notation $\Delta \left[f_i \right] := f_1 -f_2$ for any functions $f_i$ defined by $\g_i $ on $\TT$. For instance, $\Delta [\g_i](\xi) = \g_1(\xi) -\g_2(\xi)  $ and $\Delta \left[ \T_i(\xi) - \T_i(\eta) \right]= \left[ \T_1(\xi) - \T_1(\eta) \right]-\left[ \T_2(\xi) - \T_2(\eta) \right]    $. We will frequently use the telescoping formula
\begin{equation}\label{aux22122a}
\Delta [f_i g_i] = \Delta [f_i] g_1 +f_2 \Delta [ g_i].
\end{equation}
For brevity, we denote by  $\delta \geq 0$ the distance between two sets of data in $X_p$:
\begin{equation}\label{aux22122b}
\delta=  |\g_1-\g_2|_{W^{2,p} (\TT)} .
\end{equation}
These conventions allow for a more streamlined argument.

Again, we start with a few basic estimates that will serve as the ``building blocks'' in the estimation below. Recall that $\alpha:= 1- \frac{1}{p}$  and $C_M>0$ denotes a constant   depending on $M$ and $p$ that may change from line to line.

The assumption on $\g_i$ and the definition of $\delta$ imply directly the estimates
\begin{subequations}
\begin{align}
|\Delta [\T_i]|_{L^\infty(\TT)} &\leq C_M \delta \label{eq:Delta_T_g_a}\\
|\Delta [g_i]|_{L^\infty(\TT)} \leq C|\Delta [g_i]|_{W^{1,p}} &\leq C_M \delta. \label{eq:Delta_T_g_b}\\
|\Delta [\k_i]|_{L^p(\TT)} &\leq C_M \delta, \label{eq:Delta_T_g_c}
\end{align}
\end{subequations}
which by the fundamental theorem of calculus and simple telescoping further implies the following set of estimates
\begin{subequations}
\begin{align}
 \big| \Delta \left[\g_i(\xi)  - \g_i(\eta)\right] \big| &\leq C_M \delta |\xi -\eta| \label{eq:Delta_points_a}\\
\big| \Delta \left[\T_i(\xi)  - \T_i(\eta)\right] \big| &\leq C_M \delta |\xi -\eta|^{\alpha}\label{eq:Delta_points_b}\\
\big| \Delta \left[g_i(\xi)  - g_i(\eta)\right] \big| &\leq C_M \delta |\xi -\eta|^{\alpha}. \label{eq:Delta_points_c}
\end{align}
\end{subequations}

The following lemma is our main building block for proving the Lipschitz continuity for the map $\gamma \mapsto v(\g)  $ in $W^{2,p}(\TT)$.

\begin{lemma}\label{lemma:building_blocks_difference}
Let $\g_i \in B_p^M$, $i= 1,2$. For any $\xi ,\eta \in \TT$,
\begin{subequations}
\begin{align}
 \Delta \big[ (\g_{i}(\xi ) -\g_i( \eta ))  \cdot \T_i( \xi) \big] &= \Delta [g_i](\xi)  (\xi -\eta ) + O(C_M \delta |\xi -\eta|^{1+\alpha}) \label{eq:building_blocks_difference_a}\\
 \Delta \big[ \big (\gamma_{i}(\xi)- \gamma_{i}(\eta )\big)\cdot \N_{i}(\xi) \big]    & = O(C_M \delta |\xi -\eta|^{1+\alpha })\label{eq:building_blocks_difference_b}
\end{align}
\end{subequations}
and there exists a bounded function $ |C_\delta (\xi )| \leq C_M \delta $ such that
\begin{align}
\Delta\left[\frac{ 1    }{ |\gamma_{i}(\xi)- \gamma_{i}(\eta)|^{2  } }     \right] &=  \frac{ C_\delta (\xi )
     }{| \xi -\eta|^{2  }}  +O(C_M \delta |\xi -\eta|^{-2+\alpha}).\label{eq:building_blocks_difference_g}
 \end{align}
In addition, for any $\xi ,\eta \in \TT$ and any $ \zeta \in \TT$ such that $|\eta-\zeta|, |\xi-\zeta| \leq |\xi-\eta|$
%lying between $\eta$ and $\xi$,
we have the maximal estimates
\begin{subequations}
\begin{equation}
\begin{aligned}
&\left| \Delta\left[    \left( \T_i(\xi) -\T_i(\eta)  \right)  \cdot\T_i(\xi) \right] \right| \\
& \qquad \leq  C_M|\xi-\eta|^{1  +\alpha}  \Big( \delta   \max_i \mathcal{M} \k_i (\zeta) )  +  \mathcal{M} \Delta[\k_i] (\zeta)   \Big)   ,  \label{eq:building_blocks_difference_Ma}\\
\end{aligned}
\end{equation}
\begin{equation}
\begin{aligned}
& \left| \Delta\left[   \left(  \g_i (\xi) - \g_i(\eta) \right) \left( \T_i(\xi) -\T_i(\eta)  \right)    \right] \right|  \\
& \qquad \leq   C_M|\xi-\eta|^{2  +\alpha}  \Big( \delta   \max_i \mathcal{M} \k_i (\zeta) )  +  \mathcal{M} \Delta[\k_i] (\zeta)   \Big)    \label{eq:building_blocks_difference_Mb}.
\end{aligned}
\end{equation}
\end{subequations}
\end{lemma}
\begin{proof}
We prove these one by one using repeatedly the fundamental theorem of calculus, \eqref{eq:Delta_T_g_a}--\eqref{eq:Delta_T_g_c}, and \eqref{eq:Delta_points_a}--\eqref{eq:Delta_points_c}.

\noindent{\textbf{Estimate for \eqref{eq:building_blocks_difference_a}.}}

We first rewrite the left-hand side by the telescoping formula \eqref{aux22122a}
\begin{align*}
   \Delta \left[ \big( \g_{i}(\xi ) -\g_i( \eta ) \big) \cdot \T_i( \xi)\right]  &= \int_{\eta}^{\xi}    \Delta \left[   \T_i( \tau ) g_i(\tau )  \cdot \T_i( \xi)\right] \, d\tau\\
   & = \int_{\eta}^{\xi}   \left( \Delta \left[ \big( \T_i( \tau) -\T_i( \xi)\big) g_i(\tau ) \cdot \T_i( \xi)\right]  +  \Delta [g_i](\tau ) \right) \, d\tau\\
   & = \Delta [g_i](\xi) (\xi -\eta )+\int_{\eta}^{\xi}   \left(  \Delta \left[ \big( \T_i( \tau )-\T_i( \xi)\big) g_i(\tau ) \cdot \T_i( \xi)\right] \right. \\
   &   \qquad\qquad\qquad\qquad\qquad\qquad+ \left. \Delta \left[ g_i(\tau )  -g_i(\xi)\right] \right) \, d\tau.
\end{align*}
By \eqref{eq:Delta_T_g_a}, \eqref{eq:Delta_T_g_b}, and \eqref{eq:Delta_points_b},
$$
|\Delta \left[ \big( \T_i( \tau )-\T_i( \xi)\big) g_i(\tau ) \cdot \T_i( \xi)\right]| = O(C_M \delta |\xi -\tau|^\alpha),
$$
so this and \eqref{eq:Delta_points_c} prove \eqref{eq:building_blocks_difference_a}:
$$
\Delta \left[ \big( \g_{i}(\xi ) -\g_i( \eta ) \big) \cdot \T_i( \xi)\right]  = \Delta [g_i](\xi)  (\xi -\eta ) +O(C_M \delta |\xi-\eta|^{1+\alpha}).
$$

\noindent{\textbf{Estimate for \eqref{eq:building_blocks_difference_b}.}}

We use the fundamental theorem of calculus twice to obtain
\begin{align}\label{eq:building_blocks_difference_aux1}
\Delta \left[ \big( \g_{i}(\xi ) -\g_i( \eta ) \big) \cdot \N_i( \xi)\right]  &=  \Delta \left[ \int_{\eta}^{\xi}    \int^{\xi}_{\tau}    \k_i( \tau' )g_i( \tau ') \N_i( \tau' ) \cdot \N_i( \xi) g_i(\tau )  \, d\tau' \, d\tau \right].
\end{align}
Distributing $\Delta$ in \eqref{eq:building_blocks_difference_aux1} gives
\begin{align*}
\Delta \left[ \big( \g_{i}(\xi ) -\g_i( \eta ) \big) \cdot \N_i( \xi)\right]  &=   \int_{\eta}^{\xi}    \int^{\xi}_{\tau}  \Delta \left[  \k_i( \tau' )\right] g_1( \tau ') \N_1( \tau' ) \cdot \N_1( \xi) g_1(\tau )  \, d\tau' \, d\tau \\
& \qquad +   \int_{\eta}^{\xi}    \int^{\xi}_{\tau}    \k_2( \tau' ) \Delta \left[g_i( \tau ') \N_i( \tau' ) \cdot \N_i( \xi) g_i(\tau )  \right]\, d\tau' \, d\tau .
\end{align*}
We can estimate the first term by its absolute value and the second term by using the pointwise bound $| \Delta[    g_i( \tau ') \N_i( \tau' ) \cdot \N_i( \xi) g_i(\tau )  ] | \leq C_M \delta $ thanks to \eqref{eq:Delta_T_g_a}--\eqref{eq:Delta_T_g_c}. These considerations along with the assumption $|\k|_{L^p(\TT)} \leq M$ and the H\"older inequality imply
\begin{align*}
\Delta \left[ \big( \g_{i}(\xi ) -\g_i( \eta ) \big) \cdot \N_i( \xi)\right]  &=   O(C_M|\Delta[\k_i]|_{L^p(\TT)} |\xi -\eta|^{1+\alpha } )  \\
& \qquad\qquad +  O\left(C_M  \delta \int_{\eta}^{\xi}    \int^{\xi}_{\tau}   | \k_2( \tau' )|   \, d\tau' \, d\tau \right) \\
& =   O(C_M \delta |\xi -\eta|^{1+\alpha }).
\end{align*}

\noindent{\textbf{Estimate for   \eqref{eq:building_blocks_difference_g}.}}

We start with
\begin{align}\label{eq:Delta_gamma-2_1}
\Delta\left[\frac{ 1    }{ |\gamma_{i}(\xi)- \gamma_{i}(\eta)|^{2  } }     \right] &= \frac{ \Delta[\g_i(\xi) - \g_i(\eta)] \cdot \left(\g_1(\xi) -\g_1(\eta)+ \g_2(\xi) -\g_2(\eta)  \right) }{|\g_1 (\xi) - \g_1(\eta)|^2|\g_2 (\xi) - \g_2(\eta)|^2}.
\end{align}
Observe that
\begin{equation}\label{eq:Delta_gamma-2_2}
\Delta \left[  \gamma_i(\xi) - \g_i(\eta)  \right]    = \Delta \left[ \dot{\g}_i(\xi)   \right] (\xi -\eta) + O(C_M \delta |\xi - \eta|^{1+\alpha}).
\end{equation}
Also, as in \eqref{aux22122}, for any $\xi,\eta \in \TT$ with $\xi\neq \eta$, we have
\begin{equation}\label{eq:Delta_gamma-2_3}
 \frac{   |\xi  -   \eta|^{2 } }{  |   \gamma_i(\xi) - \gamma_i(\eta)   |^{2 } }   = \frac{   1 }{  g_i (\xi)^{2 } } + O(  C_M |\xi -\eta|^{\alpha}  ).
\end{equation}
So by \eqref{eq:Delta_gamma-2_1}, \eqref{eq:Delta_gamma-2_2}, and \eqref{eq:Delta_gamma-2_3}, we have for some bounded function $|C_\delta(\xi)| \leq C_M \delta $
%(determined by $g_i$)
that
\begin{align*}
\Delta\left[\frac{ 1    }{ |\gamma_{i}(\xi)- \gamma_{i}(\eta)|^{2  } }     \right] &=    \frac{ \left[  \Delta \left[ \dot{\g}_i(\xi)   \right] (\xi -\eta) + O(C_M \delta |\xi - \eta|^{1+\alpha})  \right]     }{  |g_{1}(\xi)|^2|g_{2}(\xi)|^2 | \xi -  \eta |^{4  } }   \\
&\qquad \qquad \times\left[ (\dot{\g_1}(\xi) +  \dot{\g_2}(\xi) )(\xi - \eta) + O(C_M |\xi -\eta|^{1+\alpha}) \right] \\
&\qquad\qquad   + O(C_M \delta|\xi -\eta|^{-2+\alpha })  \\
     & = \frac{ C_\delta(
     \xi )
     }{|\xi-\eta|^2} +O(C_M \delta |\xi -\eta|^{\alpha-2}).
\end{align*}

\noindent{\textbf{Estimate for   \eqref{eq:building_blocks_difference_Ma}.}}

By the telescoping formula \eqref{aux22122a},
\begin{align*}
  \Delta\left[    \left( \T_i(\xi) -\T_i(\eta)  \right)  \cdot\T_i(\xi) \right] &= -\int_{\eta}^{\xi}    \Delta \left[   \N_i( \tau )   \cdot \T_i( \xi) \k_i(\tau ) g_i(\tau ) \right] \, d\tau \\
&= -\int_{\eta}^{\xi}    \Delta \left[   \N_i( \tau )   \cdot \T_i( \xi) \right] \k_1(\tau ) g_1(\tau )  \, d\tau \\
&\qquad- \int_{\eta}^{\xi}    \N_2( \tau )   \cdot \T_2( \xi)  \Delta \left[  \k_i(\tau )\right] g_1(\tau )  \, d\tau \\
&\qquad\qquad- \int_{\eta}^{\xi}    \N_2( \tau )   \cdot \T_2( \xi)    \k_2(\tau ) \Delta \left[g_i(\tau ) \right] \, d\tau .
\end{align*}

We use \eqref{eq:building_blocks_c}, the difference bounds \eqref{eq:Delta_points_a}--\eqref{eq:Delta_points_c} to obtain that
\begin{align*}
  \Delta\left[    \left( \T_i(\xi) -\T_i(\eta)  \right)  \cdot\T_i(\xi) \right]
&= O\left( C_M \delta |\xi-\eta|^\alpha  \int_{\eta}^{\xi}  |\k_1(\tau ) |    \, d\tau \right)\\
&\qquad + O\left( C_M   |\xi-\eta|^\alpha  \int_{\eta}^{\xi}    \big|  \Delta \left[  \k_i(\tau )\right]    \big| \, d\tau\right) \\
&\qquad\qquad+ O\left(  C_M \delta |\xi-\eta|^\alpha  \int_{\eta}^{\xi}       \big|\k_2(\tau ) \big|   \, d\tau \right) .
\end{align*}
So by the definition of maximal function we have
\begin{align*}
  \Delta\left[    \left( \T_i(\xi) -\T_i(\eta)  \right)  \cdot\T_i(\xi) \right]
&= O ( C_M \delta |\xi-\eta|^{1+\alpha}   \max_i \mathcal{M} \k_i (\zeta) ) \\
&\qquad  + O ( C_M   |\xi-\eta|^{1+\alpha}    \mathcal{M} \Delta[\k_i] (\zeta) ) .
\end{align*}

\noindent{\textbf{Estimate for   \eqref{eq:building_blocks_difference_Mb}}}

Since by the fundamental theorem of calculus,
\begin{align*}
\Delta & \left[   \left(  \g_i (\xi) - \g_i(\eta) \right) \left( \T_i(\xi) -\T_i(\eta)  \right)    \right] \\
&= - \int_{\eta}^{\xi} \int_{\eta}^{\xi} \Delta \left[ g_i(\tau ) \T_i (\tau ) \cdot \N_i(\tau' )  \k_i (\tau' )  g_i(\tau' ) \right] \, d\tau '  d\tau,
\end{align*}
the bound follows from the same argument for \eqref{eq:building_blocks_difference_Ma}. The extra power of $|\xi -\eta|$ is given by the extra layer of integration.

\end{proof}

Thanks to \eqref{eq:building_blocks_difference_g} and \eqref{eq:building_blocks_difference_a}, we have the following estimate for the differences of factors with linear dispersion in $\p_s^2 v$.

\begin{corollary}\label{corollary:delta_linear}
Let $\g_i\in B_p^M$, $i= 1,2$. There exists a bounded function $|C_\delta (\xi)| \leq C_M \delta  $ such that for all $\xi , \eta \in \TT$, we have
\begin{subequations}
\begin{align}
    \Delta\left[\frac{ (\g_i(\xi) - \g_i(\eta)) \cdot \T_i(\xi )    }{ |\gamma_{i}(\xi)- \gamma_{i}(\eta)|^{2  } }     \right] & =  C_\delta (\xi) \frac{\xi -\eta}{|\xi -\eta|^2} + O(C_M \delta |\xi -\eta|^{1-\alpha }).
\end{align}
\end{subequations}

\end{corollary}

With all the preparations, we first show the Lipschitz continuity of $\g \to v(\g)$ in $B_p^M$ with a norm depending on $M$. As before, we will frequently use the telescoping formula
$$
\Delta [f_i g_i] = \Delta [f_i] g_1 +f_2 \Delta [ g_i].
$$

\begin{proposition}\label{prop:Lip_v}
Let $\g_i \in B_p^M$, $i= 1,2$. Denote the distance between them in $X_p$ by $\delta =  |\Delta[ \g_i ]  |_{W^{2,p}(\TT)}$. The map $\g\mapsto v(\g)$ satisfies
\begin{equation}
\left|    v(\g_1 )  -  v (\g_2) \right|_{ L^{\infty}(\TT)  } \leq C_M \delta.
\end{equation}
%where $ \p_s  v(\g_i)$ indicates the   arc-length differentiation on each $\g_i$.
\end{proposition}
\begin{proof}

\begin{align*}
\Delta  \left[ v(\g_i )\right]
&= \int_{\TT} \Delta  \left[ \T_i(\eta) g_i(\eta)\right] \ln|\g_1(\xi) - \g_1(\eta)|  \, d\eta \\
&\qquad + \int_{\TT}  \T_2(\eta) g_2(\eta) \Delta  \left[\ln|\g_i(\xi) - \g_i(\eta)| \right] \, d\eta .
\end{align*}
For the first term, since $\g_i \in B_p^M$, we have $C_M^{-1}|\xi -\eta | \leq |\g_i(\xi) - \g_i(\eta)| \leq C_M |\xi -\eta | $. This implies the bound  $\big| \ln|\g_i(\xi) - \g_i(\eta)| \big|  \leq C_M +  \left| \ln   |\xi -\eta |    \right|$. It then follows from \eqref{eq:Delta_T_g_a} and \eqref{eq:Delta_T_g_b} that
\begin{equation}\label{eq:lip_v_aux1}
\left| \int_{\TT} \Delta  \left[ \T_i(\eta) g_i(\eta)\right] \ln|\g_1(\xi) - \g_1(\eta)|  \, d\eta \right| \leq C_M \delta \int_{\TT} \big| \ln|\g_1(\xi) - \g_1(\eta)| \big| \, d\eta \leq C_M \delta.
\end{equation}

For the second term, without loss of generality we assume $\left| \frac{\g_1(\xi) - \g_1(\eta)}{\g_2(\xi) - \g_2(\eta)}\right|\geq 1 $. We start with an elementary inequality $\ln(1+x ) \leq x$ for $x>0$:
\begin{align*}
\left|  \Delta  \left[\ln|\g_i(\xi) - \g_i(\eta)| \right]   \right|&=       \ln  \left| \frac{\g_1(\xi) - \g_1(\eta)}{\g_2(\xi) - \g_2(\eta)}\right|    \\
& \leq \ln  \left(1 + \left|  \frac{\Delta\left[\g_i(\xi) - \g_i(\eta)\right]}{\g_2(\xi) - \g_2(\eta)} \right|\right) \leq \left|  \frac{\Delta\left[\g_i(\xi) - \g_i(\eta)\right]}{\g_2(\xi) - \g_2(\eta)} \right| .
\end{align*}
Since the definition of $\delta$ and assumptions $\g_i \in B^M_p$ imply that
$$
\left|  \Delta  \left[\ln|\g_i(\xi) - \g_i(\eta)| \right]   \right| \leq C_M \delta,
$$
we have the bound for the second part:
\begin{equation}\label{eq:lip_v_aux2}
\left| \int_{\TT}  \T_2(\eta) g_2(\eta) \Delta  \left[\ln|\g_i(\xi) - \g_i(\eta)| \right] \, d\eta \right| \leq C_M   \delta.
\end{equation}
Combining the two parts \eqref{eq:lip_v_aux1} and \eqref{eq:lip_v_aux2}, we conclude that
\begin{align*}
\left|   \Delta  \left[ v(\g_i )\right] \right|_{ L^{\infty}(\TT)  } \leq C_M \delta .
\end{align*}

\end{proof}

Next, we estimate the second-order derivative for the map $\g \mapsto v(\g)$ as the bound of the first-order derivative can be recovered by interpolation.
\begin{proposition}\label{prop:Lip_ds2v}
Let $\g_i \in B_p^M$, $i= 1,2$. Denote the distance between them in $X_p$ by $\delta =    |\Delta[\g_i]  |_{W^{2,p}(\TT)}$. The map $\g \mapsto v(\g)$ satisfies
\begin{equation}
\left| \p_s^2 v(\g_1) - \p_s^2 v(\g_2) \right|_{ L^{p} (\TT) } \leq C_M \delta.
\end{equation}
%where $ \p_s^2 v(\g_i)$ indicates the   arc-length differentiation on each $\g_i$.
\end{proposition}

\begin{proof}
We apply the difference $\Delta$ to each term in $\p_s^2 v$ \eqref{eq:d2_v_xi} and obtain the decomposition:
\begin{align*}
\Delta [ \p_s^2 v (\g_i)] :=  \sum_{1\leq n \leq 4}    H_{n}(\xi )  ,
\end{align*}
with
\begin{align*}
  H_{ 1}(\xi )  & := -   P.V. \int_{\TT }   \Delta\left[  \k_i(\eta)\N_i (\eta)    \frac{  (\gamma_i(\xi)- \gamma_i(\eta)) \cdot \T_i (\xi)}{  |\gamma_i (s)- \gamma_i (\eta)|^{2   }} \, g_i (\eta) \right] d\eta \\
 H_{ 2}(\xi )   & :=\int_{\TT } \Delta \left[ \T_i (\eta)    \frac{ \big[ (\T_i (\xi ) - \T_i (\eta))\cdot \T_i ( \xi)  \big] }{ |\gamma_i (\xi)- \gamma_i (\eta)|^{2  } }    \, g_i (\eta) \right] d\eta\\
 H_{ 3}(\xi )    & := -\Delta \left[\k_i(\xi)\int_{\TT }    \T_i (\eta)    \frac{ \big[ (\g_i (\xi ) - \g_i (\eta))\cdot \N_i ( \xi)  \big] }{ |\gamma_i (\xi)- \gamma_i (\eta)|^{2  } }    \, g_i (\eta) \right] d\eta
\end{align*}
and
\begin{align*}
 H_{ 4} (\xi )   :=& -  2  \int_{\TT }   \Delta \Big[  \T_i (\eta ) \big( (  \gamma_i  (\xi)- \gamma_i (\eta)  ) \cdot \T_i (\xi) \big) \\ & \qquad \frac{     \big( (  \gamma_i (\xi)- \gamma_i  (\eta)  ) \cdot (\T_i (\xi) -\T_i (\eta)) \big)  }{|\gamma_i (\xi)- \gamma_i (\eta)|^{4  }} \, g_i (\eta)  \Big]d\eta .
\end{align*}

\noindent
\textbf{Estimates of $  H_{ 1}$:}

We first telescope and obtain a further decomposition
\begin{align*}
H_{ 1}(\xi)& =  P.V.  \int_{\TT }   \Delta\left[  \k_i(\eta)  g_i (\eta)\N_i (\eta)\right]    \frac{  (\gamma_1(\xi)- \gamma_1(\eta)) \cdot \T_1 (\xi)}{  |\gamma_1 (s)- \gamma_1 (\eta)|^{2   }} \,  d\eta\\
& \qquad +  P.V. \int_{\TT }    \k_2(\eta)  g_2 (\eta) \N_2 (\eta) \Delta\left[ \frac{  (\gamma_i(\xi)- \gamma_i(\eta)) \cdot \T_i (\xi)}{  |\gamma_i (s)- \gamma_i (\eta)|^{2   }} \right]\,  d\eta\\
&:= H_{11}(\xi)+ H_{12}(\xi).
\end{align*}
By Corollary \ref{corollary:building_blocks}, the first term is equal to
\begin{align*}
 H_{11} (\xi ) & =  P.V.  \int_{\TT} \Delta\left[  \k_i(\eta)  g_i (\eta) \N_i(\eta )\right] \left( \frac{\xi-\eta}{g_1 (\eta )|\xi-  \eta|^2} + O(C_M|\xi-\eta|^{\alpha-1}) \right)\,d\eta.
\end{align*}
Thus by the $L^p$ boundedness of the Hilbert transform and Young's inequality, we get
\begin{align*}
|H_{11}|_{L^p (\TT)}
& \leq C_M \left| \Delta\left[  \k_i    \N_i g_i \right]\right|_{L^p(\TT) }   \leq C_M \delta,
\end{align*}
where in the last step we have used bounds \eqref{eq:Delta_T_g_a}--\eqref{eq:Delta_T_g_c}.

For the second term $H_{12}$, we apply Corollary \ref{corollary:delta_linear} to obtain that
\begin{align*}
|H_{12}|_{L^p (\TT) } & \leq  \left(\int_{\TT } \left| |C_\delta (\xi) |  P.V.  \int_{\TT }    \k_2(\eta)  \N_2 (\eta)   g_2 (\eta) \frac{\xi -\eta}{|\xi -\eta|^{2}}\,  d\eta \right|^p \, d\xi \right)^\frac{1}{p}\\
&\qquad +C_M \delta \left(\int_{\TT } \left|  \int_{\TT }      \k_2(\eta)   |\xi -\eta|^{\alpha-1} \,  d\eta \right|^p \, d\xi \right)^\frac{1}{p} .
\end{align*}
Then by the $L^p$ boundedness of the Hilbert transform and $|C_\delta (\xi)| \leq C_M \delta $ we get the same bound as $H_{11}$ for  $H_{12}$.

\noindent
\textbf{Estimates of $  H_{2}$:}

As before, we further split $H_2$,
\begin{align*}
H_{2}(\xi ) &= \int_{\TT } \Delta \left[ \T_i (\eta)  g_i (\eta) \right] \frac{ \big[ (\T_1 (\xi ) - \T_1 (\eta))\cdot \T_1 ( \xi)  \big] }{ |\gamma_1 (\xi)- \gamma_1 (\eta)|^{2  } }    \,  d\eta\\
&\qquad  +\int_{\TT }  \T_2 (\eta)     g_2 (\eta) \Delta \left[\frac{ \big[ (\T_i (\xi ) - \T_i (\eta))\cdot \T_i ( \xi)  \big] }{ |\gamma_i (\xi)- \gamma_i (\eta)|^{2  } }    \, \right] d\eta\\
&:= H_{21} (\xi ) +H_{22} (\xi ) .
\end{align*}
We first look at the term $ H_{21} $. By \eqref{eq:Delta_T_g_a} and \eqref{eq:Delta_T_g_b},
$$
\Delta \left[ \T_i (\eta)   g_i (\eta) \right] = O(C_M \delta  ),
$$
while \eqref{eq:building_blocks_f} and \eqref{eq:building_blocks_M_d} from Lemma \ref{lemma:building_blocks} imply that
$$
\frac{ \big[ (\T_1 (\xi ) - \T_1 (\eta))\cdot \T_1 ( \xi)  \big] }{ |\gamma_1 (\xi)- \gamma_1 (\eta)|^{2  } }= O(C_M  \mathcal{M}\k_1 (\xi)|\xi -\eta|^{\alpha -1 }).
$$
From these bounds, it follows that
$$
| H_{21}|_{L^p (\TT)} \leq C_M \delta  \left[ \int_{\TT } \Big| \int_{\TT }  \mathcal{M} \k_1(\xi) |\xi -\eta|^{ \alpha -1}    \,  d\eta\Big|^p \, d\xi \right]^\frac{1}{p} \leq C_M \delta,
$$
where we have used again the $L^p$ boundedness of the maximal function.

Next, we similarly bound $H_{22}$. As before, we first telescope,
\begin{equation}\label{eq:H2_aux1}
\begin{aligned}
\Delta \left[\frac{ \big[ (\T_i (\xi ) - \T_i (\eta))\cdot \T_i ( \xi)  \big] }{ |\gamma_i (\xi)- \gamma_i (\eta)|^{2  } }    \, \right] &= \Delta \left[\big[ (\T_i (\xi ) - \T_i (\eta))\cdot \T_i ( \xi)  \big]   \right] \frac{1 }{ |\gamma_1 (\xi)- \gamma_1 (\eta)|^{2  } }  \\
&\qquad+\big[ (\T_2 (\xi ) - \T_2 (\eta))\cdot \T_2 ( \xi)  \big]    \Delta \left[\frac{1 }{ |\gamma_i (\xi)- \gamma_i (\eta)|^{2  } }  \right].
\end{aligned}
\end{equation}
Let us consider terms in \eqref{eq:H2_aux1} one by one. By \eqref{eq:building_blocks_difference_Ma} from Lemma \ref{lemma:building_blocks_difference} and \eqref{eq:building_blocks_f} from Lemma \ref{lemma:building_blocks},
\begin{equation}
\begin{aligned}\label{eq:H2_aux2}
\frac{\Delta \left[\big[ (\T_i (\xi ) - \T_i (\eta))\cdot \T_i ( \xi)  \big]   \right]   }{ |\gamma_1 (\xi)- \gamma_1 (\eta)|^{2  } }   & = O ( C_M \delta |\xi-\eta|^{-1+\alpha}   \max_i \mathcal{M} \k_i (\xi) )  \\
&\qquad\qquad + O ( C_M   |\xi-\eta|^{-1+\alpha}    \mathcal{M} \Delta[\k_i] (\xi ) )
\end{aligned}
\end{equation}
and by \eqref{eq:building_blocks_M_d} and \eqref{eq:building_blocks_difference_g},
\begin{equation}
\begin{aligned}\label{eq:H2_aux3}
& \big[ (\T_2 (\xi ) - \T_2 (\eta))\cdot \T_2 ( \xi)  \big]    \Delta \left[\frac{1 }{ |\gamma_i (\xi)- \gamma_i (\eta)|^{2  } }  \right]\\
 & \qquad =  O(C_M \mathcal{M}\k_2(\xi)|\xi - \eta|^{-1+\alpha} C_\delta(\xi))  +O(C_M \delta  \mathcal{M}\k_2(\xi)|\xi - \eta|^{-1+2\alpha} ).
\end{aligned}
\end{equation}

Therefore, combining \eqref{eq:H2_aux2} and \eqref{eq:H2_aux3} we obtain
\begin{align*}
 \big| H_{22}(\xi)  \big| \leq C_M \delta \max_i \mathcal{M} \k_i (\xi) )      \int_{\TT }    |\xi -\eta|^{ \alpha -1}    \,  d\eta  + C_M    \mathcal{M} \Delta[\k_i] (\xi )     \int_{\TT }    |\xi -\eta|^{ \alpha -1}    \,  d\eta ,
\end{align*}
and integrating in $\xi$ gives  $| H_{21}|_{L^p (\TT)} \leq  C_M \delta. $

\noindent
\textbf{Estimates of $ H_{ 3}$:}

To show $|  H_{ 3}|_{ L^{p} (\TT) }
 \leq  C_M \delta  $, it suffices to obtain the bound
\begin{align}\label{eq:H3_aux1}
\sup_\xi \int_{\TT }
\left|  \T_i (\eta)  g_i (\eta)  \Delta \left[    \frac{ \big[ (\g_i (\xi ) - \g_i (\eta))\cdot \N_i ( \xi)  \big] }{ |\gamma_i (\xi)- \gamma_i (\eta)|^{2  } }    \, \right] \right|\,d\eta \leq C_M  \delta
\end{align}
since when the  difference $\Delta$ applies to $\k_i(\xi)$ or $\T_i(\eta)g_i(\eta)$, the bound of the resulting terms follows from \eqref{eq:building_blocks_d} and the definition of $\delta$.

To show \eqref{eq:H3_aux1}, we first telescope
\begin{align*}
\Delta \left[    \frac{ \big[ (\g_i (\xi ) - \g_i (\eta))\cdot \N_i ( \xi)  \big] }{ |\gamma_i (\xi)- \gamma_i (\eta)|^{2  } }    \, \right] &  =     \frac{ \Delta \big[ (\g_i (\xi ) - \g_i (\eta))\cdot \N_i ( \xi)  \big] }{ |\gamma_1 (\xi)- \gamma_1 (\eta)|^{2  } }    \,  \\
&\qquad+ \big[ (\g_2 (\xi ) - \g_2(\eta))\cdot \N_2 ( \xi)  \big] \Delta \left[    \frac{ 1}{ |\gamma_i (\xi)- \gamma_i (\eta)|^{2  } }    \, \right]
\end{align*}
and then apply \eqref{eq:building_blocks_difference_b} and \eqref{eq:building_blocks_f} to obtain that
\begin{align}\label{eq:H3_aux2}
    \frac{ \Delta \big[ (\g_i (\xi ) - \g_i (\eta))\cdot \N_i ( \xi)  \big] }{ |\gamma_1 (\xi)- \gamma_1 (\eta)|^{2  } }
 = O(C_M \delta |\xi -\eta|^{-1+\alpha}),
\end{align}
and \eqref{eq:building_blocks_d} and  \eqref{eq:building_blocks_difference_g} to obtain that
\begin{align}\label{eq:H3_aux3}
\big[ (\g_2 (\xi ) - \g_2(\eta))\cdot \N_2 ( \xi)  \big] \Delta \left[    \frac{ 1}{ |\gamma_i (\xi)- \gamma_i (\eta)|^{2  } }    \, \right]
 = O(C_M \delta |\xi -\eta|^{-1+\alpha}).
\end{align}

The bound \eqref{eq:H3_aux1} follows from \eqref{eq:H3_aux2} and \eqref{eq:H3_aux3} by a direct integration.

\noindent
\textbf{Estimates of $  H_{ 4}$:}

\begin{align*}
H_4 & = -  2  \int_{\TT }   \Delta [  \T_i (\eta )g_i (\eta) ]\big( (  \gamma_1  (\xi)- \gamma_1 (\eta)  ) \cdot \T_1 (\xi) \big)    \\
& \qquad \qquad \frac{     \big( (  \gamma_1 (\xi)- \gamma_1  (\eta)  ) \cdot (\T_1 (\xi) -\T_1 (\eta)) \big)  }{|\gamma_i (\xi)- \gamma_i (\eta)|^{4  }} \,   d\eta \\
& -2 \int_{\TT }     \T_2 (\eta )g_2(\eta) \Delta   \bigg[   \frac{    \big( (  \gamma_i  (\xi)- \gamma_i (\eta)  ) \cdot \T_i (\xi) \big)  }{|\gamma_i (\xi)- \gamma_i (\eta)|^{4  }} \\
& \qquad\qquad \times \big( (  \gamma_i (\xi)- \gamma_i  (\eta)  )  \cdot (\T_i (\xi) -\T_i (\eta)) \big) \bigg] \,   d\eta \\
& := H_{41} + H_{42}.
\end{align*}
We first claim that it suffices to consider the term $ H_{42}$ since the estimates for $K_4 $ in Proposition \ref{prop:derivatives_v} together with the simple bound $ |\Delta [  \T_i (\eta )g_i (\eta) ]| \leq C_M \delta$ imply the estimate for $H_{41}$.

For $H_{42}$, applying absolute value, let us further consider the decomposition
\begin{align*}
H_{42} (\xi ) \leq  & C_M (H_{421}(\xi )  + H_{422}(\xi ))
\end{align*}
where $H_{421} $ and $H_{422}$ are respectively
\begin{align*}
H_{421}(\xi ): = \int_{\TT }   \left|\Delta \left[   \frac{    \big( (  \gamma_i  (\xi)- \gamma_i (\eta)  ) \cdot \T_i (\xi) \big)    }{|\gamma_i (\xi)- \gamma_i (\eta)|^{2  }}\right]   \right|     \left| \frac{     \big( (  \gamma_1 (\xi)- \gamma_1  (\eta)  ) \cdot (\T_1 (\xi) -\T_1 (\eta)) \big)  }{|\gamma_1 (\xi)- \gamma_1 (\eta)|^{2  }} \right|  \, d\eta \\
H_{422}(\xi ): =    \int_{\TT }  \left|\frac{    \big( (  \gamma_2  (\xi)- \gamma_2 (\eta)  ) \cdot \T_2 (\xi) \big)    }{|\gamma_2 (\xi)- \gamma_2 (\eta)|^{2  }} \right|   \left| \Delta \left[   \frac{     \big( (  \gamma_i (\xi)- \gamma_i  (\eta)  ) \cdot (\T_i (\xi) -\T_i (\eta)) \big)  }{|\gamma_i (\xi)- \gamma_i (\eta)|^{2  }}\right]  \right|\, d\eta  .
\end{align*}

For $H_{421}$, we use Corollary \ref{corollary:delta_linear} and \eqref{eq:building_blocks_M_e}  to obtain that
\begin{align}
H_{421}(\xi ) \leq  O(C_M \max_i \mathcal{M}\k_i(\xi) \delta |\xi -\eta|^{1-\alpha }),
\end{align}
and Young's inequality implies that $|H_{421}|_{L^p(\TT)} \leq C_M \delta $.

For $H_{422}$,  we first bound the first factor by Corollary \ref{corollary:building_blocks},
\begin{align}\label{eq:H442_aux1}
H_{422}(\xi ) \leq  C_M  \int_{\TT }   |\xi -\eta|^{-1}  \left| \Delta \left[   \frac{     \big( (  \gamma_i (\xi)- \gamma_i  (\eta)  ) \cdot (\T_i (\xi) -\T_i (\eta)) \big)  }{|\gamma_i (\xi)- \gamma_i (\eta)|^{2  }}\right]  \right|\, d\eta  .
\end{align}
By  \eqref{eq:building_blocks_difference_Mb}, \eqref{eq:building_blocks_difference_g} and also \eqref{eq:building_blocks_M_e}, we telescope again and infer the bound
\begin{align}\label{eq:H442_aux2}
\Delta \left[   \frac{     \big( (  \gamma_i (\xi)- \gamma_i  (\eta)  ) \cdot (\T_i (\xi) -\T_i (\eta)) \big)  }{|\gamma_i (\xi)- \gamma_i (\eta)|^{2  }}\right]  =   C_M|\xi-\eta|^{ \alpha}  \Big( \delta   \max_i \mathcal{M} \k_i (\xi ) )  +  \mathcal{M} \Delta[\k_i] (\xi )   \Big).
\end{align}
So combining \eqref{eq:H442_aux1} and \eqref{eq:H442_aux2} we obtain
\begin{align*}
| H_{422} |_{L^p(\TT)}  & \leq  C_M \left[   \int_\TT \Big[  \int_{\TT }     C_M|\xi-\eta|^{-1 +  \alpha}  \Big( \delta   \max_i \mathcal{M} \k_i (\xi )    +  \mathcal{M} \Delta[\k_i] (\xi )  \Big)   \, d\eta \Big]^p \, d\xi \ \right]^\frac{1}{p} \\
\leq  C_M \delta
\end{align*}
where we have used the $L^p$-boundedness of the maximal function and Young's inequality.

Since all the terms in $H_4$ have been estimated, we conclude that
$$
|  H_{ 4}|_{ L^{p} (\TT) }
 \leq  C_M \delta  .
$$

\end{proof}

%%%%%%%%%%%%%%%%%%%%%%%%%%%%%%%%%%%%%%%%%%%%%%%%%%%%%%%%%
\subsection{The fixed-point argument}
%%%%%%%%%%%%%%%%%%%%%%%%%%%%%%%%%%%%%%%%%%%%%%%%%%%%%%%%%
To run the fixed point argument, we will rewrite \eqref{eq:CDE} in integral form:
\begin{equation}\label{cdeint}
\g(\xi,t) = \g_0(\xi) + \int_0^t v(\g(\xi,t' ),t' ) \, d t' .
\end{equation}
Now we prove that the integral equation \eqref{cdeint} is suited for a fixed-point argument.
\begin{proposition}\label{prop:fix_point}
For any  $1<p < \infty $ and $M>1$, there exist $T=T(M , p )>0$ such that the following holds. Given $  \g_0 \in B^{ \frac{M}{2}}_p$, the solution map $S:C([0,T]; B^M_p) \to C([0,T]; B^M_p) $
\begin{align}\label{eq:fixed_point_aux1}
\g  \mapsto S( \g )= \g_0(\xi) + \int_0^t v(\gamma(\xi,t'), t')  d t'
\end{align}
is well-defined and is a contraction on $C([0,T]; B^M_p)$, namely
\[ \|S(\g)\|_{C([0,T]; B^M_p)} \leq \rho \|\g\|_{C([0,T]; B^M_p)}, \]
with the contracting factor $\rho \leq C_M T <1$.
\end{proposition}
\begin{proof}
Let us first show that the map $S$ is $C([0,T]; B^M_p) \to C([0,T]; B^M_p) $. Given $ \g (\xi, t)$, denote by $  \widetilde{\g}(\xi,t) $ the image, $ \widetilde{\g}  =S(  \g)$.
By \eqref{eq:fixed_point_aux1}, for any $t,t_0 \in [0,T]$, we have
\begin{align*}
  | \widetilde{\g}(t) -     \widetilde{\g} (t_0)|_{W^{2,p}(\TT)}   \leq |t- t_0|   \sup_{t' \in [0,T]} \left |  v(\g )    \right |_{W^{2,p}(\TT)}
  \leq |t-t_0|  C(M, p) .
\end{align*}
So $\sup_{t\in[0,T]} |   \widetilde{\g}  |_{W^{2,p}(\TT)} \leq M$ if we choose $T =T(M , p )$ sufficiently small. We also have the continuity of the map in the norm $ | \cdot |_{W^{2,p}(\TT)} $.

It remains to show $  | \widetilde{g}  |_* \leq M $ and $ \Gamma( \widetilde{\g}  ) \leq M $. The bound $  | \widetilde{g}  |_* \leq M $ follows from the embedding $  W^{1,p}(\TT)\subset C(\TT)$ by the following argument. Since
$$
\big|  \widetilde{g} -  g_0\big|  \leq   T  \sup_{t' \in [0,T]} \int_0^t  \left | \p_s v(\gamma(\xi,t'), t') g(\xi,t) \right |_{L^\infty(\TT)}  d t'     \leq  T  C(M, p) \leq \frac{1}{2M}  ,
$$
and $\frac{2}{M} \leq g_0 \leq \frac{M}{2}$, we can choose $T>0$ sufficiently small depending on $M$ and $p$ such that $\frac{1}{M} \leq \widetilde{g} \leq  M  $.

The bound $ \Gamma( \widetilde{\g}) \leq M $ follows from $\Gamma( \k_0 , g_0 ) \leq M/2$ and arguing similarly to Lemma \ref{lemma:completeness}. Indeed, for any $t\in[0,T]$ and any $\xi\neq \eta$, by \eqref{eq:fixed_point_aux1} and the mean value theorem we have
\begin{align*}
|\widetilde{\g}(\xi,t ) -  \widetilde{\g}(\eta,t )| & \geq | {\gamma}_0 (\xi) -  {\gamma}_0 (\eta)| -  \int_0^t \Big|v(\gamma(\xi,t'), t') -v(\gamma(\eta,t'), t')\Big|   d t'\\
&\geq | {\gamma}_0 (\xi) -  {\gamma}_0 (\eta)| -  M T |\xi - \eta |\sup_{t \in[0,T]} \Big|    \p_s v(\gamma(\cdot,t ), t )  \Big|_{L^\infty(\TT)}
\end{align*}
So if $T>0$ is sufficiently small depending on $M$, then by Proposition \ref{prop:derivatives_v} and the assumption $\g_0 \in B^{\frac{M}{2}}_p$
\begin{align*}
\frac{|\xi - \eta|}{ | \widetilde{\g}(\xi,t ) -  \widetilde{\g}(\eta,t )|} & \leq \frac{1}{\frac{|  {\g}_0(\xi ) -   {\g}_0(\eta )| }{|\xi - \eta|} - C_M T      } \\
&\leq \frac{1}{\frac{2 }{M} - C_M T      }  \leq M.
\end{align*}

We have thus shown the solution map  $S:C([0,T]; B^M_p) \to C([0,T]; B^M_p) $ is well-defined.

Finally, we show the contraction property of $S$ on $C([0,T]; B^M_p)$, which will follow directly from the Lipschitz continuity estimates of $\p_s v$ and $\p_s^2 v$ from Proposition \ref{prop:Lip_v} and \ref{prop:Lip_ds2v}. Let $\g_i \in C([0,T]; B^M_p) $ and denote by $ \widetilde{\g}_i =S(\g_i) $ their respective images. Then by \eqref{eq:fixed_point_aux1}
$$
|\Delta \left[ \widetilde{\g}_i \right]   |_{C W^{2,p}} \leq T \sup_t \left |  \Delta \left[  v_i(\cdot, t ) \right]   \right |_{ W^{2,p}(\TT)}
$$
By taking sufficiently small $T>0$, using Proposition \ref{prop:Lip_v} and interpolations, it suffices to show
\begin{equation}\label{eq:fixed_point_aux2}
\sup_t \left | \Delta \left[\p_\xi^2 v_i(\cdot, t )   \right]  \right |_{ L^{ p}(\TT)}  \leq C_M \delta .
\end{equation}
Since $\p_\xi  = g_i^{-1} \p_s   $ on each $\g_i$, we have that
\begin{align*}
\Delta \left[   \p_\xi^2 v_i(\cdot, t )  \right]   = \Delta \left[ \frac{-\dot{g_i}}{g_i^2} \p_s v_i + \frac{\p_s^2 v_i }{g_i^2 }  \right] .
\end{align*}
Since $\sup_t|\Delta [g_i]|_{W^{1,p}(\TT) } \leq C_M \delta $ and $g_i\geq \frac{1}{M}$, \eqref{eq:fixed_point_aux2} holds thanks to Proposition \ref{prop:Lip_ds2v}.

\end{proof}

Proposition \ref{prop:fix_point} yields the  local wellposedness for $W^{2,p} $ patches.

\begin{theorem}\label{thm:localW2p}
Let $1<p<\infty$. For any $ \g_0 \in X_p$, there is a unique local solution $\g$ to \eqref{eq:CDE} in $C([0,T]; X_p)  $ for some $T>0$ satisfying $\gamma(0)=\g_0.$
\end{theorem}

The continuity in time with values in $X_p$ is understood in the sense of $W^{2,p}$ norm.
Moreover, due to $W^{2,p}$ bounds on $v$ from Proposition \ref{prop:derivatives_v} and Proposition \ref{prop:Lip_ds2v}, it follows from \eqref{cdeint} that the solution is actually in $C^{1 }([0,T]; W^{2,p}(\TT))$ as we show in the next section.

%%%%%%%%%%%%%%%%%%%%%%%%%%%%%%%%%%%%%%%%%%%%%%%%%%%%%%%%%%%%%%%%%%%%%%%%%%%%%
\section{Global wellposedness in \texorpdfstring{$W^{2,p}$}{W2p} from  \texorpdfstring{$C^{1,\alpha}$}{C1a}}\label{sec:globalregularity}

In this section, we explain how to use the global $C^{1,\alpha}$ regularity to obtain the global $W^{2,p}$ regularity for the vortex patches.

%%%%%%%%%%%%%%%%%%%%%%%%%%%%%%%%%%%%%%%%%%%%%%%%%%%%%%%%%
\subsection{Improved estimates assuming \texorpdfstring{$C^{1,\alpha}$}{C 1,alpha}}
%%%%%%%%%%%%%%%%%%%%%%%%%%%%%%%%%%%%%%%%%%%%%%%%%%%%%%%%%

The key ingredient we need is the following proposition, which is a direct consequence of the local well-posedness of the CDE in $C^{1,\alpha}$ combined with the global regularity of $C^{1,\alpha}$ patches.

\begin{proposition}\label{prop:C_1_alpha_regularity}
Let $1 < p < \infty $. For any $ \g_0 \in X_p$, the unique solution $\g \in C([0,T];X_p (\TT) )$ to \eqref{eq:CDE}  on some $[0,T]$
can be uniquely continued for all times as a $C^{1,\alpha}$   solution  to \eqref{eq:CDE}  satisfying
\begin{equation}\label{globreg313}
\sup_{t\in[0,T_1]} \left( |g|_* +   |\T|_{C^\alpha(\TT)}+ \Gamma + |\p_s v|_{L^\infty(\TT)} \right)\leq C( \g_0, T_1 ) < \infty.
\end{equation}
\end{proposition}
\begin{proof}
By the embedding $ W^{2,p} (\TT) \subset C^{1,\alpha}(\TT) $, we have that $\g_0 \in C^{1,\alpha}(\TT).$
The result then immediately follows from the estimates for the contour equation and the global  $C^{1,\alpha}$ regularity in \cite[Section 8.3]{MR1867882}.
%the uniqueness is a consequence of the Yudovich theory.

\end{proof}

In the remainder of this section, we consider a fixed vortex patch solution $\gamma \in C( [0,T];X_p)$ with initial data $\g_0$. To show global regularity in $X_p$, it suffices to show $\g \in C( [0,\infty );\dot{W}^{2,p}(\TT))$ as the rest of the information is encoded in the global $C^{1,\alpha}$ regularity.

%By Proposition \ref{prop:C_1_alpha_regularity}, we denote by
The constant $C_0>0$ below is a positive constant depending only on the initial data $\g_0$, $T>0$, and $p>1$ that may change from line to line.
Its existence is ensured by Proposition \ref{prop:C_1_alpha_regularity}.
Using the additional bounds in \eqref{globreg313},
%Proposition \ref{prop:C_1_alpha_regularity},
we have the improved building blocks estimates, analogous to Lemma \ref{lemma:building_blocks}.

\begin{lemma}\label{lemma:improved_building_blocks}
Let $\gamma \in C( [0,T];X_p)$ be a solution of  \eqref{eq:CDE} with initial data $\g_0$. Let $\alpha = 1- \frac{1}{p}$. For any $t\in [0,T]$ and any $\xi, \eta\in \TT$, we have
\begin{subequations}
\begin{align}
\T(\xi)  \cdot \T(\eta) &= 1+ O(C_0 |\xi -\eta |^{2\alpha }) \label{eq:improved_a}\\
\T(\xi)  -\T(\eta) &=  O(C_0 |\xi -\eta |^{ \alpha }) \label{eq:improved_b}\\
\T(\eta) \cdot \N(\xi) & = O( C_0 |\xi -\eta |^{\alpha }) \label{eq:improved_c} \\
(\g(\xi) - \g(\eta)) \cdot \N(\xi) &= O(C_0 |\xi -\eta |^{1+\alpha })  \label{eq:improved_d} \\
(\T(\xi) - \T(\eta)) \cdot \T(\xi) &= O( C_0 |\xi -\eta |^{2\alpha }) \label{eq:improved_e} \\
|\g(\xi) - \g(\eta)|^{-1} &= O( C_0 |\xi -\eta|^{-1}) \label{eq:improved_f}
\end{align}
\end{subequations}
and  for any $ \zeta \in \TT$ lying between $\eta$ and $\xi$, the maximal estimates
\begin{subequations}
\begin{align}
\T(\eta) \cdot \N(\xi) & =  O( C_0 \mathcal{M}\k ( \zeta ) |\xi -\eta | ) \label{eq:improved_M_a} \\
\T(\eta) \cdot \T(\xi) & = 1+ O( C_0 \mathcal{M}\k ( \zeta ) |\xi -\eta |^{1+\alpha } ) \label{eq:improved_M_b} \\
(\g(\xi) - \g(\eta)) \cdot \N(\xi) &= O( C_0 \mathcal{M}\k ( \zeta ) |\xi -\eta |^{2})  \label{eq:improved_M_c} \\
  \T(\xi)  \cdot \left[\T(\xi) -\T(\eta)  \right] &= O( C_0 \mathcal{M}\k ( \zeta ) |\xi -\eta |^{1+\alpha}) \label{eq:improved_M_d} \\
\left[(\g(\xi) - \g(\eta) \right]\cdot \left[\T(\xi) -\T(\eta)  \right] &= O( C_0 \mathcal{M}\k (\zeta ) |\xi -\eta |^{2+\alpha}) \label{eq:improved_M_e}.
\end{align}

\end{subequations}

Finally,
\begin{align}\label{eq:improved_linear}
(\g(\xi) - \g(\eta)) \cdot \T(\xi) &= g(\zeta)(\xi -\eta )+ O(C_0 |\xi -\eta |^{1+ \alpha }).
\end{align}
\end{lemma}

The proof of Lemma \ref{lemma:improved_building_blocks} follows the same argument as Lemma \ref{lemma:building_blocks}: we just need to replace the constant $C_M$ there by $C_0$.

We are ready to prove the improved Sobolev estimate for $v(\g)$.

\begin{proposition}\label{prop:Global_Sobolev}
Let $1< p < \infty $. If $ \g \in C( [0,T];X_p)$ is a solution of \eqref{eq:CDE}, then
\begin{align}\label{eq:Sobolev_K}
|\p_s^2 v  |_{L^p(\TT)} \leq C_0 | \gamma  |_{W^{2,p}(\TT)}
\end{align}
for all $0 \leq t \leq T.$
\end{proposition}
\begin{proof}
The key is the linear appearance of $| \gamma  |_{W^{2,p}(\TT)}$ in \eqref{eq:Sobolev_K}.
We use the same decomposition as in Proposition \ref{prop:derivatives_v},
\begin{align*}
		 \p^2_s   v(\xi )  & =   \sum_{1\leq i\leq 4}  K_i(\xi )
\end{align*}
with
\begin{align*}
K_1 &  =  -P.V.  \int_{\TT }   \k( \eta )\N ( \eta )    \frac{  (\gamma( \xi )- \gamma(  \eta )) \cdot \T( \xi )}{  |\gamma( \xi )- \gamma(  \eta )|^{2  }}  g(\eta )\, d \eta   \\
K_2 &= \int_{\TT } \T( \eta )  \frac{  (\T( \xi ) - \T(\eta  ))\cdot \T( \xi )   }{ |\gamma(\xi )- \gamma( \eta )|^{2 } }    g(\eta ) \, d\eta\\
K_3 &= -\int_{\TT } \T( \eta )  \frac{  \k( \xi ) (\gamma(\xi )- \gamma(\eta ))\cdot \N( \xi )  \big] }{ |\gamma(\xi )- \gamma( \eta )|^{2 } }    g(\eta ) \, d\eta\\
K_4 &= - 2\int_{\TT } \T(\eta ) \frac{    \big( (  \gamma(\xi)- \gamma( \eta )  ) \cdot \T(\xi) \big) \big( (  \gamma(\xi)- \gamma(\eta )  ) \cdot (\T(\xi) -\T(\eta )) \big)  }{|\gamma(\xi)- \gamma(\eta )|^{4  }}  g(\eta )\, d\eta.
\end{align*}

The proof goes almost identically to Proposition \ref{prop:derivatives_v}, so we only sketch the details for $K_1$.
\begin{align*}
 | K_1 	 |_{L^p(\TT )}   &\leq  \left[   \int_{ \TT }  \left|   \int_{\TT}   \k(\eta  )\N ( \eta)   \frac{  ( \xi  - \eta  )  }{  g(\eta)| \xi - \eta  |^{2   }}  g(\eta ) \, d \eta    \right|^p  \, d\xi  \right]^\frac{1}{p} \\
 & \quad + O\left( C_0  \int_{\TT}  \left|   \int_{ \TT }  | \k(\eta )|  | \xi-   \eta  |^{ \alpha -1   }   \, d \eta    \right|^p  \, d\xi    \right)^\frac{1}{p}.
\end{align*}
By the $L^p$-boundedness of the Hilbert transform and Young's inequality imply that $ | K_1 	 |_{L^p(\TT )}    \leq C_0 |\k |_{L^p(\TT )}$.

Once we have $ | \p^2_s   v 	 |_{L^p(\TT )}    \leq C_0 |\k |_{L^p(\TT )} $, the conclusion follows, since by Proposition \ref{prop:C_1_alpha_regularity} we have uniform control of $|g|_*$
and thus $|\k |_{L^p(\TT )} \leq C_0 | \gamma  |_{W^{2,p}(\TT)} $ for all $t\in [0,T]$.

\end{proof}

%%%%%%%%%%%%%%%%%%%%%%%%%%%%%%%%%%%%%%%%%%%%%%%%%%%%%%%%%
\subsection{Global \texorpdfstring{$W^{2,p}$}{W2p} regularity}
%%%%%%%%%%%%%%%%%%%%%%%%%%%%%%%%%%%%%%%%%%%%%%%%%%%%%%%%%

With the previous proposition, we obtain the global $W^{2,p}$ regularity for vortex patches.
\begin{theorem}\label{thm:global_sobolev}
Let $1 < p < \infty $. For any $   \g_0   \in X_p$, there is a unique global solution $ \g $ to \eqref{eq:CDE} in $C([0,\infty); X_p) \cap C^{1 }([0,\infty); W^{2,p}(\TT))$ satisfying $\g(0) =\g_0.$
\end{theorem}
\begin{proof}
Let  $\gamma$ be the unique local solution given by Theorem \ref{thm:localW2p}. Since $ \gamma$  is a global $C^{1,\alpha}$ patch solution for $\alpha = 1 - \frac{1}{p}$
satisfying \eqref{globreg313} for all $t \geq 0$, $ \gamma$ ceases to be a $W^{2,p}$ patch solution if and only if $ |\gamma(t)|_{W^{2,p}(\TT) }     $ blows up at some finite time $T>0$.

We show that  $ \sup_{t\in [0,T)} |\gamma(t)|_{W^{2,p}  (\TT)} \leq C(\g_0,T) <\infty $, arriving at a contradiction due to local $W^{2,p}$ wellposedness. For any $t<T$, we have
\begin{align}\label{eq:global_sobolev_aux1}
 |\gamma(t)|_{W^{2,p}(\TT) } \leq |\g_0 |_{W^{2,p}(\TT) } + \int_0^t |v(\g(\cdot,t),t)|_{W^{2,p} (\TT) }.
\end{align}
By Proposition \ref{prop:Global_Sobolev} and Proposition \ref{prop:C_1_alpha_regularity},
\begin{align}
|v(\g(\cdot,t),t)|_{W^{2,p} (\TT) } & \leq |v|_{L^\infty (\TT)}+ | \nabla v |_{L^{\infty } (\TT) } + | \p_{\xi}^2 v |_{L^{ p} (\TT) } \nonumber \\
& \leq C(\g_0,T)  + C(\g_0,T) | \g (\cdot, t) |_{W^{2, p} (\TT) }\label{eq:global_sobolev_aux2}.
\end{align}
By Gronwall's inequality, it then follow from \eqref{eq:global_sobolev_aux1} and \eqref{eq:global_sobolev_aux2} that
$$
 |\gamma(t)|_{W^{2,p}(\TT) } \leq  C  e^{C t},
$$
and we conclude that $\g$ remains a $W^{2,p}$ patch solution up to $T$.

The above augment shows that $ \g  \in  C([0,\infty  );W^{2,p})$. We now show the regularity $C^{1 }([0,\infty); W^{2,p}(\TT))$ using the integral formulation of $\p_t \g = v(\g)$. By Proposition \ref{prop:Lip_ds2v}, for any $t_1,t_2 \in[0,\infty  )$, we have
$$
 |v(\g(\cdot, t_1),t_1) - v(\g(\cdot, t_2),t_2 )|_{W^{2,p}(\TT)} \lesssim  |\g(\cdot, t_1 ) - \g(\cdot, t_2 ) | \to 0\quad \text{ as $|t_1 -t_2| \to 0$.}
$$
So $v (\g )\in  C ([0,\infty  );W^{2,p})$ and this in turn implies  the regularity $\g \in C^1([0,\infty  );W^{2,p})$.

\end{proof}

%%%%%%%%%%%%%%%%%%%%%%%%%%%%%%%%%%%%%%%%%%%%%%%%%%%%%%%%%%%%%%%%%%%%%%%%%%%%%
\section{The curvature equation and equivalent formulation of patch evolution}\label{sec:curvatureeq}
%%%%%%%%%%%%%%%%%%%%%%%%%%%%%%%%%%%%%%%%%%%%%%%%%%%%%%%%%%%%%%%%%%%%%%%%%%%%%

In this section, we derive the equations for the evolution of geometric quantities such as the tangent vector, arc-length metric, and curvature of the vortex patches according to \eqref{eq:CDE}.
To this end, we first sketch the derivation of these equations for general velocity fields, which is a classical topic in differential geometry, see e.g. \cite{MR1145840,MR1824511}. Then we justify the computation for $W^{2,p}$ vortex patches.

%%%%%%%%%%%%%%%%%%%%%%%%%%%%%%%%%%%%%%%%%%%%%%%%%%%%%%%%%
\subsection{Derivation of the curvature equation}
%%%%%%%%%%%%%%%%%%%%%%%%%%%%%%%%%%%%%%%%%%%%%%%%%%%%%%%%%
In our derivations in this section, we assume that all occurring objects are sufficiently regular. Later we will make the regularity assumptions more precise in Proposition \ref{prop:equivalence}.

\subsubsection{Evolution of the arc-length}\label{subsec:arclength}

Since the evolution of the curve $\gamma$ is enabled by the velocity $v$, we project the vector field $v$ in \eqref{eq:gamma_curve} to the tangent and normal vectors on the curve $\g$ so that
\begin{align*}
\p_t \gamma & =(  v \cdot \T) \T + (  v \cdot \N) \N\\
&=v_{\tau}\T  + v_{n} \N,
\end{align*}
where $ v_{\tau}$ and $  v_{n} $ are respectively the tangent and normal component of the velocity on the curve $\g$.

To simply the derivation, we denote by $\theta$ the angle between $x_1$-axis and the tangent, namely $\T = ( \cos(\theta ) , \sin(\theta) )$ and $\N = ( \sin(\theta) , -\cos(\theta ) )$. Then the signed curvature $\k$ can be computed by
\begin{equation}\label{eq:def_theta}
\k = - \partial_s \T \cdot \N = \p_s \theta.
\end{equation}

Next, we derive the evolution of the arc-length $g:= |\dot{\gamma}|$. Using \eqref{Neq1222} we obtain
\begin{align*}
\p_t g^2  & = 2 \dot{\gamma} \cdot \p_t \dot{\gamma}= 2 g \T  \cdot  \p_\xi   (v_{\tau} \T + v_n \N)
=  2  g    ( \dot{v}_{\tau}   + v_n \k g).
\end{align*}
Since by definition and \eqref{Neq1222},
$$
\dot{v}_{\tau} = \dot{v} \cdot \T - v_n \k g,
$$
we have
\begin{align*}
\p_t g^2   %& = 2 \dot{\gamma} \cdot \frac{\p \dot{\gamma} }{\p t}= 2 (g \T ) \cdot  \p   (v_{\tau} \T + v_n \N) \\
=  2  g   \dot{v} \cdot \T = 2g^2 \p_s v  \cdot \T .
\end{align*}
Then we have the evolution of the metric $g:$
\begin{align}\label{eq:evo_metric}
\p_t g  = g \p_s v \cdot \T.
\end{align}

\subsubsection{Evolution of the tangent, normal and curvature}

To derive the curvature equation, we need to make use of the commutation identity:
\begin{align}\label{eq:comm_id}
\p_t \p_s = -\left( \p_s  v \cdot \T \right) \p_s  + \p_s \p_t,
\end{align}
which follows from \eqref{eq:evo_metric} and $\p_s = g^{-1} \p_\xi$ by a routine application of product rule.

Using  \eqref{eq:comm_id}
we obtain the evolution of the tangent $\T$,
\begin{align}\label{eq:evo_tangent}
\p_t \T = -\left( \p_s  v \cdot \T \right) \T + \p_s  v =
\left( \p_s v \cdot \N \right) \N.
\end{align}

Since
\[ 0 = \partial_t (\N \cdot \T) = \partial_t \N \cdot \T + \N \cdot \partial_t \T, \]
for the normal $\N$, we have by \eqref{eq:evo_tangent}
\begin{align}\label{eq:evo_normal}
\p_t \N =  -\left( \p_s v \cdot \N \right)  \T.
\end{align}

Finally, using \eqref{eq:def_theta} and \eqref{eq:comm_id}, we have
\begin{align}
\p_t \k & = \p_t \p_s \theta  = -\left( \p_s v \cdot \T \right) \p_s \theta  + \p_s \left(     - \p_s v \cdot \N \right) \label{aux12221}
\end{align}
which yields the curvature evolution equation
\begin{equation}\label{eq:evo_curvature}
\p_t \k   = -2 \k \p_s v  \cdot \T       - \p^2_s v \cdot \N  .
\end{equation}

%%%%%%%%%%%%%%%%%%%%%%%%%%%%%%%%%%%%%%%%%%%%%%%%%%%%%%%%%
\subsection{The curvature equation of vortex patches}
%%%%%%%%%%%%%%%%%%%%%%%%%%%%%%%%%%%%%%%%%%%%%%%%%%%%%%%%%

The next lemma can be used to simplify the right-hand side of \eqref{eq:evo_curvature}.
\begin{lemma}\label{lemma:simplification}
For $\g \in X_p $, the following identity holds:
\begin{equation}\label{aux1522}
\p_s v   \cdot \T   =  -\int_{\TT  } \T( \eta  )  \cdot \N( \xi )\frac{ (\gamma( \xi  )- \gamma( \eta  ))\cdot \N( \xi ) }{ |\gamma( \xi  )- \gamma( \eta )|^{2  } }    g(\eta)\, d \eta .
\end{equation}
\end{lemma}
\begin{proof}
Denote by $\Omega$ the domain with $\g$ as the boundary. In this proof, we write everything in arc-length parametrization. Recall \eqref{eq:d_v_xi} and note that $\T(s) \cdot \T(s') = \N(s) \cdot \N(s').$
We apply the divergence theorem (not difficult to justify rigorously) to obtain
\begin{align}
\p_s v   \cdot \T  &= P.V. \int_{\gamma } \N (s) \cdot \N ( s' )  \frac{   (\gamma(s)-\gamma( s' )) \cdot \T(s) }{  |\gamma(s) - \gamma(s' )|^{2 }} \,d s' \nonumber \\
 & =  P.V. \int_{\Omega} \D_y \left( \N (s)  \frac{   (\gamma(s)- y ) \cdot \T(s) }{  |\gamma(s) - y|^{2 }}  \right) \,dy \nonumber \\
&  = -2 P.V. \int_{\Omega}     \frac{ (\gamma(s)- y ) \cdot \N(s)   (\gamma(s)- y ) \cdot \T(s) }{  |\gamma(s) - y|^{4  }} \label{dvT2d}   \,dy .
\end{align}

On the other hand, the right-hand side of \eqref{aux1522} is equal to
\begin{align*}
    \int_{\gamma  } \N( s' )  \cdot \T(s )\frac{ (\gamma(s )- \gamma( s'  ))\cdot \N(s )  }{ |\gamma( s )- \gamma( s' )|^{2  } }    \, d s' ,
\end{align*}
which after integration by parts is equal to
\begin{align*}
-2 P.V. \int_{\Omega}     \frac{ (\gamma(s)- y ) \cdot \N(s)   (\gamma(s)- y ) \cdot \T(s) }{  |\gamma(s) - y|^{4 }}    \,dy .
\end{align*}
\end{proof}

Based on the previous derivations \eqref{eq:evo_metric} and \eqref{eq:evo_curvature} and the integrals for derivatives $\p_s v $ and $\p_s^2 v$ in Proposition \ref{prop:derivatives_v}, we introduce the system for arc-length $g$ and curvature $\k$ evolution equations: % set in terms of the arc-length metric $g$ and the curvature $\k$,
\begin{equation}\label{eq:curvature_and_metric}
\begin{cases}
\p_t \k   = K(g, \k ) &\\
\p_t g   = G(g, \k )&
\end{cases}
\end{equation}
where $ K(g, \k )$ and $G(g, \k)$ are the nonlinear functionals given by (in part due to Lemma \ref{lemma:simplification})
\begin{equation}\label{eq:def_K_term}
\begin{aligned}
K(g, \k ) =      -   \k(\xi )\p_s v \cdot \T (\xi ) + P.V. \int_{\TT }   \k(\eta)\N (\eta)  \cdot \N(\xi) \frac{  (\gamma(\xi)- \gamma(\eta)) \cdot \T(\xi)}{  |\gamma(\xi)- \gamma(\eta)|^{2   }} \, g(\eta) d\eta\\
	   -\int_{\TT } \T(\eta)  \cdot \N(\xi)\frac{ \big[ (\T(\xi ) - \T(\eta))\cdot \T( \xi)  \big] }{ |\gamma(\xi)- \gamma(\eta)|^{2  } }    \, g(\eta) d\eta\\
		   +  2  \int_{\TT } \T(\eta)\cdot \N(\xi) \frac{    \big( (  \gamma(\xi)- \gamma(\eta)  ) \cdot \T(\xi) \big) \big( (  \gamma(\xi)- \gamma(\eta)  ) \cdot (\T(\xi) -\T(\eta)) \big)  }{|\gamma(\xi)- \gamma(\eta)|^{4  }} \, g(\eta) d\eta
\end{aligned}
\end{equation}
and
\begin{equation}\label{eq:def_G_term}
G(g, \k ) = g(\xi) \p_s v \cdot \T(\xi).
\end{equation}

We remark that the nonlinear functionals $K(g,\k)$ and $G(g, \k)$ are well-defined in the sense that their values can be computed using only arc-length $g$ and curvature $\k$ of the curve $\gamma.$
%It would be of interest to show the global regularity of vortex patches with an intrinsic proof, however, at the moment, we are not able to prove wellposedness directly using the system \eqref{eq:curvature_and_metric}.
In particular, they do not depend on the orientation or location of the patch.

%%%%%%%%%%%%%%%%%%%%%%%%%%%%%%%%%%%%%%%%%%%%%%%%%%%%%%%%%
\subsection{Equivalence of the arc-length/curvature system and vortex patch evolution}
%%%%%%%%%%%%%%%%%%%%%%%%%%%%%%%%%%%%%%%%%%%%%%%%%%%%%%%%%

In this subsection, we show that a solution $ (g, \k ) \in C([0,T];W^{1,p}(\TT)\times L^p(\TT))$ of  \eqref{eq:curvature_and_metric} corresponds to the unique  Euler patch solution $\g  \in C([0,T];W^{2,p}(\TT) ) $ of \eqref{eq:CDE} and vice versa.

\begin{proposition}\label{prop:equivalence}
Let $\g_0 \in  W^{2,p}(\TT) $ be a proper parametrization of a simple closed curve and let $ (g_0, \k_0 ) \in  W^{1,p}(\TT)\times L^p(\TT))$ be the corresponding  arc-length and curvature. Then the following statements are true.

\begin{enumerate}
\item  If $\g \in C([0,T];X_{p} )$ is a solution of \eqref{eq:CDE} with initial data $\g_0$, then its arc-length $g$ and curvature $\k$ of $\g $ must satisfy the equations \eqref{eq:curvature_and_metric} with initial data $g_0$ and $\k_0$.\label{item:gamma_CDE}

\item If $ (g, \k ) \in C([0,T];W^{1,p}(\TT)\times L^p(\TT))$ as the arc-length and curvature defines a simple closed curve and satisfies the equations \eqref{eq:curvature_and_metric} with initial data $g_0$ and $\k_0$, then there exist a solution $\g \in C([0,T];W^{2,p}(\TT))$ of \eqref{eq:CDE} with initial data $\g_0$ such that $(g, \k )$ is the arc-length and curvature of $\g$.\label{item:gk_curvature}
\end{enumerate}
\end{proposition}
Before proving the proposition, we remark that the condition of $(g, \k)$ defining a closed curve in at each time $t\in [0,T]$ is equivalent to the conditions
$$
\int_{\TT} (\cos( \theta(\xi)) ,  \sin( \theta(\xi))) g (\xi) \, d \xi  =0
$$ with $\theta(\xi) = \int_0^\xi \k(\eta) g(\eta) \, d\eta $ and $\int_\TT \k(\xi) g (\xi)\, d \xi = 2\pi$ for every $t\in [0,T]$.

\begin{proof}

\noindent
\eqref{item:gamma_CDE}  $\Rightarrow$ \eqref{item:gk_curvature}:

Since  $\g \in C([0,T];X_{p})$, it satisfies the assumptions of Proposition \ref{prop:derivatives_v} on $[0,T]$ for a sufficiently large $M$. By Theorem \ref{thm:global_sobolev} we also have $\g \in C^1([0,T];W^{2,p})$.

This $C^1([0,T];W^{2,p})$ regularity of $\g$ implies the regularity $g   \in C^1([0,T];W^{1,p})$, $\T    \in C^1([0,T];W^{1,p})$ and $\k \in C^1([0,T];L^{p})$ thanks to the fact that $g =|\dot{\g}| >0$ uniformly on $\TT\times [0,T]$. Such regularity in turn allows us to derive the arc-length equation \eqref{eq:evo_metric} using the (now rigorous) computation in \eqref{subsec:arclength}. The derivations of the evolution of tangent \eqref{eq:evo_tangent} and curvature \eqref{aux12221} are also justified following \eqref{eq:comm_id}.

\noindent\eqref{item:gk_curvature} $\Rightarrow$ \eqref{item:gamma_CDE}:

Let us define the tangent vectors $\T(\xi,t)$
\begin{equation}\label{tangent1522}
\T(\xi,t)=(\cos \theta(\xi,t), \sin \theta(\xi,t)), \,\,\, %by computing the angle
\theta (\xi, t) = \theta(0,t) + \int_{0}^{\xi } \k(\eta,t) g(\eta,t) \, d\eta,
\end{equation}
and the curve
\begin{equation}\label{eq:modulated_gamma}
\gamma(\xi,t) = \gamma(0,t)+ \int_{0}^{\xi } \T(\eta,t) g(\eta,t) \, d\eta,
\end{equation}
with $\theta(0,t) $ and $\gamma(0,t)$ to be determined. By the assumptions on $(g,\k)$, $\g $ is a well-defined simple closed curve, and $\g  \in C([0,T];W^{2,p}(\TT))$ provided $\theta(0,t)  ,\gamma(0,t) \in C([0,T])$. Therefore, we only need to show how to determine  $\theta(0,t) $ and $\gamma(0,t)$ so that $\g$ solves \eqref{eq:CDE}.

Thanks to the invariance of translation and rotation in the definition of nonlinear functionals $K(g,\k)$ and $G(g,\k)$, we can use \eqref{tangent1522} and \eqref{eq:modulated_gamma} to evaluate the equations \eqref{eq:def_K_term} and \eqref{eq:def_G_term}. In particular, we know that the quantities $\p_s^2 v\cdot \N$, $\p_s v\cdot \T$  are known functions on $\TT\times[0,T]$ given by $g$ and $\k$, but $\p_s v$ and $v$ are to be determined along with $\T$, $\N$, and $\g$.
%we have that $\int_{\TT} g \k =2\pi$ for all $t\in [0,T]$. Indeed, observe that from \eqref{eq:evo_curvature} and \eqref{eq:evo_metric} it follows that
%\begin{equation}\label{gk1522}
%\p_t(\k g) = -\p_s (\p_s v \cdot \N) g = -\p_s \left( \int_{\TT } \N (\xi) \cdot \T ( \eta )  \frac{   (\gamma(\xi)-\gamma( \eta )) \cdot \T(\xi) }{  |\gamma( \xi ) - \gamma(\eta )|^{2  }} \, g(\eta)d \eta   \right)g(\xi);
%\end{equation}
%recall that the expression on the right is independent of $\theta(0,t)$ and $\g(0,t).$
%As a consequence of this equality, we must have $\int_{\TT} \k g \,d\xi = 2\pi$ for all $t \in [0,T],$
%since
%\begin{equation}\label{intinv5122}
%\p_t \int_{\TT} \k g\,d\xi = - \int_{\TT} \frac{\p_\xi}{g} (\p_s v \cdot \N) g \,d\xi =0.
%\end{equation}
%This, ensures that the curve $\g(\xi,t)$ we recover in \eqref{eq:modulated_gamma} is closed.

Differentiating \eqref{tangent1522} in time, we find that
\begin{equation}\label{aux11522}
 \p_t \T(\xi,t) = -\left( \int_0^\xi \p_t (\k g)\,d\eta +\p_t \theta(0,t) \right) \T^\perp (\xi,t).
\end{equation}
Since  $(g,\k)$ is a solution of \eqref{eq:curvature_and_metric}
\begin{equation}\label{gk1522}
\p_t(\k g) = -\p_s (\p_s v \cdot \N) g .
%= -\p_s \left( \int_{\TT } \N (\xi) \cdot \T ( \eta )  \frac{   (\gamma(\xi)-\gamma( \eta )) \cdot \T(\xi) }{  |\gamma( \xi ) - \gamma(\eta )|^{2  }} \, g(\eta)d \eta   \right)g(\xi);
\end{equation}
Using \eqref{gk1522}, we find that the expression in the bracket in \eqref{aux11522} is equal to
\[ -(\p_s v \cdot \N) (\xi,t) + (\p_s v \cdot \N) (0,t) +\p_t \theta(0,t). \]
Let us now define $\theta(0,t)$ to be the solution of
\begin{equation}\label{eq:theta_ODE}
\begin{cases}
 \p_t \theta(0,t) = - (\p_s v \cdot \N) (0,t),& \\
\theta(0,0)= \theta_0(0),
\end{cases}
\end{equation}
where $\theta_0 (\xi)$ the initial angle is computed using the initial data $\g_0$. Note that the right-hand side of \eqref{eq:theta_ODE} is a known function in terms of $g$ and $\k$  since   $\p_s v \cdot \N$ only involves relative angles and relative distance by \eqref{tangent1522} and \eqref{eq:modulated_gamma}.
Then by \eqref{tangent1522} and \eqref{aux11522} we arrive at
\begin{equation}\label{aux21522}
 \p_t \T(\xi,t) =  -(\p_s v \cdot \N) (\xi,t)  \T^\perp (\xi,t).
\end{equation}
Now that we have determined the tangent $ \T$ and normal $\N$, we have
\begin{align}\label{eq:modulated_gamma_dv}
\p_\xi v = g (\p_s v \cdot \T  ) \T +  g (\p_s v \cdot \N  ) \N
\end{align}
is also determined. Furthermore, since $v = \int_{\TT} \T(\eta) g(\eta) \ln|\g (\xi) - \g(\eta)| \, d\eta $ does not depend on $\g(0,t)$ in \eqref{eq:modulated_gamma}, we have fully determined the velocity $v$ on $\g$ as a function $\TT \times [0,T] \to \RR^2$.

Next, we will determine the curve $\g$ by solving for $ \g(0,t)$, and show that $\g$ is a solution of \eqref{eq:CDE}. Let us differentiate \eqref{eq:modulated_gamma} in time. We obtain
\[ \p_t \g(\xi,t) = \p_t \gamma(0,t) + \int_0^\xi \p_t (\T(\eta,t) g(\eta,t))\,d\eta. \]
Using \eqref{aux21522}, \eqref{eq:evo_metric}, and $\N= -\T^\perp,$ we find that the expression under the integral is equal to
\[ -(\p_s v \cdot N) \T^\perp g + \T g (\p_s v \cdot \T)= (\p_\xi v \cdot \T)\T + (\p_\xi v \cdot \N) \N = \p_\xi v. \]
Therefore,  $\g$ satisfies the equation
\[ \p_t \g(\xi,t) = v(\g(\xi,t),t) - v(\g(0,t),t) + \p_t \gamma(0,t). \]
Define $\g(0,t)$ by solving $\p_t \g(0,t) = v(\g(0,t),t),$
then we see that $\gamma(\xi,t)$ satisfies the Euler patch equation \eqref{eq:CDE}.

\end{proof}

\section{Illposedness of \texorpdfstring{$C^2$}{C2} patches}\label{sec:illposedness}

In this section, we will prove the main illposedness results based on the curvature equation. The proof goes in several steps.
\begin{enumerate}
    \item The first step is to rewrite the curvature equation \eqref{eq:evo_curvature} into \eqref{eq:evo_curvature_illposedness} by isolating the linear dispersion effect.

    \item We then show that in the $W^{2,p}$ setting, namely when $\g \in X_p$ for $p$ large, the linear term in \eqref{eq:evo_curvature_illposedness} is the dominant term.

    \item The last step is to use Duhamel's principle and pick initial data to show the evolution group induces the norm inflation $|\k|_{L^p (\TT)} \to \infty $ as $p \to \infty$ over a fixed time interval $[0,T]$ that is independent of $p$.
\end{enumerate}

\subsection{Reformulation of the curvature equation}
We start by recasting the curvature equation. Denote by $\mathcal{H}$ the Hilbert transform on $\TT$, namely
$$
\mathcal{H} f(\xi ) = \frac{1}{2\pi} P.V. \int_{\TT}  f(\eta )\cot\left(\frac{\xi - \eta }{2}\right) \, d\eta.
$$
Note that the kernel is obtained by periodizing the Hilbert transform on the real line,
\begin{align}
\frac{1}{2\pi} \cot \left(\frac{x}{2} \right)  = \frac{1}{\pi x}+ \frac{1}{\pi  }\sum_{n \geq 1} \frac{1}{x+ 2\pi n}+ \frac{1}{x- 2\pi n}.
\end{align}
It is classical~\cite{MR69310} (see also e.g. \cite{MR3052498} for an easy reference)
that the periodic Hilbert transform is bounded on $L^p(\TT)$ for $1<p <\infty$ and on $C^{\alpha}(\TT)$ for $0<\alpha<1$. Here we recall the periodic H\"older space $C^{\alpha}(\TT)$ is equipped with the norm
\begin{equation}\label{eq:def_holder_norm}
| f |_{C^{\alpha} (\TT)} = |f|_{L^\infty(\TT)} + \sup_{\xi \neq \eta} \frac{|f(\xi)  - f(\eta) |}{|\xi -\eta |}.
\end{equation}

We will now further analyze the curvature equation \eqref{eq:curvature_and_metric}, \eqref{eq:def_K_term} to obtain the following

\begin{theorem}\label{curveq9122}
Suppose that $ \g  \in C([0,T]; B^M_p),$ $p>2,$ and set $\alpha = 1 -\frac1p,$ $\beta = 1- \frac2p.$
 Then on the time interval $[0,T]$, the curvature $\k$ satisfies
\begin{align}\label{eq:evo_curvature_illposedness}
\frac{ \p \k }{ \p t}  = a(\xi,t)\k +  \pi\mathcal{H} (\k)(\xi,t) + F(\xi,t),
\end{align}
where $a \in C^\alpha(\TT)$ and $F \in C^\beta(\TT)$ uniformly in $t \in [0,T]$ with norms depending on $T$ and $M$.
\end{theorem}

Here $ a = -      \p_s    v   \cdot \T (\xi,t )$ and we will split $F := F_L + F_N$ as shown below. %is a sum of the linear error term $F_L$ and the nonlinear error term $F_N$.
The two error terms $F_L$ and $F_N$ are defined respectively by
\begin{equation}\label{eq:def_linear_error}
\begin{aligned}
F_L(\xi ) := -\pi\mathcal{H} (\k) + P.V. \int_{ \TT }   \k(\eta)\N (\eta)  \cdot \N(\xi) \frac{  (\gamma(\xi)- \gamma(\eta)) \cdot \T(\xi)}{  |\gamma(\xi)- \gamma(\eta)|^{2   }} \, g(\eta) d\eta
\end{aligned}
\end{equation}
and
\begin{multline}\label{eq:def_nonlinear_error}
F_N (\xi )  =  -\int_{\TT } \T(\eta)  \cdot \N(\xi)\frac{ \big[ (\T(\xi ) - \T(\eta))\cdot \T( \xi)   \big] }{ |\gamma(\xi)- \gamma(\eta)|^{2  } }    \, g(\eta) d\eta\\
		  +  2  \int_{\TT }  \T(\eta )\cdot \N(\xi) \big( (  \gamma(\xi)- \gamma(\eta)  ) \cdot \T(\xi) \big) \times \\ \frac{    \big( (  \gamma(\xi)- \gamma(\eta)  ) \cdot (\T(\xi) -\T(\eta)) \big)  }{|\gamma(\xi)- \gamma(\eta)|^{4  }} \, g(\eta) d\eta.
\end{multline}

The driving mechanism of $C^{2}$ illposedness is the dispersion of the Hilbert transform in \eqref{eq:evo_curvature_illposedness}. In fact, since $\mathcal{H}^2 = -\Id $, one has the following formula
\begin{equation}\label{eq:etH}
e^{t \pi\mathcal{H}}  = \sum_n \frac{(t \pi\mathcal{H})^n}{n !}  =  \cos(\pi t ) \Id  +\sin(\pi t ) \mathcal{H} .
\end{equation}
To exploit the above dispersion of the Hilbert transform, we first need to establish suitable estimates for $a$ and $F$. Then by a simple application of the Duhamel formula (namely \eqref{eq:Duhamel_for_K} below), we can establish norm inflation for a suitable new variable that implies $C^2$ illposedness.

%%%%%%%%%%%%%%%%%%%%%%%%%%%%%%%%%%%%%%%%%%%%%%%%%%%%%%%%%
\subsection{H\"older estimates of coefficients }
%%%%%%%%%%%%%%%%%%%%%%%%%%%%%%%%%%%%%%%%%%%%%%%%%%%%%%%%%
We first show that $ F$ is H\"older if $p> 2$ with a H\"older exponent $ \beta =1 -\frac{2}{p}   $.

Before we proceed, recall that $C_M$ denotes a positive constant depending only on $M, p$ that may change from line to line and the big O notation $X = O(Y)$ for a quantity $X$ such that $|X| \leq C Y$ for some absolute constant $C>0$.

\begin{proposition}\label{prop:F_L_holder}
Let  $ 2 < p\leq \infty $ and $  \beta = 1- \frac{2}{p}$. If $ \g  \in C([0,T]; B^M_p)$ for some $M>1$, then the error term $F_L$ defined by \eqref{eq:def_linear_error} satisfies the estimate
$$
| F_L |_{C^\beta (\TT)} \leq C_M.
$$

\end{proposition}
\begin{proof}

Let us write $F_L$ in abbreviation
\begin{equation}\label{eq:F_L_1}
F_L =  P.V. \int_{\TT} \k(\eta) Q_{L}(\xi,\eta) \, d\eta + \int_{\TT} \k(\eta)Q_S (\xi,\eta), d\eta.
\end{equation}
where the first term with the kernel $Q_L$ is the main term
\begin{equation}\label{eq:F_L_1_aux1}
 Q_{L}(\xi,\eta)= -\frac{1}{\xi -\eta} + \N (\eta)  \cdot \N(\xi) \frac{  (\gamma(\xi)- \gamma(\eta)) \cdot \T(\xi)}{  |\gamma(\xi)- \gamma(\eta)|^{2   }}   g(\eta)
\end{equation}
while the second term has a smoothing kernel $Q_S(\xi, \eta) = -\sum_{n\geq 1} (\xi -\eta +2\pi n)^{-1}-(\xi -\eta -2\pi n)^{-1}.$

It suffices to show the bound only for the first term in \eqref{eq:F_L_1}.
Similarly to the earlier arguments and with slight abuse of notation, let us denote $\Delta_\delta f(\xi)=f(\xi+\delta)-f(\xi).$ %the difference of forward spacing $\delta>0$.
We have
\begin{align*}
\Delta_\delta P.V. \int_{\TT} \k(\eta) Q_{L}(\xi,\eta) \, d\eta&  =  \int_{|\xi -\eta | < 2\delta }\k(\eta) \Big(Q_{L}(\xi+\delta ,\eta)  - Q_{L}(\xi,\eta) \Big) \, d\eta  \\
&\qquad +  \int_{|\xi -\eta | \geq 2\delta }\k(\eta) \Big(Q_{L}(\xi+\delta ,\eta)  - Q_{L}(\xi,\eta) \Big) \, d\eta \\
& :=F_1(\xi ) +F_2(\xi ).
\end{align*}
The rest of the proof is devoted to proving $|F_i(\xi ) | \leq C_M \delta^{\beta}$.

For the inner region $|\xi -\eta| < 2\delta$,  by Corollary \ref{corollary:building_blocks} and \eqref{eq:building_blocks_a} from Lemma \ref{lemma:building_blocks} for each $\eta \in \TT$, there exists a bounded function $C_{\eta}(\xi)  $ with $|C_{\eta}(\xi) | \leq C_M$ such that for all $\xi \neq \eta$
\begin{align*}
|Q_L(\xi ,\eta )| =  \left| \frac{\xi -\eta}{|\xi - \eta |^2} -  \N (\eta)  \cdot \N(\xi) \frac{  (\gamma(\xi)- \gamma(\eta)) \cdot \T(\xi)}{  |\gamma(\xi)- \gamma(\eta)|^{2   }} \, g(\eta) \right|  =C_{\eta}(\xi) |\xi -\eta|^{-1+ \alpha} .
\end{align*}
We then apply absolute value to the integrand in $F_1$ and use the bound on $C_{\eta}(\xi) $  to obtain that
\begin{align*}
| F_1 (\xi) | & \leq C_M  \int_{|\xi -\eta | < 2\delta }|\k(\eta)| \Big(  |\xi+\delta -\eta|^{-1+\alpha} +   |\xi -\eta|^{-1+\alpha}\Big) \, d\eta .
\end{align*}
By H\"older's inequality with $p'$ being the H\"older dual of $p$, it follows that
 \begin{align*}
| F_1 (\xi) | & \leq C_M  |\k |_{L^p}  \left[ \int_{|\xi -\eta | < 4\delta }      |\xi -\eta|^{(-1+\alpha)p'}  \, d\eta \right]^{\frac{1}{p'}}  = C(M,p)  |\k |_{L^p} \delta^{1 - \frac{2}{p}}.
\end{align*}
Here in the last step we have used $(-1+\alpha)p' = -\frac{1}{p-1} >-1$ when $p>2$.

For the outer region $|\xi -\eta| \geq 2\delta$, $Q_L$ is a.e. differentiable in $\xi$, and we first derive a bound for $\p_\xi Q_L$. Differentiating \eqref{eq:F_L_1_aux1} in $\xi$ gives
\begin{align*}
\p_\xi Q_{L}(\xi, \eta)  & =  \Big(  \frac{1}{| \xi - \eta|^2}   + \N (\eta)\cdot \N(\xi) \frac{  g(\xi)}{  |\gamma(\xi)- \gamma(\eta)|^{2   }}  g(\eta) \\
& \qquad \qquad \qquad  - 2\N (\eta)\cdot \N(\xi) \frac{  \left[ (\gamma(\xi)- \gamma(\eta)) \cdot \T(\xi)  \right]^2 g(\xi) }{  |\gamma(\xi)- \gamma(\eta)|^{4   }}g(\eta)  \Big)  \\
& \qquad+\k(\xi) g(\xi)   \Big(    \N (\eta)\cdot  \T(\xi) \frac{  (\gamma(\xi)- \gamma(\eta)) \cdot \T(\xi)}{  |\gamma(\xi)- \gamma(\eta)|^{2   }}-  \\
&\qquad\qquad\qquad\qquad \qquad\N (\eta)\cdot  \N(\xi) \frac{  (\gamma(\xi)- \gamma(\eta)) \cdot \N(\xi)}{  |\gamma(\xi)- \gamma(\eta)|^{2   }} \Big) g(\eta)  .
\end{align*}
We now use the estimates in Lemma \ref{lemma:building_blocks} to obtain the bounds on each summand in $\p_\xi Q_L$: for instance, \eqref{eq:building_blocks_a}, \eqref{aux22122} and the $C^\alpha$ continuity of $g$ yield
\begin{equation} \label{eq:F_L_derivative_factor_1}
- \N (\eta)\cdot \N(\xi) \frac{  g(\xi)}{  |\gamma(\xi)- \gamma(\eta)|^{2   }}  g(\eta)  = \frac{-1}{|\xi -\eta |^2  } + O(C_M  |\xi -\eta |^{-2+\alpha});
\end{equation}
\eqref{eq:building_blocks_a}, Corollary \ref{corollary:building_blocks}, and the $C^\alpha$ continuity of $g$ yield
\begin{equation}\label{eq:F_L_derivative_factor_2}
2\N (\eta)\cdot \N(\xi) \frac{  \left[ (\gamma(\xi)- \gamma(\eta)) \cdot \T(\xi)  \right]^2 g(\xi) }{  |\gamma(\xi)- \gamma(\eta)|^{4   }}g(\eta)     = \frac{ 2}{|\xi -\eta |^2  } + O(C_M    |\xi -\eta |^{-2+\alpha}) ;
\end{equation}
and similarly
\begin{equation}
\begin{aligned}\label{eq:F_L_derivative_factor_3}
     \left|  \N (\eta)\cdot  \T(\xi) \frac{  (\gamma(\xi)- \gamma(\eta)) \cdot \T(\xi)}{  |\gamma(\xi)- \gamma(\eta)|^{2   }} \right|  & \leq C_M      |\xi -\eta |^{-1  +\alpha} ,\\
  \left| \N (\eta)\cdot  \N(\xi) \frac{  (\gamma(\xi)- \gamma(\eta)) \cdot \N(\xi)}{  |\gamma(\xi)- \gamma(\eta)|^{2   }}  \right|     & \leq C_M    |\xi -\eta |^{-1 +\alpha}.
\end{aligned}
\end{equation}
Then it follows from \eqref{eq:F_L_derivative_factor_1}, \eqref{eq:F_L_derivative_factor_2}, and \eqref{eq:F_L_derivative_factor_3}  that for $\xi,\eta \in \TT $  such that $|\xi -\eta| \geq 2\delta$, we have
\begin{align}\label{eq:F_L_derivative_factor_4}
\p_\xi Q_{L}(\xi, \eta)    =  O(C_M   |\xi -\eta|^{-2+\alpha })  +\k(\xi)  O(C_M    |\xi -\eta |^{-1  + \alpha}).
\end{align}
By the fundamental theorem of calculus and \eqref{eq:F_L_derivative_factor_4}, for any $\delta>0$ and any $\eta \in \TT$ with $|\xi -\eta| \geq 2\delta $ we have
\begin{equation}\label{eq:Q_aux44}
\begin{aligned}
\Big|  Q_L(\xi +\delta,\eta)  -Q_L(\xi, \eta) \Big| & \leq  \int_{\xi}^{\xi+\delta} \Big| \p_\xi Q_{L}( \xi'  , \eta)  \Big| \,d \xi'  \\
& \leq C_M  \int_{\xi}^{\xi+\delta}  \left( |\xi'  -\eta|^{-2+\alpha } + |\k(\xi') |     |\xi'   -\eta |^{-1  + \alpha}\right) \,d \xi'     .
\end{aligned}
\end{equation}

Now we compute the integral on the right-hand side above. Since $|\xi -\eta|\geq 2\delta$, we have $| \xi'   -\eta| \leq 2|\xi -\eta|$ in the above integral, and it follows from \eqref{eq:Q_aux44}, $ |\k|_{L^p(\TT)} \leq C_M $, and H\"older's inequality that
\begin{align}
\Big|  Q_L(\xi +\delta,\eta)  -Q_L(\xi, \eta) \Big|   & \leq C_M \left(\delta  |\xi -\eta |^{-2+\alpha}  +  |\xi -\eta |^{-1+\alpha}\int_{\xi}^{\xi + \delta }  |\k(\xi'   )|\,d \xi'  \right)\nonumber  \\
 & \leq C_M \left( \delta |\xi -\eta|^{-2+\alpha} + \delta^{\alpha}|\xi-\eta|^{-1 +\alpha} \right).\label{eq:Q_aux46} %\nonumber\\
%& \leq C_M \delta |\xi -\eta|^{-2+\alpha} .
\end{align}
Now that we have a good bound on the finite difference of $Q_L$, it follows from \eqref{eq:Q_aux46} that
 \begin{align*}
| F_2 (\xi) |  & \leq C_M  \int_{|\xi -\eta | \geq 2\delta }\k(\eta) \Big|Q_{L}(\xi+\delta ,\eta)  - Q_{L}(\xi,\eta) \Big| \, d\eta \\
& \leq C_M  \int_{|\xi -\eta | \geq 2\delta }\k(\eta) ( \delta |\xi -\eta|^{-2+\alpha} +\delta^\alpha |\xi -\eta|^{-2+\alpha}) \, d\eta \\
& \leq  C_M  \delta  \left[ \int_{|\xi -\eta | \geq 2\delta }  |\xi -\eta|^{( -2+\alpha)p'} \, d\eta \right]^{\frac{1}{p'}}
+C_M \delta^\alpha \left[ \int_{|\xi -\eta | \geq 2\delta }  |\xi -\eta|^{( -1+\alpha)p'} \, d\eta \right]^{\frac{1}{p'}}  \\
& \leq  C_M \delta^{1- \frac{2}{p}},
\end{align*}
where we have used that $|\k|_{L^p(\TT)} \leq C_M $
%and $ (-2+\alpha)p'  = - \frac{p+1}{p-1}<-1$
together with the H\"older inequality.

\end{proof}

Next, we show that the other error term $F_N$ is also H\"older continuous in space.
\begin{proposition}\label{prop:F_holder}
Let $T>0$. For any $  \frac{3}{2}< p\leq \infty $, if $ \g  \in C([0,T]; B^M_p)$ for some $M>1$, then the nonlinear error term $F_N$ defined by \eqref{eq:def_nonlinear_error} satisfies the estimate
$$
| F_N |_{C^\alpha (\TT)} \leq C_M,
$$
where as before $\alpha = 1 - \frac{1}{p}$.
\end{proposition}
\begin{proof}
Let us introduce the shorthand notation
\begin{align*}
F_N (\xi )=   \int_{\TT} R_{N}(\xi,\eta) \, d\eta.
\end{align*}
As before denote by $\Delta_\delta$ the difference of forward spacing $\delta>0$. We have
\begin{align*}
\Delta_\delta F_N (\xi ) &  =  \int_{|\xi -\eta | < 2\delta }  \Big(R_{N}(\xi+\delta ,\eta)  - R_{N} (\xi,\eta) \Big) \, d\eta  \\
&\qquad +  \int_{|\xi -\eta | \geq 2\delta }R_{N}(\xi+\delta ,\eta)  - R_{N}(\xi,\eta) \Big) \, d\eta \\
& := F_3(\xi ) +F_4(\xi ).
\end{align*}
We need to show $|F_i(\xi ) | \leq C_M \delta^{\alpha}$.

For the inner region, $ |\xi -\eta| \leq 2\delta $, we first show the bound
\begin{align}\label{eq:F_N_aux1}
| R_N (\xi,\eta ) | \leq C_M \mathcal{M}\k(\eta)    |\xi -\eta|^{-1+2 \alpha} .
\end{align}
Indeed, as in the proof of Proposition \ref{prop:derivatives_v}, by \eqref{eq:building_blocks_c}, \eqref{eq:building_blocks_M_d}, and \eqref{eq:building_blocks_M_e}
\begin{align*}
| R_N (\xi,\eta ) |  & \leq C_M   \Big|   \T(\eta )\cdot \N(\xi) \big( (  \gamma(\xi)- \gamma(\eta)  ) \cdot \T(\xi) \big)  \frac{    \big( (  \gamma(\xi)- \gamma(\eta)  ) \cdot (\T(\xi) -\T(\eta)) \big)  }{|\gamma(\xi)- \gamma(\eta)|^{4  }} \,  \Big|   \\
	&	\qquad   +C_M    \Big| \T(\eta)  \cdot \N(\xi)\frac{ \big[ (\T(\xi ) - \T(\eta))\cdot \T( \xi)   \big] }{ |\gamma(\xi)- \gamma(\eta)|^{2  } }    \,   \Big|    \\
&\leq  C_M|\xi -\eta|^\alpha  |\xi -\eta| \mathcal{M}\k(\eta) |\xi -\eta|^{2+\alpha } |\xi -\eta|^{-4}  \\
&\qquad + C_M |\xi -\eta|^\alpha \mathcal{M}\k(\eta) |\xi -\eta|^{1+\alpha } |\xi -\eta|^{-2}\\
&\leq C_M  \mathcal{M}\k(\eta) |\xi -\eta|^{-1+ 2 \alpha }.
\end{align*}
Then by \eqref{eq:F_N_aux1} and the H\"older inequality with $p' = \frac{p}{p-1}$
\begin{align*}
|F_3 (\xi)|& \leq C_M \int_{|\xi -\eta | <2\delta} \mathcal{M}\k{\eta} \left[  |\xi +\delta -\eta|^{-1+2\alpha } + |\xi -\eta|^{-1+2\alpha }\right] \, d\eta \\
&\leq C_M   \Big( \int_{|\xi -\eta | <4\delta}   |\xi -\eta|^{(-1+ 2\alpha)p' }  \, d\eta \Big)^\frac{1}{p'}.
\end{align*}
This is integrable since $(1-2\alpha)p' = \frac{p-2}{p-1} \in (-1,1]$ when $ p> \frac{3}{2}$, and we compute the integral to find that $|F_1(\xi)| \leq C_M \delta^{ 2 - \frac{3}{p}}$, which  is more than we need.

Next, we consider the outer region $ |\xi -\eta| \geq  2\delta $. We use the same strategy as in the previous proposition. As $R_N(\xi , \eta)$ is a.e. differentiable in $\xi$ in this region, we first derive a bound for the $\p_\xi R_N$. Differentiating in $\xi$ gives
\begin{equation}
\begin{aligned}\label{eq:F_N_aux2}
\p_\xi R_N(\xi ,\eta) & = \k (\xi ) g(\xi) \big(- I_1  +I_2 +2 I_3 -2 I_4 -2I_5 \big) \\
&\qquad +   g(\xi) \big( 2 J_1 +2J_2 +2J_3 - 8J_4 \big)
\end{aligned}
\end{equation}
where the terms $I_i$ are
 \begin{align*}
I_1: & =   \T( \eta  )  \cdot \T(  \xi )\frac{  \big[ (\T( \xi  ) - \T(  \eta  ))\cdot \T( \xi )   \big]     }{ |\gamma(  \xi )- \gamma( \eta  )|^{2  } } g(\eta )   \,    \\
I_2: &= \T( \eta  )  \cdot \N(  \xi )\frac{  \big[ (\T( \xi  ) - \T(  \eta  ))\cdot \N( \xi )   \big]     }{ |\gamma(  \xi )- \gamma( \eta  )|^{2  } } g(\eta )   \,     \\
I_3 :& =	  \T( \eta  )\cdot \T( \xi  ) \frac{    \big( (  \gamma(  \xi )- \gamma( \eta  )  ) \cdot \T(  \xi ) \big) \big( (  \gamma(  \xi )- \gamma( \eta  )  ) \cdot (\T(  \xi ) -\T( \eta  )) \big)  }{|\gamma(  \xi )- \gamma( \eta   )|^{4  }} g(\eta )\,   \\
I_4 :& =    \T( \eta  )\cdot \N( \xi  ) \frac{    \big( (  \gamma(  \xi  )- \gamma( \eta  )  ) \cdot \N(\xi  ) \big) \big( (  \gamma(  \xi )- \gamma( \eta  )  ) \cdot (\T(  \xi ) -\T( \eta  )) \big)  }{|\gamma(  \xi )- \gamma( \eta  )|^{4  }}  g(\eta )   \\
I_5 :& =    \T( \eta  )\cdot \N( \xi  ) \frac{    \big( (  \gamma(  \xi  )- \gamma( \eta  )  ) \cdot \T(  \xi ) \big) \big( (  \gamma(  \xi )- \gamma( \eta  )  ) \cdot  \N(  \xi )    \big)  }{|\gamma(  \xi )- \gamma(\eta )|^{4  }} g(\eta )
\end{align*}
and  terms $J_i$'s are
\begin{align*}
J_1:& =   \T(\eta  )\cdot \N( \xi ) \frac{  (\T( \xi ) -\T(\eta  ))\cdot \T(\xi)    (  \gamma(  \xi )- \gamma( \eta  )  ) \cdot  \T( \xi )   }{|\gamma(  \xi )- \gamma(\eta  )|^{4  }} g( \eta )\,  d \eta \\
J_2:& =   \T(\eta  )\cdot \N( \xi ) \frac{      (  \gamma(  \xi )- \gamma( \eta  )  ) \cdot (\T( \xi ) -\T(\eta  ))   }{|\gamma(  \xi )- \gamma(\eta  )|^{4  }} g( \eta )\,  d \eta  \\
J_3:&=   \T(\eta  )\cdot \N( \xi   ) \frac{    \big( (  \gamma(  \xi )- \gamma( \eta  )  ) \cdot \T(  \xi ) \big) \big(   \T(  \xi )  \cdot (\T(  \xi ) -\T( \eta  )) \big)  }{|\gamma(  \xi )- \gamma( \eta )|^{4  }}  g(\eta  )   \\
J_4:&=   \T(\eta  )\cdot \N( \xi   ) \frac{    \big( (  \gamma(  \xi )- \gamma( \eta  )  ) \cdot \T(  \xi ) \big)^2 \big( (  \gamma(  \xi )- \gamma( \eta  )  ) \cdot (\T(  \xi ) -\T( \eta  )) \big)  }{|\gamma(  \xi )- \gamma( \eta )|^{6  }}  g(\eta  ).
\end{align*}
Next, we will derive the bound
\begin{align}\label{eq:F_N_aux3}
 \big| \p_\xi R_N (\xi ,\eta) \big| \leq C_M \big(|\k(\xi)|  +\mathcal{M}\k(  \xi) \big) \mathcal{M}\k(  \eta )| \xi -  \eta |^{ -1+\alpha  } .
\end{align}
By the structure of $ \p_\xi R_N$ given by \eqref{eq:F_N_aux2}, it suffices to show for  $I_i$'s the bound
\begin{align}\label{eq:F_N_aux_Ii}
 \big| I_i (\xi ,\eta) \big| \leq C_M   \mathcal{M}\k(  \eta )| \xi -  \eta |^{ -1+\alpha  }
\end{align}
and
for  $J_i$'s the bound
\begin{align}\label{eq:F_N_aux_Ji}
\big| J_i (\xi ,\eta) \big| \leq C_M  \mathcal{M}\k(  \xi)   \mathcal{M}\k(  \eta )| \xi -  \eta |^{ -1+\alpha  }.
\end{align}

\noindent
{\textbf{Estimates of $I_i$'s:}}

For each $I_i$ term, we use one of the maximal bounds in Lemma \ref{lemma:building_blocks}.

By \eqref{eq:building_blocks_M_d},
\begin{align*}
 \big| \T( \eta  )  \cdot \T( \xi )  (\T(\xi ) - \T( \eta ))\cdot \T( \xi )  \big| \leq C_M         \mathcal{M}\k(\eta) | \xi - \eta |^{1+\alpha },
\end{align*}
so $I_1$ satisfies \eqref{eq:F_N_aux_Ii}.

By  \eqref{eq:building_blocks_b} and \eqref{eq:building_blocks_M_a},
\begin{equation*}
\big|  \T( \eta)  \cdot \N( \xi )    \big | \big |    (\T( \xi ) - \T( \eta ))\cdot \N( \xi )   \big |   \leq C_M  \mathcal{M}\k(\eta) |\xi-\eta |^{1+ \alpha}   ,
\end{equation*}
so $I_2$  satisfies \eqref{eq:F_N_aux_Ii}.

By \eqref{eq:building_blocks_M_e},
\begin{align*}
\big|         (  \gamma(  \xi )- \gamma( \eta  )  ) \cdot (\T(  \xi ) -\T( \eta  ))   \big| \leq C_M  \mathcal{M}\k(\eta) |\xi-\eta |^{2+ \alpha}
\end{align*}
so $I_3$   satisfies \eqref{eq:F_N_aux_Ii}.

By \eqref{eq:building_blocks_c}, \eqref{eq:building_blocks_d}, and \eqref{eq:building_blocks_M_e},
\begin{align*}
\big|  \T( \eta  )\cdot \N( \xi  )     \big( (  \gamma(  \xi  )- \gamma( \eta  )  ) &\cdot \N(\xi  ) \big) \big( (  \gamma(  \xi )- \gamma( \eta  )  ) \cdot (\T(  \xi ) -\T( \eta  )) \big)   \big| \\
&\leq C_M \mathcal{M}\k(\eta) |\xi-\eta |^{3+ 3\alpha}
\end{align*}
so $I_4$   satisfies \eqref{eq:F_N_aux_Ii}.

By \eqref{eq:building_blocks_c} and \eqref{eq:building_blocks_M_c},
\begin{align*}
 \big| \T( \eta  )\cdot \N( \xi  )    \big( (  \gamma(  \xi  )- \gamma( \eta  )  ) \cdot \T(  \xi ) \big) \big( (  \gamma(  \xi )- \gamma( \eta  )  ) \cdot  \N(  \xi )    \big)    \big|  \leq   C_M \mathcal{M}\k(\eta) |\xi-\eta |^{3+  \alpha}
\end{align*}
so $I_5$   satisfies \eqref{eq:F_N_aux_Ii}.

\noindent
{\textbf{Estimates of $J_i$:}}

Next, we look at the $J_i$ terms. Each of these requires using two maximal bounds from Lemma \ref{lemma:building_blocks}.

By \eqref{eq:building_blocks_M_a} and \eqref{eq:building_blocks_M_d}, $J_1 $ satisfies \eqref{eq:F_N_aux_Ji}:
\begin{align*}
 |J_1 (  \xi,\eta    )|  & \leq  \Big|  \T(\eta  )\cdot \N( \xi ) \frac{  (\T( \xi ) -\T(\eta  ))\cdot \T(\xi)      }{|\gamma(  \xi )- \gamma(\eta  )|^{3  }}  \Big| \\
& \leq  C_M  \mathcal{M}\k(  \xi)   \mathcal{M}\k(  \eta )| \xi -  \eta |^{ -1+\alpha  }.
\end{align*}

By \eqref{eq:building_blocks_M_a} and \eqref{eq:building_blocks_M_e}, $J_2 $ satisfies \eqref{eq:F_N_aux_Ji}:
\begin{align*}
| J_2 (  \xi,\eta  )|  & \leq C_M  \Big| \T(\eta  )\cdot \N( \xi ) \frac{      (  \gamma(  \xi )- \gamma( \eta  )  ) \cdot (\T( \xi ) -\T(\eta  ))   }{|\gamma(  \xi )- \gamma(\eta  )|^{4  }} g( \eta )\Big| \\
&\leq C_M  \mathcal{M}\k(  \xi)   \mathcal{M}\k(  \eta )| \xi -  \eta |^{ -1+\alpha  }.
\end{align*}

By \eqref{eq:building_blocks_M_a} and \eqref{eq:building_blocks_M_d}, $J_3 $ satisfies \eqref{eq:F_N_aux_Ji}:
\begin{align*}
 | J_3 (  \xi,\eta  )| & \leq C_M \Big|  \T(\eta  )\cdot \N( \xi   ) \frac{      \big(   \T(  \xi )  \cdot (\T(  \xi ) -\T( \eta  )) \big)  }{|\gamma(  \xi )- \gamma( \eta )|^{3  }}  \Big|\\
&\leq C_M  \mathcal{M}\k(  \xi)   \mathcal{M}\k(  \eta )| \xi -  \eta |^{ -1+\alpha  }.
\end{align*}

By \eqref{eq:building_blocks_M_a} and \eqref{eq:building_blocks_M_e}, $J_4 $ satisfies \eqref{eq:F_N_aux_Ji}:
\begin{align*}
 |  J_4 (  \xi,\eta  )|  &\leq C_M \Big|   \T(\eta  )\cdot \N( \xi   ) \frac{    \big( (  \gamma(  \xi )- \gamma( \eta  )  ) \cdot \T(  \xi ) \big)^2 \big( (  \gamma(  \xi )- \gamma( \eta  )  ) \cdot (\T(  \xi ) -\T( \eta  )) \big)  }{|\gamma(  \xi )- \gamma( \eta )|^{6  }} \\
& \leq C_M \Big|  \T(\eta  )\cdot \N( \xi   ) \frac{      \big( (  \gamma(  \xi )- \gamma( \eta  )  ) \cdot (\T(  \xi ) -\T( \eta  )) \big)  }{|\gamma(  \xi )- \gamma( \eta )|^{4  }}  \Big|\\
&\leq C_M  \mathcal{M}\k(  \xi)   \mathcal{M}\k(  \eta )| \xi -  \eta |^{ -1+\alpha  }   .
\end{align*}

Combining the bounds for $I_i$ and $J_i$, we have established \eqref{eq:F_N_aux3}. Since $ \xi +\delta \neq  \eta  $ in the region $|\xi-\eta|\geq 2\delta$, by the fundamental theorem of calculus, \eqref{eq:F_N_aux3}, and the bound $|\xi' -\eta| \leq 2|\xi -\eta|$ for all $\xi' \in [\xi, \xi+\delta]$, we have
\begin{align*}
\left|  R_N (\xi+\delta ,\eta) - R_N (\xi  ,\eta) \right| & \leq \int_{\xi}^{\xi+\delta}\big| \p_\xi R_N  (\xi' ,\eta)  \big|  \, d\xi' \\
& \leq C_M  \mathcal{M}\k(  \eta )  \int_{\xi}^{\xi+\delta} |\xi' -\eta|^{-1+\alpha}   \left(  \k(\xi )+\mathcal{M}\k( \xi  ) \right)\, d \xi'\\
& \leq C_M   |\xi  -\eta|^{-1+\alpha} \mathcal{M}\k(  \eta )   \int_{\xi}^{\xi+\delta} \left(  \k(\xi )+\mathcal{M}\k( \xi  ) \right) \, d\xi' \\
&  \leq C_M  \delta^{\alpha} |\xi  -\eta|^{-1+\alpha}  \mathcal{M}\k(  \eta )  .
\end{align*}
Inserting this above into $F_2$, by H\"older's inequality we have
\begin{align*}
| F_2 (\xi) | \leq C_M \delta^{\alpha} \int_{|\xi -\eta| \geq 2\delta}    |\xi -\eta|^{-1+\alpha}\mathcal{M}\k(  \eta ) \, d\eta  \leq C_M \delta^{\alpha }
\end{align*}
where we have used that $ |\xi -\eta|^{-1+\alpha} \in L^{p'}(\TT) $ when $p> \frac{3}{2}$ .

\end{proof}

We recall that the multiplicative coefficient $ a = -      \p_s    v   \cdot \T$ is also H\"older continuous, with an exponent $\alpha  = 1 - \frac{1}{p}$ by Proposition \ref{prop:derivatives_v}.

\begin{lemma}\label{eq:lemma:Holder_multiplicative_a}
Let $ 1 < p \leq \infty$, and assume $\g  \in C([0,T]; B^M_p)$, then the  coefficient $ a  =-      \p_s    v   \cdot \T$ satisfies the estimate
$$
| a |_{C^\alpha (\TT)} \leq C_M .
$$
\end{lemma}

%%%%%%%%%%%%%%%%%%%%%%%%%%%%%%%%%%%%%%%%%%%%%%%%%%%%%%%%%
\subsection{A commutator estimate}
%%%%%%%%%%%%%%%%%%%%%%%%%%%%%%%%%%%%%%%%%%%%%%%%%%%%%%%%%

The last ingredient we need for the $C^2$ illposedness  is
\begin{lemma}\label{lemma:H_holder_commutator}
Let $1<p <  \infty$ and $\frac{1}{p} < \sigma \leq 1$. Suppose that $f  \in L^p(\TT)$ and $  h \in C^\sigma(\TT)$. Then the commutator satisfies the estimate
$$
| [\mathcal{H}, h] f |_{C^\beta(\TT) } \leq C (p,\sigma) |h|_{C^\sigma (\TT)} |f |_{L^p (\TT) }
$$
for $ \beta=\sigma -\frac{1}{p}$.
\end{lemma}
\begin{proof}
This follows from a standard computation in PDE and we sketch the details here.
Denote by $H(\xi) $ the kernel of the periodic Hilbert transform. We have
\begin{align*}
[\mathcal{H}, h] f(\xi) = \int_{\TT} \left[ h (  \eta)  - h (\xi)\right]f(\eta) H(\xi - \eta) \, d\eta  .
\end{align*}
Since  $H(\xi - \eta)\sim  (\xi -\eta)^{-1} $ for small $|\xi -\eta|>0 $ and $|h (  \eta)  - h (\xi) |\leq  |h|_{C^\sigma }|\xi - \eta|^{\sigma}$, we have $|[\mathcal{H}, h] f|_{L^\infty} \lesssim |f|_{L^p}|h|_{C^\sigma} $.

Next, we show the H\"older continuity with exponent $\beta= \sigma -\frac{1}{p}$. Denote by $\Delta_\delta$ the difference operator of spacing $\delta$, namely $\Delta_\delta f = f(\xi + \delta) - f(\xi ) $. Without loss of generality we assume $\delta>0$ and aim to prove $ |\Delta_\delta[\mathcal{H}, h] f| \lesssim \delta ^{\beta} |h|_{C^\sigma  } |f |_{L^p  } $. We split the integral $\Delta_\delta[\mathcal{H}, h] f$ in the following way:
\begin{align*}
\Delta_\delta[\mathcal{H}, h] f  &= \int_{\TT}  \left[ \Big( h (\eta )  - h (\xi+\delta ) \Big)  H(\xi -\eta+\delta )   -  \Big(h (\eta)  - h (\xi)  \Big)  H(\xi  -\eta)\right] f(\eta)  \, d\eta
\\
&:= I_1 +I_2
\end{align*}
where $I_1$ is the integral on $|\xi -\eta|<   2\delta $ and $I_2$ on $|\xi -\eta| \geq  2\delta $.

\noindent
\textbf{Case 1: $|\xi -\eta|<  2\delta $}

In this case, we first use the $\sigma$-H\"older continuity of $h$ to obtain that
\begin{align*}
I_1 & \leq \left| \int_{|\xi -\eta|<  2\delta }  \left[ \Big( h (\eta )  - h (\xi+\delta ) \Big)  H(\xi -\eta+\delta )   -  \Big(h (\eta)  - h (\xi)  \Big)  H(\xi  -\eta)\right] f(\eta)  \, d\eta \right| \\
&\leq  |h|_{C^\sigma}\int_{|\xi -\eta|<  2\delta} \left(  |\xi -\eta  + \delta |^{\sigma -1}  + |\xi -\eta|^{\sigma -1}\right) |f(\eta )| \, d\eta .
\end{align*}
Since $-1 <\sigma-1 \leq 0$, a direct computation using the H\"older inequality and the bound $|\xi -\eta|<  2\delta$ gives
\begin{align*}
I_1
& \lesssim   |h|_{C^\sigma} |f  |_{L^p} \left[ \int_{|\xi -\eta|<  4\delta}   |\xi -\eta|^{(\frac{p}{p-1})(\sigma -1)}   \, d\eta \right]^{\frac{p-1}{p}}.
\end{align*}
Since $(\frac{p}{p-1})  \sigma -1  >  -1$, we can evaluate the integral and obtain that
\begin{align*}
I_1 & \lesssim  |h|_{C^\sigma} |f  |_{L^p} \delta^{\sigma -1+ \frac{p-1}{p }} = |h|_{C^\sigma} |f  |_{L^p} \delta^{\beta} .
\end{align*}

\noindent
\textbf{Case 2: $|\xi -\eta| \geq  2\delta$}

In this case, we first rearrange the terms as follows.
\begin{align*}
I_2 & \leq  \Big| \int_{|\xi -\eta| \geq  2\delta}  \Big[   \big( h (\eta )- h(\xi ) \big)\big( H(\xi -\eta+\delta )    -  H(\xi -\eta  )  \big)\\
& \qquad +  \big(h (\xi)  - h (\xi+\delta)  \big)  H(\xi -\eta+\delta )  \Big] f(\eta)  \, d\eta \Big|  .
\end{align*}
Recall that for the kernel $H(\xi) = \frac{1}{2\pi} \cot(\frac{\xi}{2}) $, we have for all $|\xi -\eta| \geq  2\delta$ the bounds
\begin{equation}
\begin{aligned}\label{eq:lemma_commutator_H}
\left|H(\xi -\eta+\delta )    -  H(\xi -\eta  ) \right|  &\lesssim \frac{\delta}{|\xi -  \eta|^2}\\
\left|H(\xi -\eta+\delta )   \right|  &\lesssim \frac{1}{|\xi -  \eta| } .
\end{aligned}
\end{equation}
It then follows from \eqref{eq:lemma_commutator_H} and the $\sigma$-H\"older continuity of $h$ that
\begin{align*}
I_2 & \lesssim |h|_{C^\sigma}\left( \delta \int_{|\xi -\eta| \geq  2\delta} |\xi -\eta|^{\sigma-2}  f(\eta) \, d\eta  + \delta^{\sigma }\int_{|\xi -\eta| \geq  2\delta} |\xi -\eta|^{-1}  f(\eta)\, d\eta  \right)  \\
& \lesssim |h|_{C^\sigma} \delta^{\sigma }\int_{|\xi -\eta| \geq  2\delta} |\xi -\eta|^{-1}  f(\eta)\, d\eta   .
\end{align*}
From here we can apply H\"older's inequality (since $p<\infty $) and use the condition $\beta= \sigma - \frac{1}{p} $ to conclude that
\begin{align*}
I_2 & \lesssim |h|_{C^\sigma} |f  |_{L^p}  \delta^{\sigma }  \Big( \int_{|\xi -\eta|\geq 2\delta} |\xi -\eta |^{-{\frac{p}{p-1} }} \Big)^\frac{p-1}{p}= |h|_{C^\sigma} |f  |_{L^p} \delta^\beta .
\end{align*}
\end{proof}

%%%%%%%%%%%%%%%%%%%%%%%%%%%%%%%%%%%%%%%%%%%%%%%%%%%%%%%%%
\subsection{Proof of \texorpdfstring{$C^2$}{C2}  illposedness}
%%%%%%%%%%%%%%%%%%%%%%%%%%%%%%%%%%%%%%%%%%%%%%%%%%%%%%%%%
In what follows, we restrict ourselves to the case $4 \leq  p \leq \infty$ (away from $p = 2$) so that there exists a fixed small $\beta>0$ with the following property.

In the remainder of this section, we denote by $ C(M,T)$ a large constant depending only on $M$ and $T$ but not $p$ that may change from line to line.

\begin{lemma}\label{lemma:uniform_illposed}
Fix $\beta = \frac{1}{2}$. For any $M>1$ and $T>0$, there is a large constant $C=C(M,T)$ depending only on $M$ and $T$ such that for any $\g \in C([0,T];B^M_4 ) $ the uniform H\"older estimates hold
\begin{align}\label{eq:uniform_illposed_1}
\sup_{t\in [0,T]} | F |_{C^\beta (\TT)} + \sup_{t\in [0,T]} | a |_{C^\beta (\TT)} \leq   C(M ,T ),
\end{align}
and
\begin{align}\label{eq:uniform_illposed_2}
\sup_{t\in [0,T]} | [  \mathcal{H}, e^{-\int_0^t a}] \k |_{C^\beta (\TT)} \leq C(M,T).
\end{align}
\end{lemma}
\begin{proof}
Since $1 -\frac{2}{4} =  \frac{1}{2}$, \eqref{eq:uniform_illposed_1} follows directly from Propositions \ref{prop:F_L_holder}, \ref{prop:F_holder} and Lemmas \ref{eq:lemma:Holder_multiplicative_a}, \ref{lemma:H_holder_commutator}.

Lemma \ref{lemma:H_holder_commutator} with $\sigma = 3/4$ and $p=4$ implies that
\begin{align*}
| [  \mathcal{H}, e^{-\int_0^t a}] \k |_{C^\frac{1}{2} (\TT)} & \leq C |\k |_{L^4  (\TT)}   |e^{-\int_0^t a} |_{C^\frac{3}{4}   (\TT)}\\
& \leq C (M, T),
\end{align*}
where  we have used $ a = \p_s v \cdot \T \in C([0,T];C^{  \frac{3}{4}} (\TT ))  $ with $|a|_{C^{\frac{3}{4}}(\TT)} \leq C_M$ due to $\g \in C([0,T];B^M_4 ) $ and also $x \mapsto e^{-x}$ is smooth.

\end{proof}

With Lemma \ref{lemma:uniform_illposed}, we are in position to prove $C^2$ illposedness. Define a new variable using a variation of parameters
\begin{equation}\label{eq:def_K}
K (\xi,t)= e^{-\int_0^t a(\xi,\tau) d\tau} \k(\xi,t),
\end{equation}
then by \eqref{eq:evo_curvature_illposedness} it satisfies the equation
$$
\p_t K =   \pi\mathcal{H}K -  \pi[  \mathcal{H}, e^{-\int_0^t a}]   \k + \widetilde{F}     ,
$$
where $\widetilde{F}: = e^{-\int_0^t a} F $ and $[A,B]:= AB -BA$ denote the commutator.

By Duhamel's principle, we recast the equation for $ K  $ into the integral form
\begin{equation}\label{eq:Duhamel_for_K}
K(t) = e^{t \pi \mathcal{H} } \k_0 + \int_0^t e^{(t-t') \pi\mathcal{H} } \left(-\pi[  \mathcal{H}, e^{-\int_0^{t'} a}]   \k + \widetilde{F} \right) \, dt'   ,
\end{equation}
By Lemma \ref{lemma:uniform_illposed}, for any $M>1$ and $T>0$ we have the following uniform bound for any $p\geq 4 $,
$$
|K (t)|_{L^p(\TT)}  \leq c(M,T) | \k (t)|_{L^p(\TT)} \quad \text{for all $t\in [0,T]$}.
$$
So it suffices to show the inflation for $K$ as $p \to \infty$.

For some suitable initial data $\g_0 \in C^2(\TT)$ with $\k_0 \in C(\TT)$, we will see that the linear dispersion yields $| e^{ t \pi \mathcal{H} } K_0 |_{L^p(\TT)} \to \infty $ as $p\to \infty$. So the task is to verify the integral terms in \eqref{eq:Duhamel_for_K} are well-controlled. %, say in $L^\infty(\TT)$ on $[0,T]$.
\begin{lemma}\label{lemma:nonlinear_for_K}
For $M>1$ and any $T>0$, there exists a constant $C(M,T)>0$ such that for any $\g  \in C([0,T];B^M_p ) $, $4 \leq p \leq \infty$, we have
$$
\sup_{t\in[0,T]}\left| \int_0^t e^{(t-t') \pi\mathcal{H} } \left(-\pi[  \mathcal{H}, e^{-\int_0^{t'} a}]   \k + \widetilde{F} \right) \, d t'  \right|_{C^{1/2}(\TT) } \leq C(M,T) .
$$

\end{lemma}
\begin{proof}
%Let $\beta = \frac{1}{2}$ be fixed as in Lemma \ref{lemma:uniform_illposed}. We replace the $L^\infty$ norm on the left-hand side by the $\beta$-H\"older norm and find that
Observe that
\begin{align*}
  \left| \int_0^t e^{(t -t')  \pi \mathcal{H} }  \left(  -\pi [  \mathcal{H}, e^{-\int_0^{t'} a}]   \k + \widetilde{F} \right) \, d t'  \right|_{  C^{1/2} (\TT) }
 \leq & T \sup_t \left|      [  \mathcal{H}, e^{-\int_0^t a}]   \k  \right|_{  C^{1/2} (\TT) } \\
& + T \sup_t|  \widetilde{F}     |_{  C^{1/2}(\TT) },
\end{align*}
where in the first step we have used that the free evolution $e^{  t   \pi \mathcal{H} }  $ is bounded on  $C^\sigma (\TT)$ for any $ 0 <\sigma <1$ by \eqref{eq:etH}.

Then by \eqref{eq:uniform_illposed_1} and \eqref{eq:uniform_illposed_2} we have
\begin{align*}
\sup_{t\in[0,T]} \left| \int_0^t e^{(t -t')  \pi\mathcal{H} } \left(-\pi [  \mathcal{H}, e^{-\int_0^{t'} a}]   \k + \widetilde{F} \right) \, dt' \right|_{C^{1/2}(\TT) }  &\leq C(M,T) .
\end{align*}

\end{proof}

Then we choose the initial configuration such that the free linear evolution inflates, or more precisely $| e^{t \pi \mathcal{H} } K_0 |_{L^p(\TT)} \to \infty $ as $p\to \infty$.
\begin{lemma}\label{lemma:linear_for_K}
There exists initial data $\g_0 \in  C^2(\TT)$ which is a simple closed curve with curvature $\k_0 \in C(\TT)$ such that the following holds. For any   $t\in ( 0,1)$,
$$
\left| e^{t  \pi\mathcal{H} } \k_0  \right|_{L^p (\TT)} \geq C_0 \max\{t(1-t) \sqrt{p}, 1 \}  \quad \text{for all $p < \infty $  }
$$
where $C_0>0$ is a constant depending only on the initial data $\g_0$.
\end{lemma}
\begin{proof}

It is well-known that the Hilbert transform is not bounded $C ( \TT)  \to L^\infty(\TT)$, and we may pick $ \g_0 \in  C^2$ with $k_0 \in C (\TT)$  such that $| \mathcal{H}  \k_0|_{L^p (\TT)}   \sim \sqrt{p} $  as $ p\to \infty $.
In fact, the rate of divergence in $p$ can be arbitrarily close to $p$ (and exactly $p$ for a function that is just bounded), but we choose a simple explicit $\k_0$ for simplicity.
Let us sketch this argument for the sake of completeness.
For instance, one can take $\ep>0$ small and a curve $\gamma_0$ whose curvature $\k_0 $ satisfies
\begin{equation}\label{eq:k_0_aux1}
 \k_0(\xi )=
\begin{cases}
 (\ln|\xi |^{-1}  )^{-\frac{1}{2}} &\quad \xi \in [0,\ep]\\
- (\ln|\xi |^{-1}  )^{-\frac{1}{2}} &\quad \xi \in [-\ep, 0]
\end{cases}
\end{equation}
Such initial data $\gamma_0 \in C^2(\TT)$ exists since \eqref{eq:k_0_aux1} is a local continuity condition near $\xi=0$. One can take a function $\k_0$ satisfying \eqref{eq:k_0_aux1} and then smoothly extend it on $[\ep, \pi] \cup [-\pi ,-\ep ]$
to obtain a curve $\gamma_0 \in C^2(\TT)$  with $\k_0 \in C(\TT)$ as its curvature.
For any $\xi \in ( 0, \ep/2]$, we have
%\begin{equation}\label{Hcal313}
\[ \pi \H \k_0(\xi) = \int_0^\xi \frac{\k_0(\eta)}{\xi-\eta}\,d\eta + \int_\xi^{2\xi} \frac{\k_0(\eta)}{\xi-\eta}\,d\eta +\int_{-\epsilon}^0 \frac{\k_0(\eta)}{\xi-\eta}\,d\eta+\int_{2\xi}^\epsilon \frac{\k_0(\eta)}{\xi-\eta}\,d\eta +R_H(\xi), \]
%\end{equation}
where $|R_H(\xi)| \leq C|\k_0|_{L^\infty(\TT)} \ln \epsilon^{-1}.$ Since $\k_0$ is increasing on $[0,\epsilon],$ the sum of the first two terms is negative. The third and fourth terms are also negative.
Therefore, for any $\xi \in  ( 0, \ep/2]$
\[ \pi |\H \k_0(\xi) | \geq   \int_{2\xi}^\epsilon \frac{\k_0(\eta)}{\eta-\xi}\,d\eta -C|\k_0|_{L^\infty(\TT)} \ln \epsilon^{-1} \geq \k_0(2\xi) \ln \xi^{-1} -C|\k_0|_{L^\infty(\TT)} \ln \epsilon^{-1}. \]
Thus by our choice \eqref{eq:k_0_aux1}, there exists a scale $0<\epsilon_1 \leq \epsilon$ so that $|\H \k_0(\xi) | \geq c (\ln \xi^{-1})^{1/2}$ for $\xi \in (0,\epsilon_1]$ (one can take $\epsilon_1 = \epsilon^n$ with sufficiently large $n$
depending on $C$ and $|\k_0|_{L^\infty(\TT)}$). Then
\[ |\H \k_0|_{L^p} \gtrsim \left( \int_0^{\epsilon_1} |\ln \xi^{-1}|^{p/2} d \xi \, \right)^{1/p} = \left( \int_{\ln \epsilon_1}^\infty z^{p/2} e^{-z}\,dz \right)^{1/p}  \gtrsim (\Gamma(p/2))^{1/p} \gtrsim \sqrt{p}. \]
%Then the main term in the Hilbert transform gives $ \mathcal{H} \k_0(\xi ) \sim (\ln |\xi |^{-1})^{\frac{1}{2}} $ near $\xi =0$, so $| \mathcal{H} \k_0|_{L^p(\TT)} \sim  (\int_0^1 | \ln |\xi |^{-1}|^{\frac{p}{2}})^{\frac{1}{p}} \sim p$ by the Sterling formula for Gamma functions.
With the initial data $\k_0$ chosen as above, we have
\begin{align*}
%e^{t \pi\mathcal{H}} K_0=
e^{t \pi \mathcal{H}} \k_0 = \sum_n \frac{(t \pi\mathcal{H})^n}{n !}\k_0 =  \cos(\pi t ) \k_0 +\sin(\pi t ) \mathcal{H}\k_0,
\end{align*}
and the conclusion follows immediately since $\sin( \pi t) \sim t $ when $t\in(0,\frac{1}{2}]$ and $\sin( \pi t) \sim 1-t $ when $t\in[\frac{1}{2} , 1 ) $.

\end{proof}

These two lemmas, combined with the global $W^{2,p}$ regularity, yield the main $C^2$ and $C^{1,1}$ illposedness result.

\begin{theorem} \label{thm:final_illposedness}
There exists initial data $ \g_0 \in  C^2(\TT)$ which is a simple closed curve with curvature $\k_0 \in C(\TT)$ such that the following holds. The unique solution $\g $ of \eqref{eq:CDE} with initial data $\g_0$ satisfies $ \g\in C ([0,\infty); X_p) $  for all $p< \infty $  and
$$
| \k (t) |_{L^\infty(\TT)}  =\infty \quad \text{for all $t  \in [0,\infty)\setminus \ZZ$}.
$$

\end{theorem}
\begin{proof}
We only demonstrate how to show $| \k (t) |_{L^\infty(\TT)}  =\infty$ for $ 0<t< 1$ as the rest of the cases $t\in \RR\setminus \ZZ$ are similar.

First of all, since $ \g_0 \in  X_\infty$, by Theorem \ref{thm:global_sobolev}, the unique solution $\g   $ belongs to $ C([0,\infty) ;X_p) $  for any $   p < \infty $.

We proceed with a proof by contradiction. Suppose $ | \k (t^*) |_{L^\infty(\TT)} <\infty$ for some $t^* \in  (  0,1)$, then one must have
$$
\limsup_{p \to \infty }| \k (t^* ) |_{L^p(\TT)} <\infty ,
$$
and hence for the new variable $K$ introduced in \eqref{eq:def_K},
\begin{equation}\label{eq:contradiction}
\limsup_{p \to \infty }| K (t^* ) |_{L^p(\TT)} <\infty  .
\end{equation}
Then by \eqref{eq:Duhamel_for_K}, for all $p\geq 4 $ we have  that
\begin{align*}
| K (t^* ) |_{L^p(\TT)} \geq  \left| e^{t^* \pi \mathcal{H} } \k_0  \right|_{L^p(\TT)} - 2\pi \left| \int_0^{t^* } e^{(t^*-t') \pi \mathcal{H} } \left(  -\pi  [  \mathcal{H}, e^{-\int_0^{t'} a}]   \k + \widetilde{F} \right) \, d {t'}  \right|_{L^\infty(\TT)}.
\end{align*}
It follows from  Lemma \ref{lemma:nonlinear_for_K} that there exists a constant $0< C<\infty$ independent of $p$ such that
\begin{align*}
| K (t^*) |_{L^p(\TT)} \geq  \left| e^{t^* \pi \mathcal{H} } \k_0  \right|_{L^p(\TT)} - C \qquad \text{for all $p\geq 4$}.
\end{align*}
which is a contradiction to \eqref{eq:contradiction} by Lemma \ref{lemma:linear_for_K} since $t^*\in(0,1)$.

\end{proof}

\noindent{\bf{Data Availability Statements}:}
Data sharing is not applicable to this article as no datasets were generated or analyzed during the current study.

\noindent{\bf{Conflict of interest}:}
The authors have no competing interests to declare that are relevant to the content of this article.

%%%%%%%%%%%%%%%%%%%%%%%%%%%%%%%%%%%%%%%%%%%%%%%%%%%%%%%%%
\appendix

\bibliographystyle{alpha}
\bibliography{euler_patch}
%\nocite{*}

\end{document}